\documentclass[11pt]{article}
\usepackage{color}
\usepackage{enumerate}
\usepackage{amsthm,amsmath,amssymb,mathabx}
\usepackage{epsfig}
\usepackage{rotating}
\usepackage{vmargin}
\usepackage{authblk}
\usepackage{graphicx}
\usepackage{caption}
\numberwithin{equation}{section}

\usepackage[english]{babel}
\usepackage[T1]{fontenc} 
\usepackage[utf8]{inputenc}
\usepackage{tikz-cd}
	\usetikzlibrary{decorations.markings}
\usepackage{tikz}
	\usetikzlibrary{tikzmark, calc}
	\usetikzlibrary{matrix}
	
\colorlet{col0}{blue!80}	
\colorlet{col1}{gray!50}
\colorlet{col2}{pink!50}
\colorlet{col3}{yellow}
\colorlet{col4}{red!50}

\makeatletter
\renewcommand*\env@matrix[1][*\c@MaxMatrixCols c]{%
  \hskip -\arraycolsep
  \let\@ifnextchar\new@ifnextchar
  \array{#1}}
\makeatother

\setmarginsrb{2.5cm}{2cm}{2.5cm}{2cm}{0cm}{2cm}{0cm}{1cm}

\theoremstyle{plain}
\newtheorem{theorem}{Theorem}
\newtheorem{lemma}[theorem]{Lemma}
\newtheorem{corollary}[theorem]{Corollary}
\newtheorem{proposition}[theorem]{Proposition}

\theoremstyle{definition}
\newtheorem{definition}[theorem]{Definition}
\newtheorem{example}[theorem]{Example}
\newtheorem{remark}[theorem]{Remark}

\newtheorem{question}[]{Question}

\newcommand{\N}{\mathbb{N}}
\newcommand{\Z}{\mathbb{Z}}
\newcommand{\for}{\mathrm{\ for \ }}
\newcommand{\dir}[3]{#1_{{\mathbf{#2},\mathbf{#3}}}}
 \newcommand{\morph}[8]
 {
 	\tikzstyle{morphi}=[draw=black,line width = 0.5pt]
 	\tikzstyle{mo}=[shape=circle,fill=none,draw=none]
 	\tikzstyle{1}=[shape=rectangle,fill=black,draw=black,,minimum size=3,inner sep=4pt]
 	\tikzstyle{0}=[shape=rectangle,fill=white,draw=black,minimum size=3,inner sep=4pt]
 	\tikzstyle{?}=[shape=rectangle,fill=gray!50,draw=black,minimum size=0.5,inner sep=4pt]
 	\tikzstyle{*}=[shape=rectangle,fill=gray!50,draw=black,minimum size=0.5,inner sep=4pt]
 	\tikzstyle{.}=[shape=rectangle,fill=gray!50,draw=black,minimum size=0.5,inner sep=4pt]
 	
 	\node[shape=rectangle,fill=black,draw=black,,minimum size=3,inner sep=1pt] at (0.4,0){1};
 	\draw[->,line width = 1.2pt,draw=black] (0.6,0) to (1,0);
 	\node[#1] at (1.25,-0.03) {};
 	\node[#2] at (1.25,0.25) {};
 	\node[#4] at (1.54,0.25) {};
 	\node[#3] at (1.54,-0.03) {};
 	\begin{scope}[yshift=-0.8cm]
 		\node[shape=rectangle,fill=white,draw=black,,minimum size=3,inner sep=3.5pt]  at (0.4,0){};
 		\draw[->,line width = 1.2pt,draw=black] (0.6,0) to (1,0);
 		\node[#5] at (1.25,-0.03) {};
		\node[#6] at (1.25,0.25) {};
		\node[#8] at (1.54,0.25) {};
		\node[#7] at (1.54,-0.03) {};
 	\end{scope}
}
\tikzset{negated/.style={
		decoration={markings,
			mark= at position 0.5 with {
				\node[transform shape] (tempnode) {$\backslash$};
			}
		},
		postaction={decorate}
	}
}

\author[1]{\'Emilie Charlier} \author[2]{Svetlana Puzynina}
\author[1]{\'Elise Vandomme}

\affil[1]{Universit\'e de Li\`ege, Belgium}
\affil[2]{Saint Petersburg State University, Russia}

\title{Recurrence along directions in multidimensional words}
\date{\today}

\begin{document}

\thispagestyle{empty}
\maketitle

\begin{abstract}
In this paper we introduce and study new notions of uniform recurrence in multidimensional words. A $d$-dimensional word is called \emph{uniformly recurrent} if for all $(s_1,\ldots,s_d)\in\N^d$ there exists $n\in\N$ such that each block of size $(n,\ldots,n)$ contains the prefix of size $(s_1,\ldots,s_d)$. We are interested in a modification of this property. Namely, we ask that for each rational direction $(q_1,\ldots,q_d)$, each rectangular prefix occurs along this direction in positions $\ell(q_1,\ldots,q_d)$ with bounded gaps. Such words are called \emph{uniformly recurrent along all directions}. We provide several constructions of multidimensional words satisfying this condition, and more generally, a series of four increasingly stronger conditions. In particular, we study the uniform recurrence along directions of multidimentional rotation words and of fixed points of square morphisms.
\end{abstract}

\section{Introduction}

Combinatorics on words in one dimension is a well-studied field of theoretical computer science with its origins in the early 20th century. The study of bidimensional words is less developed, even though many concepts and results are naturally extendable from the unidimensional case (see e.g.\  \cite{Berthe--Vuillon--2000,Cassaigne--1999,Charlier--Karki--Rigo--2010,Durand--Rigo--2013,Labbe--2018,Muchnik--2003,Semenov--1977,Vuillon--1998}). However, some words problems become much more difficult in dimensions higher than one. One of such questions is the connection between local complexity and periodicity. In dimension one, the classical theorem of Morse and Hedlund states that if for some $n$ the number of distinct length-$n$ blocks of an infinite word is less than or equal to $n$, then the word is periodic. In the bidimensional case a similar assertion is known as Nivat's conjecture, and many efforts are made by scientists for checking this hypothesis \cite{Cyr--Kra--2015,Kari--Szabados--2015,Nivat}. In this paper, we introduce and study new notions of multidimensional uniform recurrence.

A first and natural attempt to generalize the notion of (simple) recurrence to the multidimensional setting quickly turns out to be rather unsatisfying. Recall that an infinite word $w\colon\N\to A$ (where $A$ is a finite alphabet) is said to be \emph{recurrent} if each prefix occurs at least twice (and hence every factor occurs infinitely often). A straightforward extension of this definition is to say that a bidimensional infinite word is recurrent whenever each rectangular prefix occurs at least twice (and hence every rectangular factor occurs infinitely often). However, with such a definition of bidimensional recurrence, the following bidimensional infinite word is considered as recurrent, even though any column is not in the unidimensional sense of recurrence. 
\[
 \begin{matrix}
 \vdots & \vdots & \vdots & \vdots & \\
 0 & 0 & 0 & 0 & \cdots \\
 0 & 0 & 0 & 0 & \cdots \\
 0 & 0 & 0 & 0 & \cdots \\
 1 & 1 & 1 & 1 & \cdots
 \end{matrix}
\]

In order to avoid this kind of undesirable phenomenon, a common strengthening is to ask that every prefix occurs uniformly, see for example \cite{Berthe--Vuillon--2000,Durand--Rigo--2013}. In the present work, we investigate several notions of recurrence of multidimensional infinite words $w\colon\N^d\to A$, generalizing the usual notion of uniform recurrence of unidimensional infinite words. 

This paper is organized as follows. In Section~\ref{sec:def}, we define two new notions of uniform recurrence of multidimensional infinite words: the URD words and the SURD words. We also make some first observations in the bidimensional setting. In Section~\ref{sec:origin}, we show that these two new notions of recurrence along directions do not depend on the choice of the origin. This leads us to the definition of the even stronger notion of SSURDO words. In Section~\ref{sec:gcd}, we prove that all multidimensional words obtained by placing some uniformly recurrent word along every rational direction are URD. In Section~\ref{sec:rotation}, we show that all multidimensional rotation words are URD but not SURD. Thus, the notion of SURD words is indeed stronger than that of URD words, justifying the introduced terminology. In Section~\ref{sec:morphism}, we study fixed points of multidimensional square morphisms. In particular, we provide some infinite families of SURD words. We provide a complete characterization of SURD bidimensional infinite words that are fixed points of square morphisms of size $2$. In Section~\ref{sec:construction}, we show how to build uncountably many SURD bidimensional infinite words. In particular, the family of bidimensional infinite words so-obtained contains uncountably many non-morphic SURD elements. We end our study by discussing six open problems in Section~\ref{sec:perspectives}, including potential links with return words and symbolic dynamical aspects.

\section{Definitions and first observations}
\label{sec:def}

Here and throughout the text, $A$ designates an arbitrary finite alphabet and $d$ is a positive integer. For $m,n\in\N$, the notation $[\![m,n]\!]$ designates the interval of integers $\{m,\ldots,n\}$ (which is considered empty for $n<m$). We write $(s_1,\ldots,s_d)\le (t_1,\ldots,t_d)$ (resp.\ $(s_1,\ldots,s_d)< (t_1,\ldots,t_d)$) if $s_i\le t_i$ (resp.\ $s_i< t_i$) for each $i\in[\![1,d]\!]$.

A \emph{$d$-dimensional infinite word} over $A$ is a map $w\colon\N^d\to A$. A \emph{$d$-dimensional finite word} over $A$ is a map $w\colon[\![0,s_1-1]\!]\times \cdots\times [\![0,s_d-1]\!]\to A$, for some $(s_1,\ldots,s_d)\in\N^d$, which is called the \emph{size} of $w$. A finite word $f$ of size $(s_1,\ldots,s_d)$ is a \emph{factor} of a $d$-dimensional infinite word $w$ if there exists  $\mathbf{p}\in\N^d$ such that for each $\mathbf{i}\in[\![0,s_1-1]\!]\times \cdots\times [\![0,s_d-1]\!]$, we have $f(\mathbf{i})=w(\mathbf{p}+\mathbf{i})$. In this case, we say that the factor $f$ occurs at \emph{position} $\mathbf{p}$ in $w$. Similarly, a \emph{factor} of a $d$-dimensional finite word $w$ of size $(t_1,\ldots,t_d)$ is a finite word $f$ of some size $(s_1,\ldots,s_d)\le(t_1,\ldots,t_d)$ for which there exists $\mathbf{p}\in[\![0,t_1-s_1]\!]\times \cdots\times [\![0,t_d-s_d]\!]$ such that for each $\mathbf{i}\in[\![0,s_1-1]\!]\times \cdots\times [\![0,s_d-1]\!]$, we have $f(\mathbf{i})=w(\mathbf{p}+\mathbf{i})$. In both cases (infinite and finite), if $\mathbf{p}=(0,\ldots,0)$ then the factor $f$ is said to be a \emph{prefix} of $w$. In some places, for the sake of clarity, we will allow ourselves to  write $w_\mathbf{i}$ instead of $w(\mathbf{i})$.

\begin{remark}
In general, a factor need not be rectangular, i.e.\ of the form $[\![0,s_1-1]\!]\times \cdots\times [\![0,s_d-1]\!]$, but could be any polytope. Indeed, any occurrence of any given polytope is contained in a larger rectangular factor. If we are interested in bounding the gaps between occurrences of the polytope, then a bound on the gaps of the larger rectangular factor is sufficient. So, without loss of generality we can restrict our attention to rectangular factors only.

Sometimes, multidimensional words are considered over $\mathbb{Z}^d$, i.e.\ $w\colon\mathbb{Z}^d\to A$. Although in our considerations it is more natural to consider one-way infinite words, since for example we will make use of fixed points of morphisms, most of our results and notions can be straightforwardly extended to words over $\mathbb{Z}^d$. For example, general relations between the considered notions hold in $\mathbb{Z}^d$ (Figure~\ref{fig:links_recurrence}), as well as our results for rotation words (Section~\ref{sec:rotation}) and non-morphic SURD examples (Section~\ref{sec:construction}).
\end{remark}

The following notion of uniform recurrence of multidimensional infinite words was studied by many authors, see for example~\cite{Berthe--Vuillon--2000,Durand--Rigo--2013}.

\begin{definition}[UR]
A $d$-dimensional infinite word $w$ is \emph{uniformly recurrent} if for every prefix $p$ of $w$, there exists a positive integer $b$ such that every factor of $w$ of size $(b,\ldots, b)$ contains $p$ as a factor.
\end{definition}

In the previous definition, it is clearly equivalent to ask the same for every factor and not only every prefix. Whenever $d=1$, this definition corresponds to the usual notion of uniform recurrence of infinite words. In the bidimensional setting, the uniform recurrence of the word is not linked to the uniform recurrence of all rows and columns. On the one hand, the fact that rows and columns of a bidimensional word $w\colon\N^2\to A$ are uniformly recurrent (in the unidimensional sense) does not imply that $w$ is UR.

\begin{remark}
\label{rem:convention-2D}
We choose the convention of representing a bidimensional word $w\colon\N^2\to A$ by placing the rows from bottom to top, and the columns from left to right (as for Cartesian coordinates). See Figure~\ref{fig:rep2D}.
\begin{figure}[htb]
\[
    \begin{matrix}
 	 \vdots & \vdots & \vdots & \vdots \\
	 w(0,3) & w(1,3) & w(2,3) & w(3,3) & \cdots \\
	 w(0,2) & w(1,2) & w(2,2) & w(3,2) & \cdots \\
	 w(0,1) & w(1,1) & w(2,1) & w(3,1) & \cdots \\
	 w(0,0) & w(1,0) & w(2,0) & w(3,0) & \cdots 
\end{matrix}
\]
\caption{Convention for the representation of bidimensional words.}
\label{fig:rep2D}
\end{figure}
\end{remark}

\begin{proposition}
\label{prop:rows_columns_UR}
Let $w\colon\N^2\to\{0,1\}$ be the bidimensional word obtained by alternating two kinds of rows: $1F$ and $0F$ where $F=01001010\cdots$ is the Fibonacci word, i.e.\ $F$ is the fixed point of the morphism $0\mapsto 01,\, 1\mapsto 0$ (see Figure~\ref{fig:URrows-nonUR}). The rows and the columns of $w$ are all uniformly recurrent but $w$ is not UR.
\begin{figure}[htb]
\[
\begin{matrix}
 	 \vdots & \vdots & \vdots & \vdots & \vdots & \vdots & \vdots & \vdots & \vdots\\
	  0 & 0 & 1 & 0 & 0 & 1 & 0 & 1 & 0 & \cdots \\
	  1 & 0 & 1 & 0 & 0 & 1 & 0 & 1 & 0 & \cdots \\
	  0 & 0 & 1 & 0 & 0 & 1 & 0 & 1 & 0 & \cdots \\
	  1 & 0 & 1 & 0 & 0 & 1 & 0 & 1 & 0 & \cdots \\
	  0 & 0 & 1 & 0 & 0 & 1 & 0 & 1 & 0 & \cdots \\
	  1 & 0 & 1 & 0 & 0 & 1 & 0 & 1 & 0 & \cdots 
\end{matrix}
\]
\caption{A non UR bidimensional word having uniformly recurrent rows and columns.}
\label{fig:URrows-nonUR}
\end{figure}
\end{proposition}

\begin{proof}
The word $w$ is not UR since the square prefix $\left[\begin{smallmatrix}0&0\\1&0\end{smallmatrix}\right]$ only occurs within the first two columns. The columns of $w$ are uniformly recurrent since they are periodic. It is well known that the words $1F$ and $0F$ are uniformly recurrent, hence the rows of $w$ are uniformly recurrent.
\end{proof}

On the other hand, the fact that a bidimensional infinite word is UR does not imply that each of its rows/columns is uniformly recurrent either. The construction given by the following proposition is a modification of unidimensional Toeplitz words.

\begin{proposition}
\label{prop:UR_rows_columns}
Let $w\colon\N^2\to\{0,1\}$ be the bidimensional word constructed as follows. The $n$-th row (with $n\in\N$) is indexed by $k$ if $n\equiv 2^k\pmod{2^{k+1}}$ and is indexed by $-1$ if $n=0$. Let $u_k=(10^{2^k-1})^\omega$ for $k\ge0$ and $u_{-1}=10^\omega$. Now fill the rows indexed by $k$ with the words $u_k$ (see Figure~\ref{fig:UR-non-recurrent-rows}). The bidimensional word $w$ is UR, but its first row is not recurrent. 
\begin{figure}[htb]
\[
\begin{array}{c|c|cccccccccccccc}
\vdots &  \vdots & \vdots & \vdots & \vdots & \vdots & \vdots & \vdots & \vdots & \vdots & \vdots & \vdots & \vdots& \vdots \\
13 &0 & 1&1&1&1&1&1&1&1&1&1&1&1 & \cdots \\
12 &2 & 1&0&0&0&1&0&0&0&1&0&0&0 & \cdots \\
11 &0 & 1&1&1&1&1&1&1&1&1&1&1&1 & \cdots \\
10 &1 & 1&0&1&0&1&0&1&0&1&0&1&0 & \cdots \\
9 &0 & 1&1&1&1&1&1&1&1&1&1&1&1 & \cdots \\
8 &3 & 1&0&0&0&
0&0&0&0&1&0&0&0 & \cdots \\
7 &0 & 1&1&1&1&1&1&1&1&1&1&1&1 & \cdots \\
6 &1 & 1&0&1&0&1&0&1&0&1&0&1&0 & \cdots \\
5 &0 & 1&1&1&1&1&1&1&1&1&1&1&1 & \cdots \\
4 &2 & 1&0&0&0&1&0&0&0&1&0&0&0 & \cdots \\
3 &0 & 1&1&1&1&1&1&1&1&1&1&1&1 & \cdots \\
2 &1 & 1&0&1&0&1&0&1&0&1&0&1&0 & \cdots \\
1 &0 & 1&1&1&1&1&1&1&1&1&1&1&1 & \cdots \\ 
0 & -1 & 1&0&0&0&0&0&0&0&0&0&0&0 & \cdots \\
\hline
n&k & u_k 
\end{array}
\]
\caption{A UR bidimensional infinite word with a non-recurrent row.}
\label{fig:UR-non-recurrent-rows}
\end{figure}
\end{proposition}

\begin{proof}
Consider first the bidimensional infinite word $w'$ composed of the rows $u_k$ with $k\ge 0$, that is, the word $w$ without its first row. We show that each prefix of $w'$ appears according to a square network. Note that this network argument is also used in the proof of Proposition~\ref{proposition:toeplitz}. In $w'$, let $p$ be a prefix of some size $(s_1,s_2)\in\N^2$ and $N=\max(\lceil \log_2(s_1) \rceil,\lceil \log_2(s_2) \rceil)$. The prefix $p'$ of size $(2^N,2^N)$ appears periodically according to the periods $(2^{N+1},0)$ and $(0,2^{N+1})$. Therefore each factor of $w'$ of size  $(2^{N+1}+2^N-1,2^{N+1}+2^N-1)$ contains $p'$. So it contains $p$ as well. Hence $w'$ is UR.

Now let $p$ denote a prefix of $w$ of some size $(s_1,s_2)\in\N^2$. Let $N=\max(\lceil \log_2(s_1) \rceil,\lceil \log_2(s_2) \rceil)$. Using the previous paragraph, we know that the prefix of $w'$ of size $(2^N,2^N)$ occurs with periods $(2^{N+1},0)$ and $(0,2^{N+1})$. Since the $2^{N+1}$-th row of $w$ is filled with the infinite word $u_{N+1}=(10^{2^{N+1}-1})^\omega$ and that $2^{N+1}>s_1$, the prefix $p$ also appears in position $(0,2^{N+1})$ in $w$, i.e.\ in position $(0,2^{N+1}-1)$ in $w'$. As $w'$ is UR, $p$ occurs within every factor of $w'$ of size $(n,n)$ for some $n\in\N$. As $w$ is composed of $w'$ with an additional row $u_{-1}$, the prefix $p$ of $w$ occurs also within every factor of $w$ of size $(n+1,n+1)$.
\end{proof}

In order to obtain the uniform recurrence of all rows and columns in a bidimensional infinite word, we introduce a different version of uniform recurrence of multidimensional infinite words, which involves directions. Throughout this text, when we talk about a \emph{direction} $\mathbf{q}=(q_1,\ldots,q_d)$, we implicitly assume that $q_1,\ldots,q_d$ are coprime nonnegative integers. For the sake of conciseness, if $\mathbf{s}=(s_1,\ldots,s_d)$, we  write $[\![\mathbf{0},\mathbf{s}-\mathbf{1}]\!]$ in order to designate the $d$-dimensional interval $[\![0,s_1-1]\!]\times \cdots\times [\![0,s_d-1]\!]$. In particular, we set $\mathbf{0}=(0,\ldots,0)$ and $\mathbf{1}=(1,\ldots,1)$.

In what follows, we will use the following notation. Let $w\colon\N^d\to A$ be a $d$-dimensional infinite word, $\mathbf{s}\in\N^d$ and $\mathbf{q}\in\N^d$ be a direction. The \emph{word along the direction $\mathbf{q}$ with respect to the size $\mathbf{s}$ in $w$} is the unidimensional infinite word $\dir{w}{q}{s} \colon \N \to A^{[\![\mathbf{0},\mathbf{s}-\mathbf{1}]\!]}$, where elements of $A^{[\![\mathbf{0},\mathbf{s}-\mathbf{1}]\!]}$ are considered as letters, defined by
\[
	\forall \ell \in\N,\ 
	\forall \mathbf{i}\in [\![\mathbf{0},\mathbf{s}-\mathbf{1}]\!],\ 
	(\dir{w}{q}{s}(\ell))(\mathbf{i})=w(\mathbf{i}+\ell \mathbf{q}).
\]
See Figure~\ref{fig:def-wqs} for an illustration in the bidimensional case.
\begin{figure}[htb]
\centering
\scalebox{0.7}{
\begin{tabular}{cc}
\begin{tikzpicture}[scale=1]
	\tikzstyle{every node}=[shape=rectangle,fill=none,draw=none,minimum size=0cm,inner sep=2pt]
	\tikzstyle{every path}=[draw=black,line width = 0.5pt]
	\foreach \x in {0,1,2,3}
{
	\draw[fill=gray!30] (\x*0.3*7,\x*0.3*2.8) rectangle (\x*0.3*7+1.6,\x*0.3*2.8+1.2);	
	\node at (\x*0.3*7+0.8,\x*0.3*2.8+1.4) {$y_\x$};
}	
	\node[fill=white] at (0.8,1.4){$y_0=p$};
	\draw (0,0) to (1.2*7,0);
	\draw (0,0) to (0,5.8);

	\tikzstyle{every path}=[draw=red,line width = 1pt]
	\draw (0,0.01) to (1.2*7,1.2*2.8);
	\node at (8.6,3.2) {\color{red}$\ell \mathbf{q}$} ;
	\tikzstyle{every path}=[draw=black,line width = 1pt, ->]
	\draw[<->] (0,-0.2)  to  node [below] {$s_1$}  (1.6,-0.2);
	\draw[<->] (1.6+0.2,0)  to  node [right] {$s_2$}  (1.6+0.2,1.2);
	\end{tikzpicture}
	
	& \hspace{1cm}
\begin{tikzpicture}[scale=1]
\tikzstyle{every node}=[shape=rectangle,fill=none,draw=none,minimum size=0cm,inner sep=2pt]
\tikzstyle{every path}=[draw=none,line width = 0.5pt]
\draw[fill=gray!30,draw=none] (0,0) rectangle (1.6,1.2);	

\foreach \x in {1,2,3,4,5}
{
	\draw[fill=gray!30,draw=none] (\x*0.15*7,\x*0.15*5.5) rectangle (\x*0.15*7+1.6,\x*0.15*5.5+1.2);	
	\node at (\x*0.15*7+0.55,\x*0.15*5.5+1.4) {$y_\x$};
}	
\foreach \x in {0,1,2,3,4,5}
{
	\draw[fill=none] (\x*0.15*7,\x*0.15*5.5) rectangle (\x*0.15*7+1.6,\x*0.15*5.5+1.2);	
}
\node at (0.53,1.4){$y_0=p$};
\draw (0,0) to (7,0);
\draw (0,0) to (0,5.8);

\tikzstyle{every path}=[draw=red,line width = 1pt]
\draw (0,0.01) to  (1*7,1*5.5);
\node at (7.3,5.4) {\color{red}$\ell \mathbf{q}$};

\tikzstyle{every path}=[draw=black,line width = 1pt, ->]
\draw[<->] (0,-0.2)  to  node [below] {$s_1$}  (1.6,-0.2);
\draw[<->] (1.6+0.2,0)  to  node [right] {$s_2$}  (1.6+0.2,1.2);
\end{tikzpicture}
\end{tabular}
}
\caption{The word $\dir{w}{q}{s}$ is built from the blocks of size $\mathbf{s}$ occurring at positions $\ell\mathbf{q}$ in $w$. Those blocks in $A^{[\![\mathbf{0},\mathbf{s}-\mathbf{1}]\!]}$ may or may not overlap.}
\label{fig:def-wqs}
\end{figure}
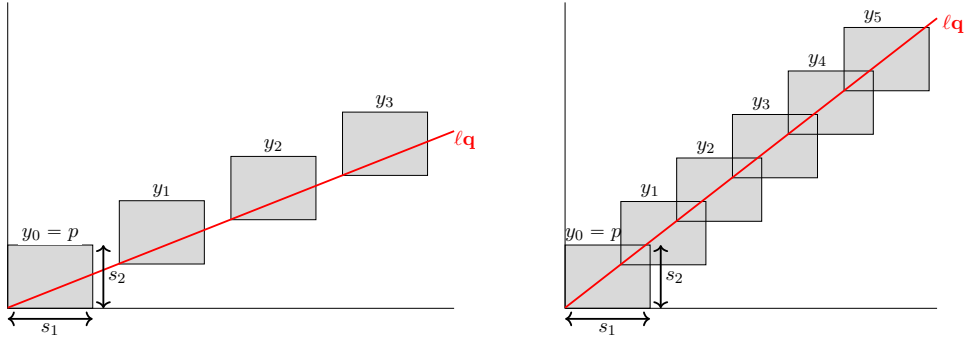

Note that, for any choice of direction $\mathbf{q}$, the first letter $\dir{w}{q}{s}(0)$ of the unidimensional infinite word $\dir{w}{q}{s}$ is the prefix of size $\mathbf{s}$ of the $d$-dimensional infinite word $w$.

\begin{definition}[URD]
\label{def:URD}
A $d$-dimensional infinite word $w\colon\N^d\to A$ is \emph{uniformly recurrent along all directions} (URD for short) if for all $\mathbf{s}\in\N^d$ and all directions $\mathbf{q}\in\N^d$, there exists $b\in\N$ such that each length-$b$ factor of $\dir{w}{q}{s}$ contains the letter $\dir{w}{q}{s}(0)$.
\end{definition}

Alternatively, we can say that the letter $\dir{w}{q}{s}(0)$ occurs infinitely often in $\dir{w}{q}{s}$ with gaps at most $b$. The same reformulation is also valid for further definitions of uniform recurrence.

\begin{proposition}
A $d$-dimensional infinite word $w\colon\N^d\to A$ is URD if and only if for all $\mathbf{s}\in\N^d$ and all directions $\mathbf{q}\in\N^d$, the unidimensional word $\dir{w}{q}{s}$ is uniformly recurrent. 
\end{proposition}

\begin{proof}
The condition is clearly sufficient. Let us show that it is also necessary. Suppose that $w\colon\N^d\to A$ is URD and let $\mathbf{s},\mathbf{q}\in\N^d$ be some fixed size and direction. We show that any prefix of $\dir{w}{q}{s}$ appears infinitely often with bounded gaps in $\dir{w}{q}{s}$. Consider a prefix $p$ of $\dir{w}{q}{s}$ of some length $\ell$. Let $\mathbf{s}'=(\ell-1)\mathbf{q}+\mathbf{s}$. Since $w$ is URD, there exists $b'\in\N$ such that each length-$b'$ factor of $\dir{w}{q}{s'}$ contains the letter $\dir{w}{q}{s'}(0)$. This implies that each length-$(b'+\ell-1)$ factor of $\dir{w}{q}{s}$ contains $p$. 
\end{proof}

We will see that the uniform recurrence along all directions implies that rows and columns are uniformly recurrent (see Proposition~\ref{prop:URD_rows_columns}). However, a URD word is not necessarily UR as shown by the following proposition. In the next section, we will show that UR does not imply URD either (see Corollary~\ref{cor:UR_not_imply_URD}). 

\begin{proposition}\label{prop:URD_isnot_UR}
For any $d\ge 2$, there exists a $d$-dimensional URD word that is not UR.
\end{proposition}

\begin{proof} 
We give a sketch of a construction to avoid cumbersome details.
Let $A$ be a finite alphabet containing at least two letters, say $0$ and $1$, and let $d\ge2$.
Consider the following recursive procedure to construct uncountably many such $d$-dimensional infinite words. See Figure~\ref{fig:URD_isnot_UR}
\begin{figure}[htb]
    \centering
    		\begin{tikzpicture}[every node/.style={scale=0.7}]
		
		\tikzset{square matrix/.style={
				matrix of nodes,
				column sep=-\pgflinewidth, row sep=-\pgflinewidth,
				nodes={
					minimum height=#1,
					anchor=center,
					text width=#1,
					align=center,
					inner sep=0pt
				},
			},
			square matrix/.default=0.5cm
		}
		
		\matrix[square matrix]
		{
|[fill=col1]| 0 & |[fill=col1]| 0 & |[fill=col2]| 0 & |[fill=col3]| 0 & |[fill=col4]| 0 & ? & ? & ? & ? & ? & ? & ? & ? & ? & |[fill=col2]| 0 & |[fill=col2]| 0 & |[fill=col2]| 0 & |[fill=col3]| 0 & ? & ? & ? & ? & ? & ? & ? & ? & ? & ? & |[fill=col1]| 0 & |[fill=col1]| 0 \\
|[fill=col1]| 1 & |[fill=col1]| 0 & |[fill=col2]| 1 & |[fill=col3]| 0 & |[fill=col4]| 1 & ? & ? & ? & ? & ? & ? & ? & ? & ? & |[fill=col2]| 1 & |[fill=col2]| 0 & |[fill=col2]| 1 & |[fill=col3]| 0 & |[fill=col4]| 1 & |[fill=col4]| 0 & |[fill=col4]| 1 & ? & ? & ? & |[fill=col4]| 1 & |[fill=col4]| 0 & |[fill=col4]| 1 & |[fill=col4]| 0 & |[fill=col1]| 1 & |[fill=col1]| 0 \\
|[fill=col1]| 0 & |[fill=col1]| 0 & |[fill=col2]| 0 & |[fill=col3]| 0 & ? & ? & ? & ? & |[fill=col3]| 0 & |[fill=col3]| 0 & |[fill=col3]| 0 & |[fill=col3]| 0 & |[fill=col3]| 0 & |[fill=col3]| 0 & |[fill=col3]| 0 & |[fill=col3]| 0 & |[fill=col3]| 0 & |[fill=col3]| 0 & |[fill=col3]| 0 & |[fill=col3]| 0 & |[fill=col4]| 0 & ? & ? & ? & |[fill=col3]| 0 & |[fill=col3]| 0 & |[fill=col1]| 0 & |[fill=col1]| 0 & |[fill=col4]| 0 & ? \\
|[fill=col1]| 1 & |[fill=col1]| 0 & |[fill=col2]| 1 & |[fill=col3]| 0 & ? & ? & ? & ? & |[fill=col3]| 1 & |[fill=col3]| 0 & |[fill=col3]| 1 & |[fill=col3]| 0 & |[fill=col2]| 1 & |[fill=col2]| 0 & |[fill=col2]| 1 & |[fill=col3]| 0 & |[fill=col3]| 1 & |[fill=col3]| 0 & |[fill=col3]| 1 & |[fill=col3]| 0 & |[fill=col4]| 0 & ? & ? & ? & |[fill=col2]| 1 & |[fill=col2]| 0 & |[fill=col1]| 1 & |[fill=col1]| 0 & |[fill=col4]| 0 & ? \\
|[fill=col1]| 0 & |[fill=col1]| 0 & |[fill=col2]| 0 & |[fill=col3]| 0 & ? & ? & ? & ? & |[fill=col3]| 0 & |[fill=col3]| 0 & |[fill=col3]| 0 & |[fill=col3]| 0 & |[fill=col2]| 0 & |[fill=col2]| 0 & |[fill=col2]| 0 & |[fill=col3]| 0 & |[fill=col3]| 0 & |[fill=col3]| 0 & |[fill=col3]| 0 & |[fill=col3]| 0 & |[fill=col4]| 0 & ? & ? & ? & |[fill=col1]| 0 & |[fill=col1]| 0 & |[fill=col2]| 0 & |[fill=col3]| 0 & |[fill=col4]| 0 & ? \\
|[fill=col1]| 1 & |[fill=col1]| 0 & |[fill=col2]| 1 & |[fill=col3]| 0 & ? & ? & ? & ? & |[fill=col3]| 1 & |[fill=col3]| 0 & |[fill=col3]| 1 & |[fill=col3]| 0 & |[fill=col2]| 1 & |[fill=col2]| 0 & |[fill=col2]| 1 & |[fill=col3]| 0 & |[fill=col3]| 1 & |[fill=col3]| 0 & |[fill=col3]| 1 & |[fill=col3]| 0 & |[fill=col4]| 1 & ? & ? & ? & |[fill=col1]| 1 & |[fill=col1]| 0 & |[fill=col2]| 1 & |[fill=col3]| 0 & |[fill=col4]| 1 & ? \\
|[fill=col1]| 0 & |[fill=col1]| 0 & |[fill=col2]| 0 & |[fill=col3]| 0 & ? & ? & ? & ? & ? & ? & |[fill=col3]| 0 & |[fill=col3]| 0 & |[fill=col3]| 0 & |[fill=col3]| 0 & |[fill=col4]| 0 & ? & ? & ? & ? & ? & |[fill=col3]| 0 & |[fill=col3]| 0 & |[fill=col1]| 0 & |[fill=col1]| 0 & ? & ? & ? & ? & ? & ? \\
|[fill=col1]| 1 & |[fill=col1]| 0 & |[fill=col2]| 1 & |[fill=col3]| 0 & ? & ? & ? & ? & ? & ? & |[fill=col2]| 1 & |[fill=col2]| 0 & |[fill=col2]| 1 & |[fill=col3]| 0 & |[fill=col4]| 0 & ? & ? & ? & ? & ? & |[fill=col2]| 1 & |[fill=col2]| 0 & |[fill=col1]| 1 & |[fill=col1]| 0 & |[fill=col4]| 1 & |[fill=col4]| 0 & |[fill=col4]| 1 & |[fill=col4]| 0 & |[fill=col4]| 1 & ? \\
|[fill=col1]| 0 & |[fill=col1]| 0 & |[fill=col2]| 0 & |[fill=col3]| 0 & ? & ? & |[fill=col3]| 0 & |[fill=col3]| 0 & |[fill=col3]| 0 & |[fill=col3]| 0 & |[fill=col2]| 0 & |[fill=col2]| 0 & |[fill=col2]| 0 & |[fill=col3]| 0 & |[fill=col3]| 0 & |[fill=col3]| 0 & ? & ? & ? & ? & |[fill=col1]| 0 & |[fill=col1]| 0 & |[fill=col2]| 0 & |[fill=col3]| 0 & |[fill=col4]| 0 & |[fill=col4]| 0 & |[fill=col4]| 0 & |[fill=col4]| 0 & |[fill=col4]| 0 & ? \\
|[fill=col1]| 1 & |[fill=col1]| 0 & |[fill=col2]| 1 & |[fill=col3]| 0 & ? & ? & |[fill=col3]| 1 & |[fill=col3]| 0 & |[fill=col3]| 1 & |[fill=col3]| 0 & |[fill=col2]| 1 & |[fill=col2]| 0 & |[fill=col2]| 1 & |[fill=col3]| 0 & |[fill=col3]| 1 & |[fill=col3]| 0 & ? & ? & ? & ? & |[fill=col1]| 1 & |[fill=col1]| 0 & |[fill=col2]| 1 & |[fill=col3]| 0 & |[fill=col4]| 1 & |[fill=col4]| 0 & |[fill=col4]| 1 & |[fill=col4]| 0 & |[fill=col4]| 0 & ? \\
|[fill=col1]| 0 & |[fill=col1]| 0 & |[fill=col2]| 0 & |[fill=col3]| 0 & ? & ? & |[fill=col3]| 0 & |[fill=col3]| 0 & |[fill=col3]| 0 & |[fill=col3]| 0 & |[fill=col3]| 0 & |[fill=col3]| 0 & |[fill=col3]| 0 & |[fill=col3]| 0 & |[fill=col3]| 0 & |[fill=col3]| 0 & |[fill=col3]| 0 & |[fill=col3]| 0 & |[fill=col1]| 0 & |[fill=col1]| 0 & ? & ? & ? & ? & |[fill=col3]| 0 & |[fill=col3]| 0 & |[fill=col3]| 0 & |[fill=col3]| 0 & |[fill=col4]| 0 & ? \\
|[fill=col1]| 1 & |[fill=col1]| 0 & |[fill=col2]| 1 & |[fill=col3]| 0 & |[fill=col4]| 1 & ? & |[fill=col3]| 1 & |[fill=col3]| 0 & |[fill=col2]| 1 & |[fill=col2]| 0 & |[fill=col2]| 1 & |[fill=col3]| 0 & |[fill=col3]| 1 & |[fill=col3]| 0 & |[fill=col3]| 1 & |[fill=col3]| 0 & |[fill=col2]| 1 & |[fill=col2]| 0 & |[fill=col1]| 1 & |[fill=col1]| 0 & ? & ? & ? & ? & |[fill=col3]| 1 & |[fill=col3]| 0 & |[fill=col3]| 1 & |[fill=col3]| 0 & |[fill=col4]| 1 & ? \\
|[fill=col1]| 0 & |[fill=col1]| 0 & |[fill=col2]| 0 & |[fill=col3]| 0 & |[fill=col4]| 0 & ? & ? & ? & |[fill=col2]| 0 & |[fill=col2]| 0 & |[fill=col2]| 0 & |[fill=col3]| 0 & ? & ? & ? & ? & |[fill=col1]| 0 & |[fill=col1]| 0 & |[fill=col2]| 0 & |[fill=col3]| 0 & ? & ? & ? & ? & |[fill=col3]| 0 & |[fill=col3]| 0 & |[fill=col3]| 0 & |[fill=col3]| 0 & ? & ? \\
|[fill=col1]| 1 & |[fill=col1]| 0 & |[fill=col2]| 1 & |[fill=col3]| 0 & |[fill=col4]| 0 & ? & ? & ? & |[fill=col2]| 1 & |[fill=col2]| 0 & |[fill=col2]| 1 & |[fill=col3]| 0 & |[fill=col4]| 1 & |[fill=col4]| 0 & |[fill=col4]| 1 & |[fill=col4]| 0 & |[fill=col1]| 1 & |[fill=col1]| 0 & |[fill=col2]| 1 & |[fill=col3]| 0 & |[fill=col4]| 1 & |[fill=col4]| 0 & |[fill=col4]| 1 & ? & |[fill=col3]| 1 & |[fill=col3]| 0 & |[fill=col3]| 1 & |[fill=col3]| 0 & |[fill=col4]| 1 & ? \\
|[fill=col1]| 0 & |[fill=col1]| 0 & |[fill=col2]| 0 & |[fill=col3]| 0 & |[fill=col3]| 0 & |[fill=col3]| 0 & |[fill=col3]| 0 & |[fill=col3]| 0 & |[fill=col3]| 0 & |[fill=col3]| 0 & |[fill=col3]| 0 & |[fill=col3]| 0 & |[fill=col3]| 0 & |[fill=col3]| 0 & |[fill=col1]| 0 & |[fill=col1]| 0 & |[fill=col4]| 0 & ? & |[fill=col3]| 0 & |[fill=col3]| 0 & |[fill=col3]| 0 & |[fill=col3]| 0 & |[fill=col4]| 0 & ? & |[fill=col3]| 0 & |[fill=col3]| 0 & |[fill=col3]| 0 & |[fill=col3]| 0 & |[fill=col4]| 0 & ? \\
|[fill=col1]| 1 & |[fill=col1]| 0 & |[fill=col2]| 1 & |[fill=col3]| 0 & |[fill=col3]| 1 & |[fill=col3]| 0 & |[fill=col2]| 1 & |[fill=col2]| 0 & |[fill=col2]| 1 & |[fill=col3]| 0 & |[fill=col3]| 1 & |[fill=col3]| 0 & |[fill=col2]| 1 & |[fill=col2]| 0 & |[fill=col1]| 1 & |[fill=col1]| 0 & |[fill=col4]| 0 & ? & |[fill=col3]| 1 & |[fill=col3]| 0 & |[fill=col3]| 1 & |[fill=col3]| 0 & |[fill=col4]| 0 & ? & |[fill=col2]| 1 & |[fill=col2]| 0 & |[fill=col2]| 1 & |[fill=col3]| 0 & |[fill=col4]| 0 & ? \\
|[fill=col1]| 0 & |[fill=col1]| 0 & |[fill=col2]| 0 & |[fill=col3]| 0 & |[fill=col3]| 0 & |[fill=col3]| 0 & |[fill=col2]| 0 & |[fill=col2]| 0 & |[fill=col2]| 0 & |[fill=col3]| 0 & |[fill=col3]| 0 & |[fill=col3]| 0 & |[fill=col1]| 0 & |[fill=col1]| 0 & |[fill=col2]| 0 & |[fill=col3]| 0 & |[fill=col4]| 0 & ? & |[fill=col3]| 0 & |[fill=col3]| 0 & |[fill=col3]| 0 & |[fill=col3]| 0 & |[fill=col4]| 0 & ? & |[fill=col2]| 0 & |[fill=col2]| 0 & |[fill=col2]| 0 & |[fill=col3]| 0 & |[fill=col4]| 0 & ? \\
|[fill=col1]| 1 & |[fill=col1]| 0 & |[fill=col2]| 1 & |[fill=col3]| 0 & |[fill=col3]| 1 & |[fill=col3]| 0 & |[fill=col2]| 1 & |[fill=col2]| 0 & |[fill=col2]| 1 & |[fill=col3]| 0 & |[fill=col3]| 1 & |[fill=col3]| 0 & |[fill=col1]| 1 & |[fill=col1]| 0 & |[fill=col2]| 1 & |[fill=col3]| 0 & |[fill=col4]| 1 & ? & |[fill=col3]| 1 & |[fill=col3]| 0 & |[fill=col3]| 1 & |[fill=col3]| 0 & |[fill=col4]| 1 & ? & |[fill=col2]| 1 & |[fill=col2]| 0 & |[fill=col2]| 1 & |[fill=col3]| 0 & |[fill=col4]| 1 & ? \\
|[fill=col1]| 0 & |[fill=col1]| 0 & |[fill=col2]| 0 & |[fill=col3]| 0 & |[fill=col3]| 0 & |[fill=col3]| 0 & |[fill=col3]| 0 & |[fill=col3]| 0 & |[fill=col3]| 0 & |[fill=col3]| 0 & |[fill=col1]| 0 & |[fill=col1]| 0 & |[fill=col3]| 0 & |[fill=col3]| 0 & |[fill=col3]| 0 & |[fill=col3]| 0 & |[fill=col3]| 0 & |[fill=col3]| 0 & |[fill=col3]| 0 & |[fill=col3]| 0 & ? & ? & ? & ? & ? & ? & ? & ? & ? & ? \\
|[fill=col1]| 1 & |[fill=col1]| 0 & |[fill=col2]| 1 & |[fill=col3]| 0 & |[fill=col2]| 1 & |[fill=col2]| 0 & |[fill=col2]| 1 & |[fill=col3]| 0 & |[fill=col2]| 1 & |[fill=col2]| 0 & |[fill=col1]| 1 & |[fill=col1]| 0 & |[fill=col3]| 1 & |[fill=col3]| 0 & |[fill=col3]| 1 & |[fill=col3]| 0 & |[fill=col2]| 1 & |[fill=col2]| 0 & |[fill=col2]| 1 & |[fill=col3]| 0 & |[fill=col4]| 1 & |[fill=col4]| 0 & |[fill=col4]| 1 & ? & |[fill=col4]| 1 & |[fill=col4]| 0 & |[fill=col4]| 1 & |[fill=col4]| 0 & |[fill=col4]| 1 & ? \\
|[fill=col1]| 0 & |[fill=col1]| 0 & |[fill=col2]| 0 & |[fill=col3]| 0 & |[fill=col2]| 0 & |[fill=col2]| 0 & |[fill=col2]| 0 & |[fill=col3]| 0 & |[fill=col1]| 0 & |[fill=col1]| 0 & |[fill=col2]| 0 & |[fill=col3]| 0 & |[fill=col3]| 0 & |[fill=col3]| 0 & |[fill=col3]| 0 & |[fill=col3]| 0 & |[fill=col2]| 0 & |[fill=col2]| 0 & |[fill=col2]| 0 & |[fill=col3]| 0 & |[fill=col3]| 0 & |[fill=col3]| 0 & |[fill=col4]| 0 & ? & |[fill=col4]| 0 & |[fill=col4]| 0 & |[fill=col4]| 0 & |[fill=col4]| 0 & |[fill=col4]| 0 & ? \\
|[fill=col1]| 1 & |[fill=col1]| 0 & |[fill=col2]| 1 & |[fill=col3]| 0 & |[fill=col2]| 1 & |[fill=col2]| 0 & |[fill=col2]| 1 & |[fill=col3]| 0 & |[fill=col1]| 1 & |[fill=col1]| 0 & |[fill=col2]| 1 & |[fill=col3]| 0 & |[fill=col3]| 1 & |[fill=col3]| 0 & |[fill=col3]| 1 & |[fill=col3]| 0 & |[fill=col2]| 1 & |[fill=col2]| 0 & |[fill=col2]| 1 & |[fill=col3]| 0 & |[fill=col3]| 1 & |[fill=col3]| 0 & |[fill=col4]| 0 & ? & |[fill=col4]| 1 & |[fill=col4]| 0 & |[fill=col4]| 1 & |[fill=col4]| 0 & |[fill=col4]| 0 & |[fill=col4]| 0 \\
|[fill=col1]| 0 & |[fill=col1]| 0 & |[fill=col2]| 0 & |[fill=col3]| 0 & |[fill=col3]| 0 & |[fill=col3]| 0 & |[fill=col1]| 0 & |[fill=col1]| 0 & |[fill=col3]| 0 & |[fill=col3]| 0 & |[fill=col3]| 0 & |[fill=col3]| 0 & ? & ? & |[fill=col3]| 0 & |[fill=col3]| 0 & |[fill=col3]| 0 & |[fill=col3]| 0 & |[fill=col3]| 0 & |[fill=col3]| 0 & |[fill=col3]| 0 & |[fill=col3]| 0 & |[fill=col4]| 0 & ? & |[fill=col4]| 0 & |[fill=col4]| 0 & |[fill=col4]| 0 & |[fill=col4]| 0 & |[fill=col4]| 0 & |[fill=col4]| 0 \\
|[fill=col1]| 1 & |[fill=col1]| 0 & |[fill=col2]| 1 & |[fill=col2]| 0 & |[fill=col2]| 1 & |[fill=col2]| 0 & |[fill=col1]| 1 & |[fill=col1]| 0 & |[fill=col2]| 1 & |[fill=col2]| 0 & |[fill=col2]| 1 & |[fill=col3]| 0 & ? & ? & |[fill=col3]| 0 & |[fill=col3]| 0 & |[fill=col3]| 0 & |[fill=col3]| 0 & |[fill=col3]| 1 & |[fill=col3]| 0 & |[fill=col3]| 1 & |[fill=col3]| 0 & |[fill=col4]| 1 & ? & |[fill=col4]| 1 & |[fill=col4]| 0 & |[fill=col4]| 1 & |[fill=col4]| 0 & |[fill=col4]| 1 & |[fill=col4]| 0 \\
|[fill=col1]| 0 & |[fill=col1]| 0 & |[fill=col2]| 0 & |[fill=col2]| 0 & |[fill=col1]| 0 & |[fill=col1]| 0 & |[fill=col2]| 0 & |[fill=col3]| 0 & |[fill=col2]| 0 & |[fill=col2]| 0 & |[fill=col2]| 0 & |[fill=col2]| 0 & |[fill=col2]| 0 & |[fill=col2]| 0 & |[fill=col3]| 0 & |[fill=col3]| 0 & |[fill=col3]| 0 & |[fill=col3]| 0 & ? & ? & ? & ? & ? & ? & ? & ? & ? & ? & ? & |[fill=col4]| 0 \\
|[fill=col1]| 1 & |[fill=col1]| 0 & |[fill=col2]| 1 & |[fill=col2]| 0 & |[fill=col1]| 1 & |[fill=col1]| 0 & |[fill=col2]| 1 & |[fill=col3]| 0 & |[fill=col2]| 1 & |[fill=col2]| 0 & |[fill=col2]| 1 & |[fill=col2]| 0 & |[fill=col2]| 0 & |[fill=col2]| 0 & |[fill=col3]| 0 & |[fill=col3]| 0 & |[fill=col3]| 0 & |[fill=col3]| 0 & |[fill=col4]| 1 & |[fill=col4]| 0 & |[fill=col4]| 1 & |[fill=col4]| 0 & |[fill=col4]| 1 & ? & ? & ? & ? & ? & ? & |[fill=col4]| 0 \\
|[fill=col1]| 0 & |[fill=col1]| 0 & |[fill=col1]| 0 & |[fill=col1]| 0 & |[fill=col1]| 0 & |[fill=col1]| 0 & |[fill=col3]| 0 & |[fill=col3]| 0 & |[fill=col3]| 0 & |[fill=col3]| 0 & ? & |[fill=col2]| 0 & |[fill=col2]| 0 & |[fill=col2]| 0 & |[fill=col3]| 0 & |[fill=col3]| 0 & ? & ? & |[fill=col3]| 0 & |[fill=col3]| 0 & |[fill=col3]| 0 & |[fill=col3]| 0 & |[fill=col4]| 0 & ? & |[fill=col3]| 0 & |[fill=col3]| 0 & |[fill=col3]| 0 & |[fill=col3]| 0 & ? & ? \\
|[fill=col1]| 1 & |[fill=col1]| 0 & |[fill=col1]| 1 & |[fill=col1]| 0 & |[fill=col1]| 0 & |[fill=col1]| 0 & |[fill=col2]| 1 & |[fill=col2]| 0 & |[fill=col2]| 1 & |[fill=col3]| 0 & ? & ? & |[fill=col2]| 1 & |[fill=col2]| 0 & |[fill=col2]| 1 & |[fill=col3]| 0 & ? & ? & |[fill=col2]| 1 & |[fill=col2]| 0 & |[fill=col2]| 1 & |[fill=col3]| 0 & |[fill=col4]| 0 & ? & |[fill=col2]| 1 & |[fill=col2]| 0 & |[fill=col2]| 1 & |[fill=col3]| 0 & ? & ? \\
|[fill=col1]| 0 & |[fill=col1]| 0 & |[fill=col1]| 0 & |[fill=col1]| 0 & |[fill=col1]| 0 & |[fill=col1]| 0 & |[fill=col1]| 0 & |[fill=col1]| 0 & |[fill=col1]| 0 & |[fill=col1]| 0 & |[fill=col1]| 0 & |[fill=col1]| 0 & |[fill=col1]| 0 & |[fill=col1]| 0 & |[fill=col1]| 0 & |[fill=col1]| 0 & |[fill=col1]| 0 & |[fill=col1]| 0 & |[fill=col1]| 0 & |[fill=col1]| 0 & |[fill=col1]| 0 & |[fill=col1]| 0 & |[fill=col1]| 0 & |[fill=col1]| 0 & |[fill=col1]| 0 & |[fill=col1]| 0 & |[fill=col1]| 0 & |[fill=col1]| 0 & |[fill=col1]| 0 & |[fill=col1]| 0 \\
|[fill=col0]| 1 & |[fill=col1]| 0 & |[fill=col1]| 1 & |[fill=col1]| 0 & |[fill=col1]| 1 & |[fill=col1]| 0 & |[fill=col1]| 1 & |[fill=col1]| 0 & |[fill=col1]| 1 & |[fill=col1]| 0 & |[fill=col1]| 1 & |[fill=col1]| 0 & |[fill=col1]| 1 & |[fill=col1]| 0 & |[fill=col1]| 1 & |[fill=col1]| 0 & |[fill=col1]| 1 & |[fill=col1]| 0 & |[fill=col1]| 1 & |[fill=col1]| 0 & |[fill=col1]| 1 & |[fill=col1]| 0 & |[fill=col1]| 1 & |[fill=col1]| 0 & |[fill=col1]| 1 & |[fill=col1]| 0 & |[fill=col1]| 1 & |[fill=col1]| 0 & |[fill=col1]| 1 & |[fill=col1]| 0 \\
		};
\end{tikzpicture}
    \caption{The first 5 steps of the construction of a URD word that is not UR, according to the procedure described in Proposition~\ref{prop:URD_isnot_UR}. The letters filled at steps $1,\ldots,5$ are respectively drawn on cells colored in blue, gray, pink, yellow and red.
    }
    \label{fig:URD_isnot_UR}
\end{figure}
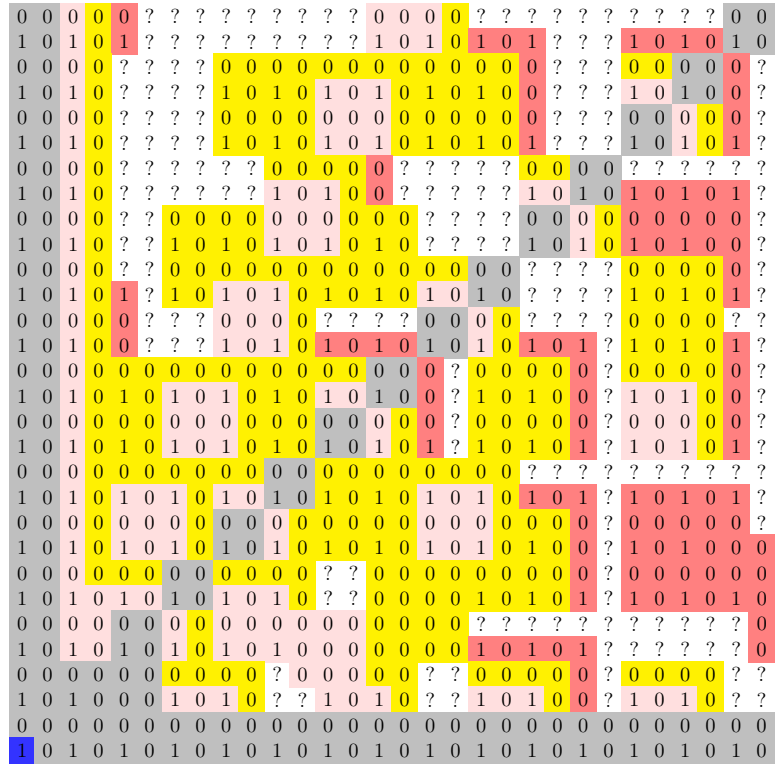
for an illustration of a bidimensional binary such word. On the first step, fill the position $\mathbf{0}$ with the letter $1$. On each step $n\ge 2$, consider the prefix $p_n$ of size $\mathbf{n}=(n,\ldots,n)$ which is partially filled. Choose arbitrary letters of $A$ to complete it (in Figure~\ref{fig:URD_isnot_UR}, we chose to complete with $0$'s at each step). For each direction $\mathbf{q}<\mathbf{n}$, choose a constant $b_\mathbf{q}$ and copy $p_n$ in all positions $\ell b_\mathbf{q} \mathbf{q}$ with $\ell \in \N$. Note that the word $\dir{w}{q}{n}$ may be already partially filled, but there always exists a constant $b_\mathbf{q}$ (potentially big) that allows us to perform this procedure. In one of the remaining factors of size $\mathbf{n}$ that do not contain any letter yet, write $n^d$ times the letter $0$ (at each step of Figure~\ref{fig:URD_isnot_UR}, we chose such a square below the diagonal and the closest possible of the origin). All so-constructed words are URD but are not UR since they contain arbitrarily large hypercubes of $0$'s.
\end{proof}

A natural strengthening of the definition of URD words is to ask that the bound between consecutive occurrences of a prefix only depends on the size of the prefix and not on the chosen direction.

\begin{definition}[SURD]\label{def:SURD}
A $d$-dimensional infinite word $w\colon\N^d\to A$ is \emph{strongly uniformly recurrent along all directions} (SURD for short) if for each $\mathbf{s}\in\N^d$, there exists $b\in\N$ such that, for each direction $\mathbf{q} \in\N^d$, each length-$b$ factor of $\dir{w}{q}{s}$ contains the letter $\dir{w}{q}{s}(0)$.
\end{definition}

In Figure~\ref{fig:links_recurrence}, we summarize the relations between the different notions of recurrence we consider.
\begin{figure}[htb]
\centering
\begin{tikzcd}[arrows=Rightarrow, column sep=1.7cm, row sep=0.8cm, every arrow/.append style={shift left=0.8ex}]
	\text{\small SSURDO}
	\arrow[shorten <= 3pt,shorten >= 3pt]{dd}{\text{Def.}~\ref{def:SSURDO}}
	\arrow{r}{\text{Prop.}~\ref{prop:SSURDO_imply_UR}} 
	& 
    \text{\small UR}
	\arrow[negated,shorten <= 3pt,shorten >= 3pt]{dd} {\text{Cor.}~\ref{cor:UR_not_imply_URD}}
	\arrow[negated]{dr}{\text{Prop.}~\ref{prop:UR_rows_columns}}
    &
    \\	
    &
    &
    \text{\small UR rows}
	\arrow[negated,shorten >= 6pt]{dl}{\text{Prop.}~\ref{prop:URD_rows_columns}}
	\arrow[negated]{ul}{\text{Prop.}~\ref{prop:rows_columns_UR}}
	\\
	\text{\small SURD}=	\text{\small SURDO}
	\arrow[negated,shorten <= 3pt,shorten >= 3pt]{uu}{\text{Ex.}~\ref{ex:SURD_isnot_SSURDO}\ }
	\arrow{r}{\text{Def.}~\ref{def:SURD}}
	\arrow[draw=red,shorten <= 3pt,shorten >= 3pt]{uur}{\text{\color{red} ?}} 
	&
    \text{\small URD}=\text{\small URDO}
	\arrow[negated]{l}{\text{Prop.}~\ref{prop:URD_isnot_SURD}\ }
	\arrow[shorten <= 6pt]{ur}{\text{Prop.}~\ref{prop:URD_rows_columns}}
	\arrow[negated,shorten <= 3pt,shorten >= 3pt]{uu}{\text{Prop.}~\ref{prop:URD_isnot_UR}\ }
    &
\end{tikzcd}
\begin{tikzpicture}[overlay,remember picture]
\node at (-4.8,-1.7){{\text{\footnotesize{Prop.~\ref{prop:URDisURDO}}}}};
\node at (-9.5,-1.7){{\text{\footnotesize{Prop.~\ref{prop:URDisURDO}}}}};
\end{tikzpicture}
\caption{In black are drawn the links between the different notions of recurrence. In red is drawn an open question.}
\label{fig:links_recurrence}
\end{figure}
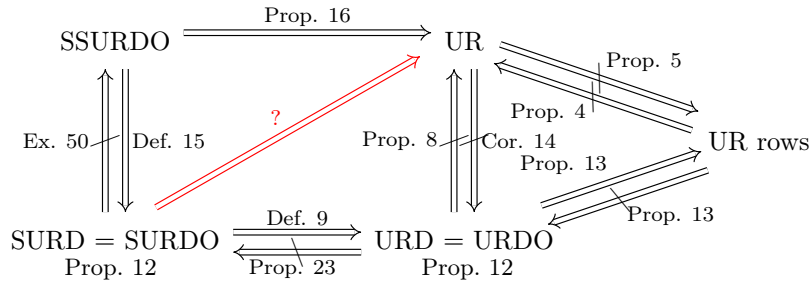

\section{Uniform recurrence along all directions from any origin}
\label{sec:origin}

As a natural generalization of $d$-dimensional URD and SURD infinite words, we could ask that the recurrence property should not just be taken into account on the lines $\{\ell\mathbf{q}\colon\ell\in\N\}$ for all directions $\mathbf{q}$ but on all lines $\{\ell\mathbf{q}+\mathbf{p}\colon\ell\in\N\}$ for all origins $\mathbf{p}$ and directions $\mathbf{q}$. In fact, this would not be a real generalization; the proof of this claim is the purpose of the present section.

\begin{definition}[URDO]
\label{def2}
A $d$-dimensional infinite word $w\colon\N^d\to A$ is \emph{uniformly recurrent along all directions from any origin} (URDO for short) if for each $\mathbf{p}\in\N^d$, the translated $d$-dimensional infinite word $w^{(\mathbf{p})}\colon \N^d\to A,\ \mathbf{i}\mapsto w(\mathbf{i}+\mathbf{p})$ is URD.
\end{definition}

\begin{definition}[SURDO]
A $d$-dimensional infinite word $w\colon\N^d\to A$ is \emph{strongly uniformly recurrent along all directions from any origin} (SURDO for short) if for each $\mathbf{p}\in\N^d$, the translated $d$-dimensional infinite word $w^{(\mathbf{p})}\colon \N^d\to A,\ \mathbf{i}\mapsto w(\mathbf{i}+\mathbf{p})$ is SURD.  
\end{definition}

\begin{proposition}\label{prop:URDisURDO}\ 
\begin{itemize}
\item A $d$-dimensional infinite word is URD if and only if it is URDO.
\item A $d$-dimensional infinite word is SURD if and only if it is SURDO.
\end{itemize}
\end{proposition}

\begin{proof} 
Both conditions are clearly sufficient. Now we prove that they are necessary. Let $w\colon\N^d\to A$ be URD (SURD, respectively), let $\mathbf{p},\mathbf{s} \in \mathbb{N}^d$ and let $f\colon [\![\mathbf{0},\mathbf{s}-\mathbf{1}]\!]\to A $ be the factor of $w$ of size $\mathbf{s}$ at position $\mathbf{p}$: for all $\mathbf{i}\in[\![\mathbf{0},\mathbf{s}-\mathbf{1}]\!]$, $f(\mathbf{i})=w(\mathbf{i}+\mathbf{p})$. We need to prove that for each direction $\mathbf{q}$, there exists $b\in\N$ such that (that there exists $b\in\N$ such that for all directions $\mathbf{q}$, respectively) each factor of length $b$ taken along the line $\ell\mathbf{q}+\mathbf{p}$ contains $f$. The situation is illustrated in Figure~\ref{fig:proof-URDO}.
\begin{figure}[htb]
\centering
\scalebox{0.7}{
\begin{tikzpicture}[scale=1]
	\clip (-0.1,-0.1)rectangle (8,5.5);
	\tikzstyle{every node}=[shape=rectangle,fill=none,draw=none,minimum size=0cm,inner sep=2pt]
	\tikzstyle{every path}=[draw=black,line width = 0.5pt]
	\draw[fill=gray!30] (0,0) rectangle (6.6,3.2);
	\draw (0,0) to (8,0);
	\draw (0,0) to (0,7);
	\draw[fill=gray] (5,2) rectangle (6.6,3.2);
	\draw[dashed] (0,3.2) to (6.6,3.2);
	\draw[dashed] (6.6,0) to (6.6,3.2);
	\node at (5.8,2.6){$f$};
	
	\tikzstyle{every path}=[draw=black,line width = 1pt, ->]
	\draw (0,0) to  node [below] {$\mathbf{p}$} (5,2);
	\draw[<->] (5,2-0.2)  to  node [below] {$s_1$}  (6.6,2-0.2);
	\draw[<->] (6.6+0.2,2)  to  node [right] {$s_2$}  (6.6+0.2,3.2);
	
	\tikzstyle{every path}=[draw=red,line width = 1pt]
	\draw (0,0) to  node [right] {\color{red}$\ell \mathbf{q}$} (2.8*1.3,7*1.3);
	\draw (0+5,0+2) to  node [right] {\color{red}$\ell \mathbf{q}+\mathbf{p}$} (2+5,5+2);
	\end{tikzpicture}}
\caption{Illustration of the proof of Proposition~\ref{prop:URDisURDO} in the bidimensional case.}
\label{fig:proof-URDO}
\end{figure}
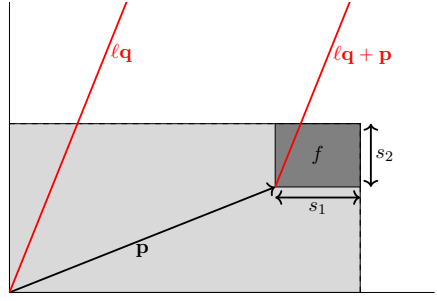
Consider the prefix $p$ of size $\mathbf{p}+\mathbf{s}$ of $w$. Since the word is URD (SURD, respectively), for all directions $\mathbf{q}$, there exists $b'$ such that (there exists $b'$ such that for all directions $\mathbf{q}$, respectively) each factor of length $b'$ taken along the line $\ell\mathbf{q}$ contains $p$. Since $f$ occurs at position $\mathbf{p}$ in $p$, this implies the condition we need with $b=b'$. 
\end{proof}

\begin{proposition}\label{prop:URD_rows_columns}
If a bidimensional infinite word is URD, then all its rows and columns are uniformly recurrent, but the converse does not hold.
\end{proposition}

\begin{proof}
Let $w$ be a URD bidimensional infinite word. From Proposition~\ref{prop:URDisURDO}, $w$ is also URDO. So, any translated word $w^{(\mathbf{p})}$ with $\mathbf{p}=(0,m)$ is also URD. Hence, in $w^{(\mathbf{p})}$ any factor of size of the form $(s,1)$ occurs along the direction $(1,0)$ with bounded gaps. In other words, any row is uniformly recurrent. The argument is similar for the columns.

In order to see that the converse is not true, we can for example consider again the bidimensional word of Proposition~\ref{prop:rows_columns_UR}.
\end{proof}

\begin{corollary}\label{cor:UR_not_imply_URD}
A bidimensional infinite UR word is not necessarily URD. 
\end{corollary}

\begin{proof}
This follows from Propositions~\ref{prop:UR_rows_columns} and~\ref{prop:URD_rows_columns}.
\end{proof}

We can also ask the constant $b$ to be uniform for all the origins. As previously, the notation $\dir{(w^{(\mathbf{p})})}{q}{s}$ designates the unidimensional infinite word along the direction $\mathbf{q}$ with respect to the size $\mathbf{s}$ in the translated $d$-dimensional infinite word $w^{(\mathbf{p})}\colon \N^d\to A,\, \mathbf{i}\mapsto w(\mathbf{i}+\mathbf{p})$. 

\begin{definition}[SSURDO]\label{def:SSURDO}
A $d$-dimensional infinite word $w\colon\N^d\to A$ is \emph{super strongly uniformly recurrent along all directions from any origin} (SSURDO for short) if for all $\mathbf{s}\in\N^d$, there exists $b\in\N$ such that, for each direction $\mathbf{q}\in\N^d$ and each origin $\mathbf{p}\in\N^d$, each length-$b$ factor of $\dir{(w^{(\mathbf{p})})}{q}{s}$ contains the letter $\dir{(w^{(\mathbf{p})})}{q}{s}(0)$.
\end{definition}

Doubly periodic words satisfy the latter definition (take $b$ the product of the coordinates of the periods) but there also exist SSURDO aperiodic words. One of them is given as the fixed point of a bidimensional morphism introduced in Section~\ref{sec:morphism} (see Proposition~\ref{prop:SSURDO}). Note that this notion of SSURDO words is distinct from that of SURD words (see Example~\ref{ex:SURD_isnot_SSURDO}).

\begin{proposition}\label{prop:SSURDO_imply_UR}
A $d$-dimensional SSURDO word is necessarily UR.
\end{proposition}

\begin{proof}
Let $w$ be a $d$-dimensional SSURDO word and let $p$ be a prefix of $w$ of some size $\mathbf{s}$. Let $b$ be the bound from Definition~\ref{def:SSURDO} and $\mathbf{b}=(b,\ldots,b)$. It is enough to prove that any factor of $w$ of size $\mathbf{b}+\mathbf{s}-\mathbf{1}$ contains $p$ as a factor. 

Let $\mathbf{p}=(p_1,\ldots,p_d)$ and let $f$ be the factor of size $\mathbf{b}+\mathbf{s}-\mathbf{1}$ occurring in $w$ at position $\mathbf{p}$. For each $i\in[\![1,d]\!]$, we let $\mathbf{e}_i$ denote the direction $(0,\ldots,0,1,0,\ldots,0)$ with $1$ in the $i$-th coordinate. By definition, in the word $(w^{(\mathbf{0})})_{\mathbf{e}_1,\mathbf{s}}$, each factor of length $b$ contains $p$ (considered as a letter). Therefore, there exists a position $k_1\mathbf{e}_1$ with $p_1 \le k_1 \le p_1 +b-1$ where $p$ occurs in $w$. By definition again, in the word $(w^{(k_1\mathbf{e}_1)})_{\mathbf{e}_2,\mathbf{s}}$, each factor of length $b$ contains an occurrence of $p$. So there exists a position $k_1\mathbf{e}_1+k_2\mathbf{e}_2$ with $p_2\le k_2 \le p_2+b-1$ where $p$ occurs in $w$. Applying the same argument $d-2$ more times, we find a position $k_1 \mathbf{e}_1 +\cdots + k_d\mathbf{e}_d\in[\![\mathbf{p},\mathbf{p}+\mathbf{b}-\mathbf{1}]\!]$ where $p$ occurs in $w$. Thus, $p$ occurs as a factor of $f$ as desired.
\end{proof}

\section{Construction of URD multidimensional words using the $\gcd$}
\label{sec:gcd}

In this section, we consider a specific construction of $d$-dimensional infinite words starting from a single unidimensional infinite word. More precisely, for any $u\colon\N\to A$, we define a $d$-dimensional infinite word $w\colon\N^d\to A$ by setting
\[
    \forall \mathbf{i}\in\N^d, \ w(\mathbf{i})=u(\gcd(\mathbf{i})),
\]
where $\gcd(\mathbf{i})=\gcd(i_1,\ldots,i_d)$ if $\mathbf{i}=(i_1,\ldots,i_d)$. Otherwise stated, one places the infinite word $u$ in every rational direction: for all directions $\mathbf{q}\in\N^d$ and all $\ell\in\N$, we have $w(\ell \mathbf{q})=u(\ell)$.

\begin{lemma}
\label{lemma:gcd}
Let  $\mathbf{q}=(q_1,\ldots,q_d)\in\Z^d$ such that $q_1,\ldots,q_d$ are coprime, 
let $\alpha_1,\ldots,\alpha_d\in\Z$ such that $\alpha_1 q_1+\cdots+ \alpha_d q_d=1$, and let $\mathbf{i}=(i_1,\ldots,i_d)\in\Z^d\setminus \Z\mathbf{q}$. Then, for all $\ell\in\Z$, we have
\begin{align*}
\gcd(\ell \mathbf{q}+\mathbf{i}) &=\gcd\big(\ell +\alpha_1 i_1+\cdots+ \alpha_d i_d,\ \gcd( i_jq_k-i_kq_j \colon j,k\in[\![1,d]\!])\big),\\
\gcd(\ell\mathbf{q}) &=\ell.
\end{align*}
In particular, the sequence $\big(\gcd(\ell\mathbf{q}+\mathbf{i})\big)_{\ell\in\Z}$ is periodic of period $\gcd(i_jq_k-i_kq_j\colon j,k\in[\![1,d]\!])$.
\end{lemma}

\begin{proof}
Let $d=\gcd(\ell \mathbf{q}+\mathbf{i})$ and $D=\gcd\big(\ell +\alpha_1 i_1+\cdots+ \alpha_d i_d,\ \gcd( i_jq_k-i_kq_j \colon j,k\in[\![1,d]\!])\big)$. Then $d$ divides 
\[
	\sum_{j=1}^d \alpha_j(\ell q_j+i_j)
    =\ell\sum_{j=1}^d \alpha_j q_j+\sum_{j=1}^d \alpha_ji_j
	=\ell +\sum_{j=1}^d \alpha_ji_j.
\]
Moreover, for all $j,k\in[\![1,d]\!]$, $d$ also divides $(\ell q_j+i_j)q_k-(\ell q_k+i_k)q_j =i_jq_k-i_kq_j$. This shows that $d\le D$. Conversely, for all $k\in[\![1,d]\!]$, $D$ divides 
\[
	\Big(\ell +\sum_{j=1}^d \alpha_ji_j\Big)q_k
    +\sum_{j\in[\![1,d]\!]} (i_kq_j-i_jq_k)\alpha_j=\ell q_k+i_k.
\]
We obtain that $D\le d$, hence $d=D$. The particular case follows from the fact that $\gcd(a,b)=\gcd(a+b,b)$.
\end{proof}

An arithmetical subsequence of a word $w\colon \N\to A$ is a word $v\colon \N\to A$ such that there exist $p,q\in\N$ with $q\ne 0$ such that, for all $\ell\in\N$, $v(\ell)=w(\ell q+p)$. A proof of the following result can be found in \cite{Avgustinovich--2011}.

\begin{lemma}\label{lem:arithm}
An arithmetical subsequence of a uniformly recurrent infinite word is uniformly recurrent.
\end{lemma}

\begin{example}
Consider the occurrence of the prefix $01$ of the Thue-Morse word at positions multiple of $3$:
\[
    {\bf 01}1{\bf 01}0{\bf 01}1001{\bf 01}1{\bf 0 1} 0{\bf 01} 0110{\bf 01} 1{\bf 01}0{\bf 01} 1001{\bf 01}1001101001
{\bf 01}1{\bf 01}0{\bf 01}1001{\bf 01}10\cdots
\]
From Lemma~\ref{lem:arithm} the distance between any two consecutive such occurrences is bounded.
\end{example}

\begin{theorem}
For any uniformly recurrent word $u\colon\N\to A$, the $d$-dimensional word $w\colon\N^d\to A,\ \mathbf{i}\mapsto u(\gcd(\mathbf{i}))$
is URD.
\end{theorem}

\begin{proof}
Let $u\colon\N\to A$ be a uniformly recurrent word and let $w\colon\N^d\to A$ be the $d$-dimensional word $w\colon\N^d\to A,\ \mathbf{i}\mapsto u(\gcd(\mathbf{i}))$. Let $\mathbf{q}$ be a direction, let $p$ be a prefix of $w$ of some size $\mathbf{s}$ and let $y\colon\N\to A^{[\![0,\mathbf{s}-\mathbf{1}]\!]}$ be the word defined by
\[
	\forall \ell \in\N,\ \forall \mathbf{i}\in[\![0,\mathbf{s}-\mathbf{1}]\!],\  (y(\ell))(\mathbf{i})=w(\ell \mathbf{q}+\mathbf{i}).
\]
We claim that $y$ contains the letter $y(0)=p$ with bounded gaps. By construction of $w$, we have
\[
	\forall \ell \in\N,\ \forall \mathbf{i}\in[\![0,\mathbf{s}-\mathbf{1}]\!],\  (y(\ell))(\mathbf{i})=u(\gcd(\ell \mathbf{q}+\mathbf{i})).
\]
Now the conclusion follows from Lemma~\ref{lemma:gcd} and the uniform recurrence of $u$. More precisely, let
\[
    B=\prod_{\substack{\mathbf{0}\le (i_1,\ldots,i_d)<\mathbf{s} \\ (i_1,\ldots,i_d)\notin\N\mathbf{q}}}
    \gcd(i_jq_k-i_kq_j\colon j,k\in[\![1,d]\!])
\]
and  $r=\min\{\lceil\frac{s_1}{q_1}\rceil,\ldots,\lceil\frac{s_d}{q_d}\rceil\}$. By Lemma~\ref{lem:arithm}, the length-$r$ prefix of $u$ occurs at positions multiples of $B$ in $u$ infinitely often with gaps bounded by some constant $C$. Then, by Lemma~\ref{lemma:gcd},  $p$ occurs infinitely often in $y$ with gaps at most $BC$.
\end{proof}

\section{Recurrence properties of multidimensional rotation words}
\label{sec:rotation}

We illustrate that URD and SURD notions are distinct using a generalization of rotation words to the multidimensional setting. This generalization includes the bidimensional Sturmian words, which were proven to be UR ~\cite{Berthe--Vuillon--2001}.

\begin{definition}
Let $\boldsymbol\alpha=(\alpha_1,\ldots,\alpha_d)\in[0,1)^d$ and $\rho\in[0,1)$ be such that $1,\alpha_1,\ldots,\alpha_d$ are rationally independent and let $\{I_1,\ldots,I_k\}$ be a partition of $[0,1)$ into half-open intervals on the right. The \emph{$d$-dimensional (lower) rotation word} $w\colon\N^d\to [\![1,k]\!]$ (with parameters $\boldsymbol\alpha,I_1,\ldots,I_k,\rho$) is defined as
\[
    \forall \mathbf{i}\in\N^d,\ \forall j\in[\![1,k]\!],\quad
    w(\mathbf{i}) = j \iff (\rho+\mathbf{i} \cdot \boldsymbol\alpha) \bmod 1\in I_j
\]
(where $\mathbf{i} \cdot \boldsymbol\alpha$ is the scalar product $i_1\alpha_1+\cdots+i_d\alpha_d$). Similarly, we can also consider half-open intervals on the left. In this case, we talk about \emph{$d$-dimensional upper rotation words.} 
\end{definition}

Note that for $d=2$, $I_1=[0,\alpha_1)$ and $I_2=[\alpha_1,1)$, we recover the definition of bidimensional Sturmian words from~\cite{Berthe--Vuillon--2001}. 

With the previous notation, for $\mathbf{s}\in\N^d$ and for $f$ a $d$-dimensional finite word of size $\mathbf{s}$ over the alphabet $\{1,\ldots,k\}$, we let 
\[
    I_f=\bigcap_{\mathbf{i}\in[\![\mathbf{0},\mathbf{s}-\mathbf{1}]\!]} R_{\mathbf{i}\cdot\boldsymbol{\alpha}}^{-1}(I_{f(\mathbf{i})})
\]
where $R_a\colon[0,1)\to[0,1),\ x\mapsto (x+a)\bmod 1$. Note that an intersection of intervals on the circle is a union of intervals (it does not have to be connected). Since $I_f$ is an intersection of finitely many intervals, it is also a finite union of nonempty disjoint intervals. We let $n(f)$ denote the number of such intervals and  $I_{f,1}, \dots  I_{f,n(f)}$ the intervals, so that: 
\begin{equation}
\label{eq:union}
    I_f=\bigcup_{j=1}^{n(f)} I_{f,j}.
\end{equation}
If $I_f$ is empty then the union is empty, meaning that there is no interval $I_{f,j}$ at all, or equivalently that $n(f)=0$.

\begin{lemma}
\label{lem:rotation}
Let $w$ be a $d$-dimensional rotation word with parameters $\boldsymbol\alpha, I_1,\ldots, I_k, \rho$. \begin{itemize}
    \item A $d$-dimen\-sional finite word $f$ occurs as a factor of $w$ at some position $\mathbf{p}$ if and only if $(\rho+\mathbf{p} \cdot \boldsymbol\alpha) \bmod 1\in I_f$.
    \item A $d$-dimensional finite word $f$ is a factor of $w$ if and only if $I_f$ is nonempty.
\end{itemize}
\end{lemma}

\begin{proof}
The proof is an adaptation of that of~\cite[Lemma~1]{Berthe--Vuillon--2001}. 
Let $f$ be a $d$-dimensional finite word. Then $f$ occurs in $w$ at position $\mathbf{p}$ if and only if for all $\mathbf{i}\in[\![\mathbf{0},\mathbf{s}-\mathbf{1}]\!]$ we have that  $(\rho+(\mathbf{p}+\mathbf{i}) \cdot \boldsymbol\alpha) \bmod 1\in I_{f(\mathbf{i})}$, which is equivalent to saying that $(\rho+\mathbf{p}\cdot\boldsymbol{\alpha})\bmod 1\in I_f$.

If $I_f$ is nonempty then it is a nonempty union of half-open intervals, and hence $I_f$ has nonempty interior. Moreover, by Kronecker's theorem (see for example~\cite{Hardy-Wright--2008}) and since $\alpha_d$ is irrational, we know that the orbit $\{(\rho+p_d\alpha_d)\bmod 1\colon p_d\in\N\}$ of $\rho$ under the rotation $R_{\alpha_d}$ is dense in $[0,1)$. Therefore, if $I_f$ is nonempty then for any $p_1,\ldots,p_{d-1}\in\N$, there exists some $p_d\in\N$ such that $\rho+p_1\alpha_1+\cdots+p_{d-1}\alpha_{d-1}+p_d\alpha_d$ belongs to $I_f$, so $f$ occurs as a factor of $w$ at position $\mathbf{p}=(p_1,\ldots,p_d)$.
\end{proof}

\begin{proposition}
\label{prop:URD_isnot_SURD}
All $d$-dimensional rotation words are URD, but none of them are SURD.
\end{proposition}

\begin{proof}
Consider a $d$-dimensional rotation word $w$ with parameters $\boldsymbol\alpha, I_1,\ldots, I_k, \rho$. First, we show that $w$ is URD. Let $\mathbf{q}\in\N^d$ be a direction and $\mathbf{s}=(s_1,\ldots,s_d)\in\N^d$. We claim that the unidimensional word $\dir{w}{\mathbf{q}}{\mathbf{s}}$ is the image of a unidimensional rotation word under a letter-to-letter projection. Indeed, by definition, for each $\ell$, the letter $\dir{w}{\mathbf{q}}{\mathbf{s}}(\ell)$ corresponds to the factor of size $\mathbf{s}$ occurring at position $\ell\mathbf{q}$ in $w$. By Lemma~\ref{lem:rotation}, we get that the word $\dir{w}{\mathbf{q}}{\mathbf{s}}$ is the coding of the rotation on the unit circle of the point $\rho$ under the irrational angle $\mathbf{q}\cdot\boldsymbol\alpha$ with respect to the interval partition $\{I_{f_1,1},\ldots,I_{f_1,n(f_1)},\ldots,I_{f_r,1},\ldots,I_{f_r,n(f_r)}\}$ where $f_1,\ldots,f_r$ are the factors of $w$ of size $\mathbf{s}$ and the intervals $I_{f_i,j}$ are defined as in~\eqref{eq:union}. Note that, since for each $i$, the intervals $I_{f_i,1},\ldots,I_{f_i,n(f_i)}$ are coded by the same "letter" $f_i$ in $\dir{w}{\mathbf{q}}{\mathbf{s}}$, we do not necessarily obtain a rotation word but a letter-to-letter projection of a rotation word. Now,  we obtain that $w$ is URD as a direct consequence of the three-gap theorem \cite{Swierczkowski--1959,Slater--1967} stating the following: if $\delta$ is an irrational number and $I$ is an interval of the unit circle then the gaps between the successive integers $j$ such that $\delta j \in I$ take at most three values. So, the letter $\dir{w}{\mathbf{q}}{\mathbf{s}}(0)$ occurs in $\dir{w}{\mathbf{q}}{\mathbf{s}}$ with gaps bounded by the largest gap corresponding to $\delta=\mathbf{q}\cdot\boldsymbol\alpha$ and the interval $I=I_{\dir{w}{\mathbf{q}}{\mathbf{s}}(0),j}$ where $j\in[\![1,k(\dir{w}{\mathbf{q}}{\mathbf{s}}(0))]\!]$ corresponds to the index of the interval $I_{\dir{w}{\mathbf{q}}{\mathbf{s}}(0),j}$ containing $\rho$.

However, $w$ is not SURD since the uniform recurrence constant of $\dir{w}{q}{1}$ can be arbitrarily large depending on the direction $\mathbf{q}$. Indeed, by Kronecker's theorem, for each integer $N$, one can choose $\mathbf{q}_N=(q_1,N, \ldots,N)$ so that $\ell(\mathbf{q}_N\cdot \boldsymbol{\alpha} \bmod 1) < \min(|I_1|,\ldots,|I_k|)$ for any $\ell\in[\![0,N]\!]$. Therefore, the word $w_{\mathbf{q}_N,\mathbf{1}}$ contains all the factors $j^N$ for $j\in[\![1,k]\!]$. 
\end{proof}

To end this section, we present an alternative proof of Proposition~\ref{prop:URD_isnot_SURD} using the notion of direct product of words. As it happens, this second proof reveals a property of $d$-dimensional rotation words which is stronger than the URD property (see Remark~\ref{rem:stronger-than-URD}). Further, we hope that this technique could be useful in order to prove that some other families of $d$-dimensional infinite words are URD. 

Recall that the \emph{direct product} of two unidimensional words $v\colon\N\to A$ and $w\colon\N\to B$ (possibly over different alphabets $A$ and $B$) is defined as the word $v\times w\colon\N\to A\times B$ where the $i$-th letter is $(v(i),w(i))$; for example see~\cite{Salimov--2010}. The direct product of $k\ge 2$ unidimensional words can be defined inductively.

First, we need a lemma based on Furstenberg's results \cite{Furstenberg--1981} and their consequences on the direct product of unidimensional rotation words.

\begin{lemma}
\label{lem:direct_prod_sturmian} 
Any direct product of unidimensional lower (resp.\ upper) rotation words is uniformly recurrent.
\end{lemma}

\begin{proof}
Let $k\ge 2$ and consider $k$ unidimensional lower (resp.\ upper) rotation words $\mathcal{R}_1,\ldots,\mathcal{R}_k$. For each $i$, suppose that $\mathcal{R}_i$ has slope $\alpha_i$ and intercept $\rho_i$. Let $T_i$ be the transformation associated with $\mathcal{R}_i$, i.e.\ $T_i\colon[0,1)\to[0,1),\, x\mapsto (x+\alpha_i)\bmod 1$. By definition, $\mathcal{R}_i$ is the coding of the orbit of the intercept $\rho_i$ in the dynamical system $([0,1),T_i)$ with respect to some interval partition $(I_{i,1},\ldots,I_{i,\ell_i})$ of $[0,1)$ where each interval $I_{i,j}$ is half open on the right. Moreover, the direct product of $k$ codings can be seen as the coding of the dynamical system product $([0,1)^k,T_1\times\cdots\times T_k)$ where $T_1\times\cdots\times T_k\colon (x_1,\ldots,x_k)\mapsto (T_1(x_1),\ldots,T_k(x_k))$.

The maps $T_i$ correspond to the transformation $T$ defined in \cite[Prop. 5.4]{Furstenberg--1981} with $d=1$. Thus, from a dynamical point of view, every point of $([0,1),T_i)$ is recurrent \cite[Prop. 5.4]{Furstenberg--1981} and their product with any recurrent point of $([0,1),T_j)$ is also recurrent \cite[Prop.\ 5.5]{Furstenberg--1981} for the product system. Such points are called \emph{strongly recurrent} by Furstenberg. Since the direct product of strongly recurrent points is also strongly recurrent \cite[Lem.\ 5.10]{Furstenberg--1981} and since strong recurrence implies uniform recurrence \cite[Thm. 5.9]{Furstenberg--1981}, we obtain that the product of any $k$ points of $([0,1),T_1), \ldots, ([0,1),T_k)$ respectively is uniformly recurrent with respect to the product system $([0,1)^k,T_1\times\cdots\times T_k)$.

Finally, since $\mathcal{R}_1,\ldots,\mathcal{R}_k$ are all rotation words of the same orientation of the intervals, there is no ambiguity in the coding and the dynamical systems results can be translated in terms of words. Therefore, their direct product is uniformly recurrent.  
\end{proof}

\begin{proof}[Alternative proof of the URD part of Proposition~\ref{prop:URD_isnot_SURD}]
Consider a $d$-dimensional rotation word $w$ with parameters $\boldsymbol\alpha, I_1,\ldots, I_k, \rho$. Let $\mathbf{q}\in\N^d$ be a direction and $\mathbf{s}=(s_1,\ldots,s_d)\in\N^d$. For any $\mathbf{p}\in\N^d$, the unidimensional word  $(w(\ell\mathbf{q}+\mathbf{p}))_{\ell\in\N}$ is a rotation word. Indeed, it is the coding of the rotation of the point $\rho+\mathbf{p}\cdot\boldsymbol\alpha$ of the unit circle under the irrational angle $\mathbf{q}\cdot\boldsymbol\alpha$, with respect to the partition into the intervals $I_1,\ldots,I_k$. Therefore, the word $\dir{w}{q}{s}$ is a direct product of $s_1\cdots s_d$ unidimensional rotation words (of the same orientation):
\[
    \dir{w}{q}{s}
    =\bigtimes_{\mathbf{i}\in[\![\mathbf{0},\mathbf{s}-\mathbf{1}]\!]} 
    w(\ell\mathbf{q}+\mathbf{i})_{\ell\in\N}
\]
By Lemma~\ref{lem:direct_prod_sturmian}, we obtain that $\dir{w}{q}{s}$ is uniformly recurrent. This proves that $w$ is URD. 
\end{proof}

\begin{remark}
\label{rem:stronger-than-URD}
We argue that the second proof of Proposition~\ref{prop:URD_isnot_SURD} shows that $d$-dimensional rotation words satisfy a stronger property than URD which is not the SURD property. For any direction $\mathbf{q}$ and any position $\mathbf{p}$, the rotation angle of $(w(\ell\mathbf{q}+\mathbf{p}))_{\ell\in\N}$ is independent of $\mathbf{p}$. Moreover,
for any size $\mathbf{s}$, any direction $\mathbf{q}$ and any position $\mathbf{p}$, we have
\[
    \dir{w}{q}{s}^{(\mathbf{p})}
    =\bigtimes_{\mathbf{i}\in[\![\mathbf{0},\mathbf{s}-\mathbf{1}]\!]} 
    w(\ell\mathbf{q}+\mathbf{i}+\mathbf{p})_{\ell\in\N}.
\]
Thus, for any size $\mathbf{s}$ and any direction $\mathbf{q}$, there exists a constant $b$ such that for any $\mathbf{p}$, each factor of length $b$ of the unidimensional word
$\dir{w}{q}{s}^{(\mathbf{p})}$ contains all factors of size $\mathbf{s}$. Indeed, the constant $b$ only depends on the rotation angles of the words $    w(\ell\mathbf{q}+\mathbf{i}+\mathbf{p})_{\ell\in\N}$, hence is independent of the origin $\mathbf{p}$ (which is stronger than URD), although depends on the direction $\mathbf{q}$ (which is weaker than SURD).
\end{remark}

\begin{remark} In the particular case of rotation words of the same slope $\alpha$, one can directly prove (without using Furstenberg's results) that their direct product is uniformly recurrent, except in some exceptional cases described below. See~\cite{Didier--1998} for similar concerns on rotation words.

The proof goes as follows. Let $\mathcal{R}_1,\ldots,\mathcal{R}_k$ be unidimensional lower (resp.\ upper) rotation word of intercepts $\rho_1,\ldots,\rho_k$ respectively. As in the proof of Proposition~\ref{prop:URD_isnot_SURD}, for each $i$, the factor of length $s$ at position $m$ in $\mathcal{R}_i$ corresponds to the interval of the point $\rho_i+m\alpha$. 

For each $i$, let us shift all the intervals of the $i$-th circle by $\rho_1-\rho_i$. Now the factor at position $m$ of each $\mathcal{R}_i$ corresponds to the (shifted) interval of the point $\rho_i+(\rho_1-\rho_i)+m\alpha = \rho_1+m\alpha$. Consider the intervals created as the intersections of all shifted intervals (we have at most $n_1\cdots n_k$ of them where each $n_i$ is the number of intervals in the interval partition corresponding to $\mathcal{R}_i$). These new intervals correspond to the factors of the product $\mathcal{R}_1\times\cdots\times\mathcal{R}_k$. Namely, the factor at position $m$ of $\mathcal{R}_1\times\cdots\times\mathcal{R}_k$ corresponds to the interval containing the point $\rho_1+m\alpha$. This shows that $\mathcal{R}_1\times\cdots\times\mathcal{R}_k$ is a rotation word. So, it is uniformly recurrent by the three-gap theorem. 

If some words are upper rotational, and some are lower rotational, their direct product might not be uniformly recurrent. This occurs when some intersection of the intervals is a single point, which can only happen in the case when in one of the words the intervals are half-open on the right, and in the other one they are half-open on the left, and the orbit of each point contains this point. On the other hand, if one of the words never touches the intervals endings (which corresponds to an orbit not containing zero), it means that orientation does not play any role for this word and we can assume it is the same as for the other word.
\end{remark}

\section{Fixed points of multidimensional square morphisms}
\label{sec:morphism}

Similarly to unidimensional words, one can define morphisms and their fixed points in any dimension; for example, see~\cite{Charlier--Karki--Rigo--2010,Rigo--Maes--2002,Salon--1987}. For other kinds of multidimensional substitutions, we refer to the survey \cite{PriebeFrank--2008}.
 For simplicity, we only consider constant length morphisms. 
\begin{definition}
A \emph{$d$-dimensional morphism} of constant size $\mathbf{s}=(s_1,\ldots,s_d)\in\N^d$ is a map $\varphi\colon A\to A^{[\![\mathbf{0},\mathbf{s}-\mathbf{1}]\!]}$. For each $a\in A$ and for each integer $n\ge 2$, $\varphi^n(a)$ is recursively defined as 
\[
	\varphi^n(a)\colon 
    [\![\mathbf{0},\mathbf{s}^n-\mathbf{1}]\!]\to A,\ 
    \mathbf{i}\mapsto \Big(\varphi\big((\varphi^{n-1}(a))(\mathbf{q})\big)\Big)(\mathbf{r}),
\]
where $\mathbf{q}$ and $\mathbf{r}$ are defined by the componentwise Euclidean division of $\mathbf{i}$ by $\mathbf{s}$: $\mathbf{i}=\mathbf{q}\mathbf{s}+\mathbf{r}$. With these notation, the \emph{preimage} of the letter $\big(\varphi^n(a)\big)(\mathbf{i})$  is the letter $\big(\varphi^{n-1}(a)\big)(\mathbf{q})$. In the case $\mathbf{s}=(s,\ldots,s)$, we say that $\varphi$ is a \emph{$d$-dimensional square morphism of size $s$}.
\end{definition}

Note that $\varphi^n(a)$ is obtained by concatenating $\prod_{i=1}^d s_i$ copies of the images $\varphi^{n-1}(b)$ for the letters $b$ occurring in $\varphi(a)$. For instance, if $d=2$ and $\mathbf{s}=(s_1,s_2)$, the $n$-th image $\varphi^n(a)$ has size $\mathbf{s}^n=(s_1^n,s_2^n)$ and,  with the convention of Remark~\ref{rem:convention-2D}, we have
\[
    \varphi^n(a)=
	\begin{bmatrix}
	\varphi^{n-1}(\varphi(a)_{0,s_2-1}) & \cdots & \varphi^{n-1}(\varphi(a)_{s_1-1,s_2-1})\\ 
    \vdots & & \vdots \\
    \varphi^{n-1}(\varphi(a)_{0,0}) & \cdots & \varphi^{n-1}(\varphi(a)_{s_1-1,0})
	\end{bmatrix}
\] 
where we have used the lighter notation $\varphi(a)_{i,j}$ instead of $\big(\varphi(a)\big)(i,j)$.

\begin{example}
In Figure~\ref{fig:preimage},
\begin{figure}[htbp]
\centering
\scalebox{0.8}{
\begin{tikzpicture}[overlay,remember picture]
	\tikzstyle{every path}=[draw=gray,line width = 1pt]
	\coordinate (a) at ( $ (pic cs:Z1)  - (0.1,0.1)  $);
	\coordinate (b) at ( $ (pic cs:Z2)  + (0.2,0.3) $);
	\draw[fill=gray!50] (a) rectangle (b);
\end{tikzpicture}
$\begin{array}{c|ccccccccccccccccccccccccccc}
	7&1&0&1& 1&1&0\tikzmark{A2}& 1&0&1& 1&0&1& 1&0&1&   1&1&0& 1&0&1& 1&1&0& 1&0&1\\
	6&1&1&0& \tikzmark{A1}0&1&1& 1&1&0& 1&1&0& 1&1&0&   0&1&1& 1&1&0& 0&1&1& 1&1&0\\
	5&1&0&1& 1&0&1& 1&1&0& 1&1&0\tikzmark{B2} & 1&0&1&  1&0&1& 1&0&1& 1&0&1& 1&1&0\\
	4&1&1&0& 1&1&0& 0&1&1& \tikzmark{B1}0&1&1& 1&1&0&   1&1&0& 1&1&0& 1&1&0& 0&1&1 \\
	3&1&0\tikzmark{A}&1& 1&1&0& 1&0&1\tikzmark{Z2}&  1&0&1& 1&1&0&  1&0&1& 1&0&1& 1&0&1\tikzmark{C2}& 1&1&0\\
	2&1&1&0& 0\tikzmark{B}&1&1& 1&1&0& 1&1&0&  0&1&1&  1&1&0& 1&1&0& \tikzmark{C1}1&1&0& 0&1&1 \\
	1&1&0&1& 1&0&1& 1&1\tikzmark{C}&0&  1&0&1 & 1&0&1& 1&1&0& 1&1&0& 1&0&1& 1&0&1\\
	0&\tikzmark{Z1}1&1&0& 1&1&0& 0&1&1& 1&1&0& 1&1&0&   0&1&1& 0&1&1& 1&1&0& 1&1&0\\\hline
	&0&1&2&  3&4&5& 6&7&8&9&10&11&12&13&14&15&16&17&18&19&20&21&22&23&24&25&26
\end{array}$
\begin{tikzpicture}[overlay,remember picture]
\tikzstyle{every path}=[draw=red,line width = 1pt]
\coordinate (x) at ( $ (pic cs:A)  + (0.12,0.3)$);
\draw (x.south) --++(0,-0.4) --++(-0.4,0)--++(0,0.4)--++(0.4,0)--++(0,-0.4);
\coordinate (a) at ( $ (pic cs:A1)  - (0.1,0.1) $);
\coordinate (b) at ( $ (pic cs:A2) + (0.2,0.4) $);
\draw (a) rectangle (b);
\draw[->] (x) to [bend left = 10]  (a);

\tikzstyle{every path}=[draw=blue,line width = 1pt]
\coordinate (x) at ( $ (pic cs:B)  + (0.12,0.3)$);
\draw (x.south) --++(0,-0.4) --++(-0.4,0)--++(0,0.4)--++(0.4,0)--++(0,-0.4);
\coordinate (a) at ( $ (pic cs:B1)  - (0.1,0.1) $);
\coordinate (b) at ( $ (pic cs:B2) + (0.2,0.4) $);
\draw (a) rectangle (b);
\draw[->] (x) to [bend left = 10]  (a);

\tikzstyle{every path}=[draw=purple,line width = 1pt]
\coordinate (x) at ( $ (pic cs:C)  + (0.12,0.3)$);
\draw (x.south) --++(0,-0.4) --++(-0.4,0)--++(0,0.4)--++(0.4,0)--++(0,-0.4);
\coordinate (a) at ( $ (pic cs:C1)  - (0.1,0.1) $);
\coordinate (b) at ( $ (pic cs:C2) + (0.2,0.4) $);
\draw (a) rectangle (b);
\draw[->] (x) to [bend left = 10]  (a);

\end{tikzpicture}
}
\caption{Third iteration of the morphism} 
\hspace{4.5cm} $0\mapsto\begin{bmatrix}
1 & 1 & 0\\ 0 & 0 & 1
\end{bmatrix},\ 
1\mapsto\begin{bmatrix}
1 & 0 & 1\\ 1 & 1 & 0
\end{bmatrix}$ 
starting from $1$.
\label{fig:preimage}
\end{figure}
the third iteration of a bidimensional morphism $\varphi$ of size $\mathbf{s}=(3,2)$ is given. The gray zone corresponds to $\varphi^2(1)$. The preimages of different letters is highlighted in colors. For instance, the preimage of $\varphi^3(1)_{4,7}$ (in red) is the letter $\varphi^2(1)_{1,3}$ as $(4,7)=(1,3)\mathbf{s}+(1,1)$ (where the product and sum are understood componentwise). Note that it is also the preimage of $\varphi^3(1)_{3,6},\varphi^3(1)_{3,7}$ and $\varphi^3(1)_{5,7}$ for example.
\end{example}

\begin{definition}
Let $\varphi$ be a $d$-dimensional morphism such that there exists $a\in A$ with $\varphi(a)_{0,0}=a$. We say that $\varphi$ is \emph{prolongable} on $a$ and the limit $\lim_{n\to\infty}\varphi^n(a)$ is well defined. The limit $d$-dimensional infinite word so obtained is called the \emph{fixed point} of $\varphi$ beginning with $a$ and it is denoted by $\varphi^\omega(a)$. A $d$-dimensional infinite word is said to be \emph{pure morphic} if it is the fixed point of a $d$-dimensional morphism.
\end{definition}

\begin{example}
Figure~\ref{fig:sierpinski} 
\begin{figure}[htb]
\centering
\scalebox{0.7}{
\includegraphics[width=0.02\textwidth]{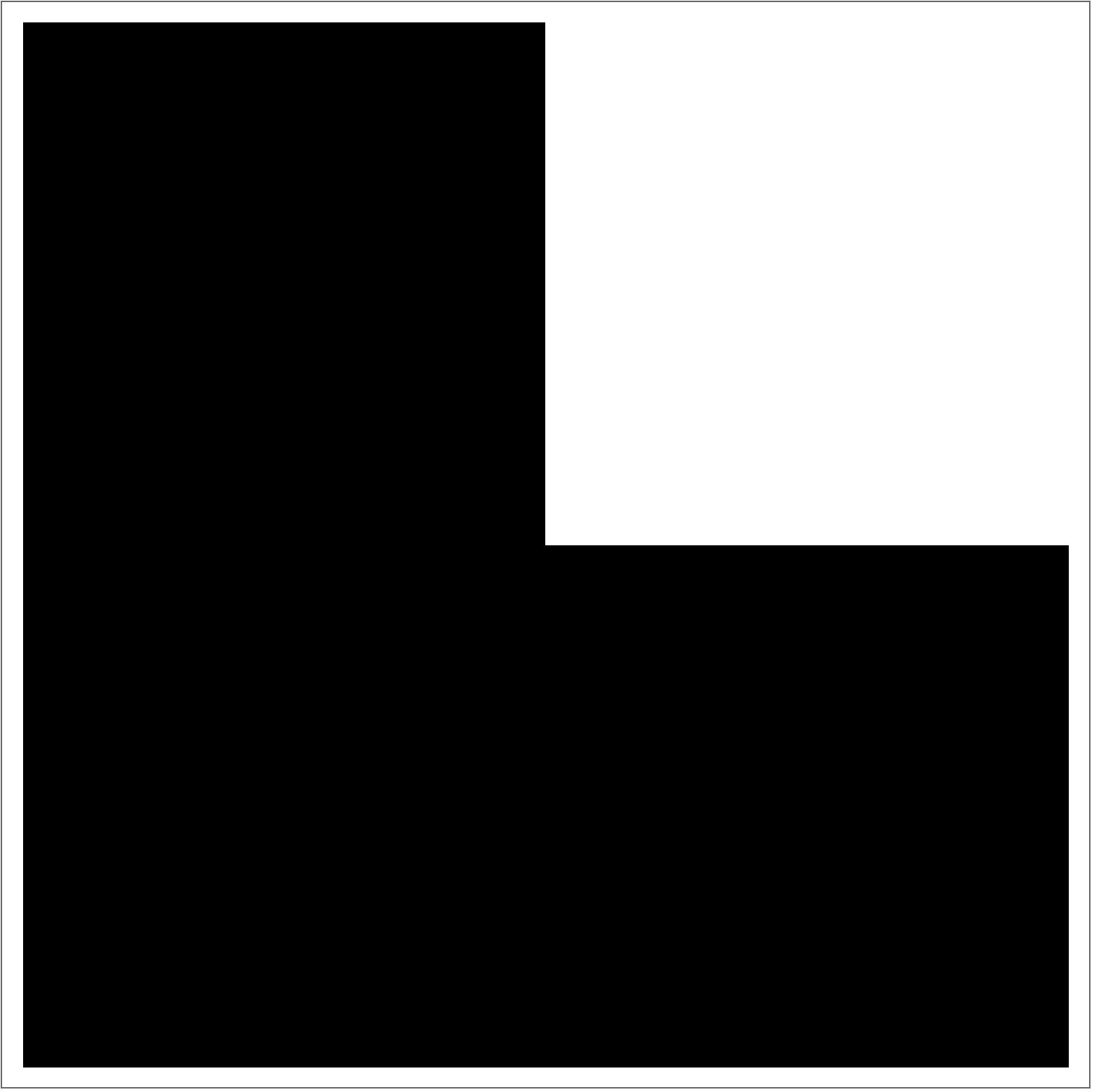}\quad
\includegraphics[width=0.04\textwidth]{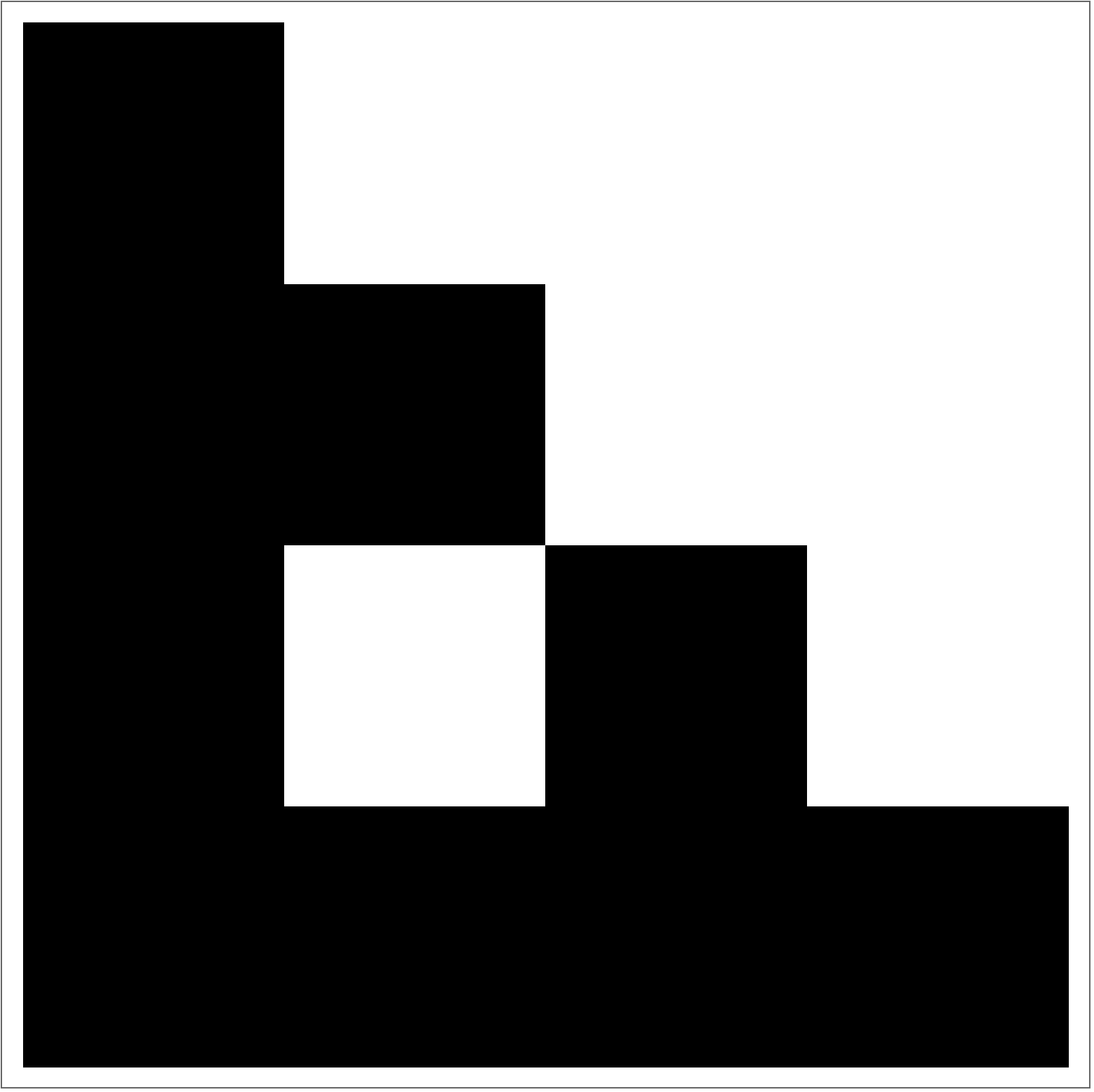}\quad
\includegraphics[width=0.08\textwidth]{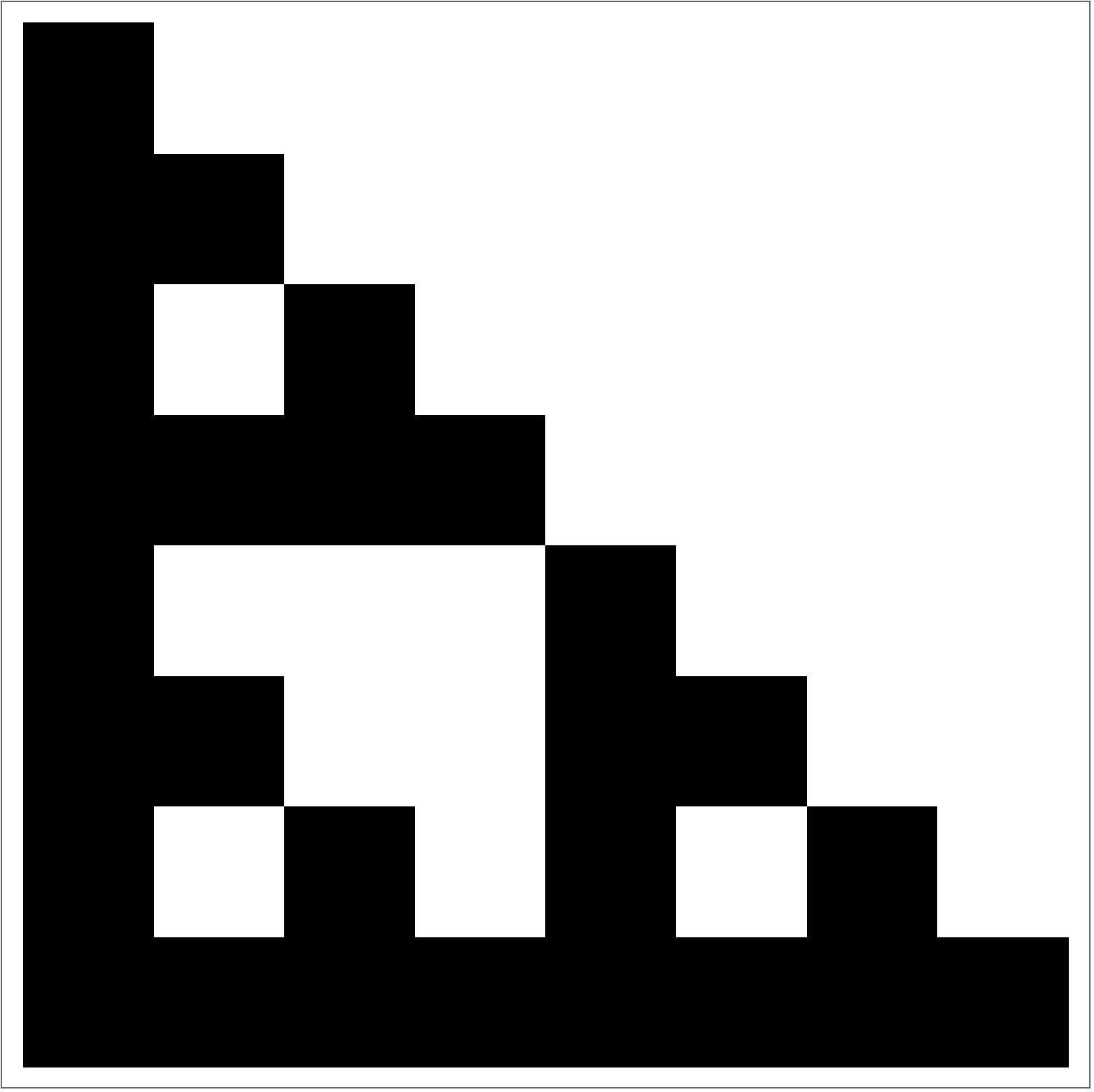}\quad
\includegraphics[width=0.16\textwidth]{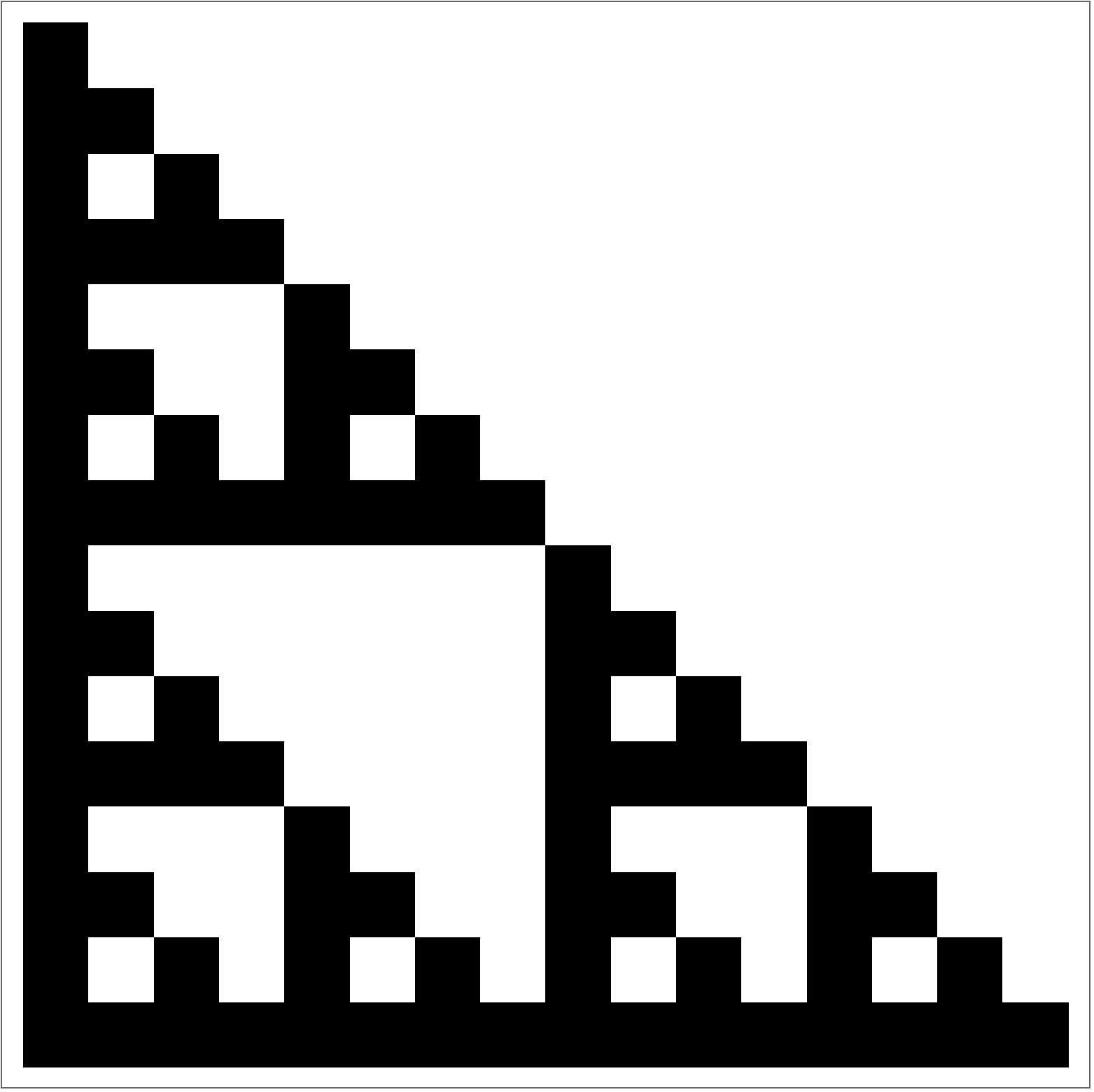}\quad
\includegraphics[width=0.32\textwidth]{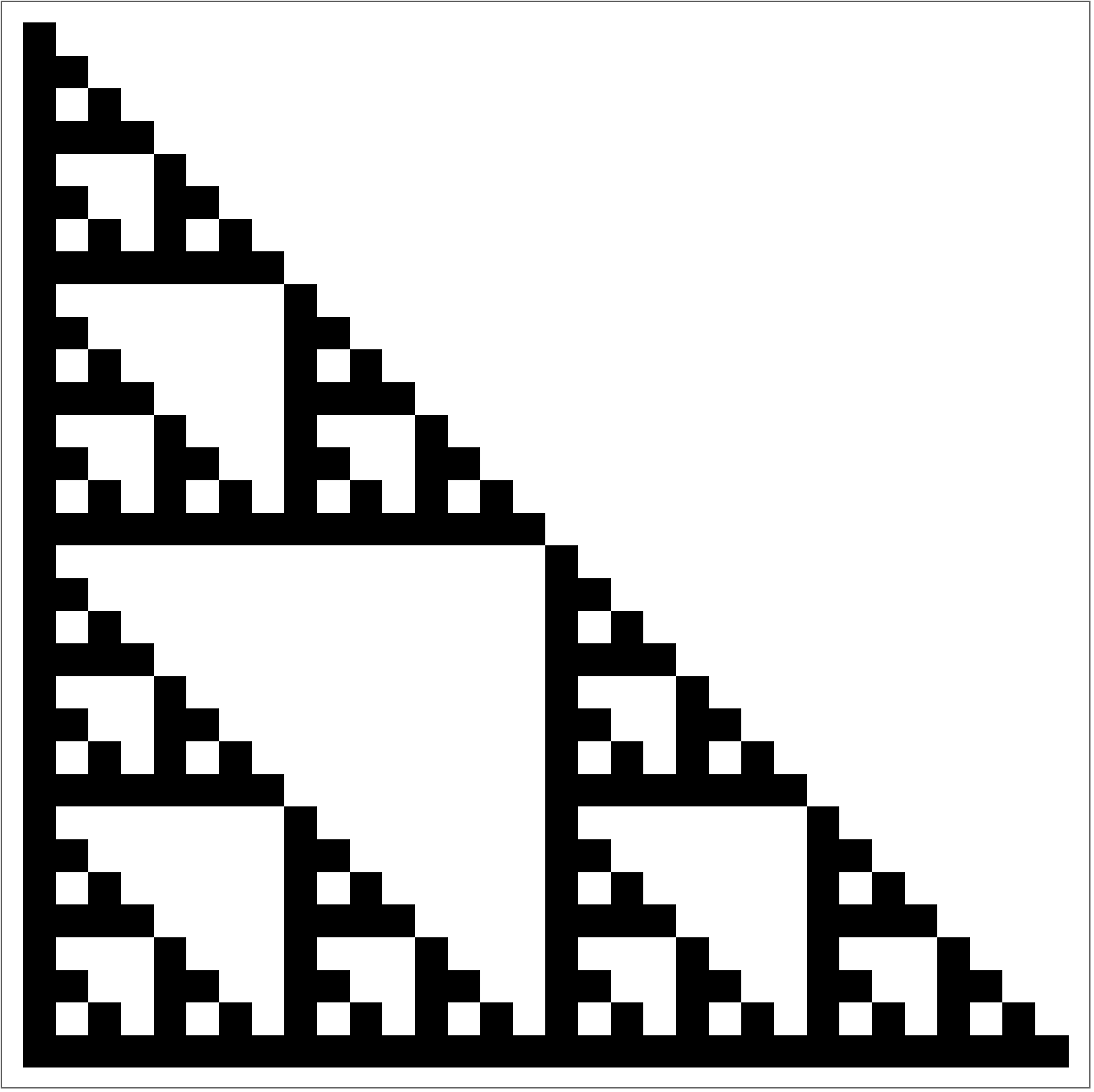}
}
\caption{The first five iterations of the 2D morphism.}
\hspace{1.5cm}$0\mapsto\begin{bmatrix}
0 & 0\\ 0 & 0
\end{bmatrix}$, 
$1\mapsto\begin{bmatrix}
1 & 0\\ 1 & 1
\end{bmatrix}$ starting from $1$.
\label{fig:sierpinski}
\end{figure}
depicts the first five iterations of a bidimensional square morphism with the convention that a black (resp.\ white) cell represents the letter 1 (resp.\ 0). The limit object of this process is the famous Sierpinski gasket \cite{Stewart--1995}.
\end{example}

A first interesting observation is that in order to study the uniform recurrence along all directions (URD) of $d$-dimensional infinite words of the form $\varphi^\omega(a)$ for a square morphism $\varphi$, we only have to consider the distances between consecutive occurrences of the letter $a$. 

\begin{proposition}\label{prop:reduction}
Let $w$ be a fixed point of a $d$-dimensional square morphism of size $s$ and let $\mathbf{q}$ be a direction. If there exists $b\in\N$ such that $w(\mathbf{0})$ occurs infinitely often along $\mathbf{q}$ with gaps at most $b$, then for all $\mathbf{m}\in\N^d$, the prefix of size $\mathbf{m}$ of $w$ occurs infinitely often along $\mathbf{q}$
with gaps at most $s^{\lceil \log_s( \max\mathbf{m})\rceil} b$.
\end{proposition}

\begin{proof}
Let $\mathbf{m}\in\N^d$ and let $p$ be the prefix of size $\mathbf{m}$ of $w$. Let $r$ be the integer defined by $s^{r-1}<\max\mathbf{m}\le s^r$. In the $d$-dimensional infinite word $w$, if the letter $w(\mathbf{0})$ occurs at position $\mathbf{i}$, then the image $\varphi^r(w(\mathbf{0}))$ occurs at position $s^r\mathbf{i}$. Therefore and because we consider a square morphism, if $w(\mathbf{0})$ occurs infinitely often along $\mathbf{q}$ with gaps at most $b$, then $p$ occurs infinitely often along $\mathbf{q}$ with gaps at most $s^r b$.
\end{proof}

In order to provide a family of SURD $d$-dimensional infinite words, we introduce the following definition.

\begin{definition}
For an integer $s\ge2$ and $\mathbf{i}=(i_1,\ldots,i_d)\in(\Z/s\Z)^d$ such that $i_1,\ldots,i_d$ are coprime, we define $\langle \mathbf{i}\rangle$ to be the additive subgroup of $(\Z/s\Z)^d$ that is generated by $\mathbf{i}$:
\[
    \langle \mathbf{i}\rangle = \{ k\mathbf{i}\colon k\in \Z/s\Z\}.
\]
Then, we let $\mathcal{C}(s)$ be the family of all cyclic subgroups of $(\Z/s\Z)^d$ generated by elements with $\gcd(\mathbf{i})=1$:
\[
    \mathcal{C}(s)=\{\langle \mathbf{i}\rangle\colon \mathbf{i}\in(\Z/s\Z)^d, \ \gcd(\mathbf{i})=1\}.
\]
\end{definition}

\begin{proposition}
\label{proposition:main morphic}
If $\varphi$ is a $d$-dimensional square morphism of size $s$ prolongable on $a\in A$ and such that, for every $C\in\mathcal{C}(s)$, there exists $\mathbf{i}\in C$ such that $\varphi(b)_\mathbf{i}=a$  for each $b\in A$, then its fixed point $\varphi^\omega(a)$ is SURD. More precisely, for each $\mathbf{m}\in\N^d$, the prefix of size $\mathbf{m}$ of $\varphi^\omega(a)$ occurs infinitely often along any direction with gaps at most $s^{\lceil \log_s( \max\mathbf{m})\rceil+1}$.
\end{proposition}

\begin{proof}
Let $\mathbf{q}\in\N^d$ be a given direction. Let $\mathbf{r}=\mathbf{q}\bmod s$ (componentwise) and $d=\gcd(\mathbf{r})$. By hypothesis, there exists $\mathbf{i}\in\langle\frac 1d\mathbf{r}\rangle$ such that $(\varphi(b))_\mathbf{i}=a$ for each $b\in A$. Let $k\in[\![0,s-1]\!]$ such that $\mathbf{i}=\frac kd\mathbf{r} \bmod s$. Then $k\mathbf{q}\equiv k\mathbf{r}\equiv d\mathbf{i}\pmod s$. Observe that $\gcd(d,s)$ divides $\mathbf{r}$ and $s$, hence also divides $\mathbf{q}$. This implies that $\gcd(d,s)=1$. Let $\ell=d^{-1}k\bmod s$. Then $\ell \mathbf{q}\equiv \mathbf{i}\pmod s$. We obtain that for all $n\in\N$, $(\ell +ns)\mathbf{q}\equiv \mathbf{i}\pmod s$, hence  $(\varphi^\omega(a))_{(\ell+ns)\mathbf{q}}=a$. This proves that the letter $a$ occurs infinitely often in $\varphi^\omega(a)$ along the direction $\mathbf{q}$ with gaps at most $s$.

Now let $\mathbf{m}\in\N^d$ and consider the prefix $p$ of size $\mathbf{m}$ of $\varphi^\omega(a)$. From the first part of the proof and by using Proposition~\ref{prop:reduction}, we obtain that $p$ occurs infinitely often along any direction with gaps at most $s^{\lceil \log_s( \max\mathbf{m})\rceil+1}$.
\end{proof}

Since each subgroup of $(\Z/s\Z)^d$ contains $\mathbf{0}$, the following result is immediate.

\begin{corollary}\label{cor:1}
Let $\varphi$ be a $d$-dimensional square morphism of size $s$ such that $(\varphi(b))_\mathbf{0}=a$ for each $b\in A$. Then the fixed point $\varphi^\omega(a)$ is SURD. More precisely, for each  $\mathbf{m}\in\N^d$, the prefix of size $\mathbf{m}$ of $\varphi^\omega(a)$ occurs infinitely often along any direction with gaps at most $s^{\lceil \log_s( \max\mathbf{m})\rceil+1}$.
\end{corollary}

When the alphabet $A$ is binary (in which case we assume without loss of generality that $A=\{0,1\}$), then we talk about \emph{binary} morphism and we always consider that it has a fixed point beginning with $1$.

\begin{example}
By Corollary~\ref{cor:1}, the fixed point $\varphi^\omega(1)$ of
\[
\varphi\colon
0\mapsto\begin{bmatrix}
0 & 0\\ 1 & 0
\end{bmatrix},\quad 
1\mapsto\begin{bmatrix}
0 & 1\\ 1 & 0
\end{bmatrix}
\]
is SURD: for all $(m,n)\in\N^2$, the prefix of size $(m,n)$ of $\varphi^\omega(1)$ occurs infinitely often along any direction with gaps at most $2^{\lceil \log_s(\max(m,n))\rceil+1}$.
\end{example}

\begin{remark}
When the size $s$ is prime, the subgroups of $\mathcal{C}(s)$ corresponding to any two elements $\mathbf{i}$ and $ \mathbf{j}$ with coprime coordinates either coincide or have only the element $\mathbf{0}$ in common. Therefore we have exactly $\frac{s^d-1}{s-1}$ distinct subgroups. In particular, for $d=2$, this gives $s+1$ distinct subgroups. Hence we can consider a partition of $(\Z/s\Z)^d$ into $s+2$ sets: $s+1$ subgroups without $\mathbf{0}$ and $\mathbf{0}$ itself.  When $s$ is not prime, the structure is a bit more complicated and we do not have such a nice partition. Below we consider
examples to illustrate the two situations.
\end{remark}

\begin{example}
Partition for $s=5$ and $d=2$ can be illustrated by the following picture where each letter in $\{\alpha,\ldots, \zeta\}$ represents a subgroup:
\[
	\begin{bmatrix}
	\beta & \zeta & \epsilon & \delta & \gamma \\
	\beta & \delta & \zeta & \gamma & \epsilon \\
	\beta & \epsilon & \gamma & \zeta & \delta \\
	\beta & \gamma & \delta & \epsilon & \zeta \\
	0 & \alpha & \alpha & \alpha & \alpha
    \end{bmatrix}
\]
Due to Proposition~\ref{proposition:main morphic}, in order to obtain a SURD fixed point of a bidimensional square morphism, it is enough to have the letter $a$ in one of the coordinates marked by each Greek letter in the image of each letter $b\in A$. And by Corollary~\ref{cor:1}, having the letter $a$ in the coordinate $(0,0)$ in the image of each letter is enough.
\end{example}

\begin{example}
For $s=6$ and $d=2$, one has 12 subgroups (which can be checked by considering the 36 possible cases of pairs of remainders of the Euclidean division by $6$, out of which there are only 21 coprime pairs to consider): 
\[
\begin{bmatrix}[c|c|c|c|c|c]
\beta 			& \kappa & \theta 		& \zeta 		& \delta 		& \gamma \\ 
\hline 
\beta,\zeta,\lambda 		& \iota & \delta,\kappa,\epsilon		& \lambda 		& \gamma,\theta,\iota 	& \epsilon \\ \hline 
\beta,\delta,\theta,\mu 	& \eta & \mu 		& \gamma,\zeta,\kappa,\eta 	& \mu 		& \eta \\ 
\hline 
\beta,\zeta,\lambda 		& \epsilon & \gamma,\theta,\iota 	& \lambda 		& \delta,\epsilon,\kappa 	& \iota \\ \hline 
\beta 			& \gamma & \delta 		& \zeta 		& \theta 		& \kappa \\ \hline
0 			& \alpha & \alpha,\eta,\mu	& \alpha,\lambda,\iota,\epsilon	& \alpha,\eta,\mu	& \alpha
\end{bmatrix}
\]
Here are the correspondence between the 12 subgroups and letters (where we do not write $(0,0)$, which belongs to every subgroup):
\[
\begin{tabular}{c|l|l}
$\alpha$	& $(1,0)$
& $\{(1,0),(2,0),(3,0),(4,0),(5,0)\}$\\
$\beta$	& $(0,1)$
& $\{(0,1),(0,2),(0,3),(0,4),(0,5)\}$\\
$\gamma$	& $(1,1)$ 
& $\{(1,1),(2,2),(3,3),(4,4),(5,5)\}$\\
$\delta$	& $(2,1),(4,5)$ 
& $\{(2,1),(4,2),(0,3),(2,4),(4,5)\}$\\
$\epsilon$	& $(1,2),(5,4)$
& $\{(1,2),(2,4),(3,0),(4,2),(5,4)\}$\\
$\zeta$	& $(3,1),(3,5)$
& $\{(3,1),(0,2),(3,3),(0,4),(3,5)\}$\\
$\eta$	& $(1,3),(5,3)$
& $\{(1,3),(2,0),(3,3),(4,0),(5,3)\}$\\
$\theta$	& $(4,1),(2,5)$
& $\{(4,1),(2,2),(0,3),(4,4),(2,5)\}$\\
$\iota$	& $(1,4),(5,2)$
& $\{(1,4),(2,2),(3,0),(4,4),(5,2)\}$\\
$\kappa$	& $(5,1),(1,5)$
& $\{(5,1),(4,2),(3,3),(2,4),(1,5)\}$\\
$\lambda$	& $(3,4),(3,2)$
& $\{(3,4),(0,2),(3,0),(0,4),(3,2)\}$\\
$\mu$ & $(4,3),(2,3)$
& $\{(4,3),(2,0),(0,3),(4,0),(2,3)\}$
\end{tabular}
\]
We remark that here the subgroups intersect. For example, the first and third subgroups have the element $(3,3)$ in common. Due to Proposition~\ref{proposition:main morphic}, in order to obtain a SURD word, it suffices to have the letter $a$ in the image of each letter in at least one of the elements of each subgroup. For example, it is the case of the fixed point of any morphism with $a$'s in the marked positions in the images of each letter: 
\[
\begin{bmatrix}
* & * & * & * & * & *\\
a & * & a & * & * & *\\
* & * & * & a & * & *\\
* & * & a & * & * & *\\
* & * & * & * & * & *\\
* & * & * & * & a & *\\
\end{bmatrix}
\]
\end{example}

\begin{corollary}\label{cor:morphism power}
If $\psi$ is a $d$-dimensional square morphism of size $s$ such that for some integer $i$, its power $\varphi=\psi^i$ satisfies the conditions of Proposition~\ref{proposition:main morphic}, then the fixed point $\psi^\omega(a)$ is SURD. More precisely, for all $\mathbf{m}\in\N^d$, the prefix of size $\mathbf{m}$ of $\psi^\omega(a)$ occurs infinitely often along any direction with gaps at most $s^{i\lceil \log( \max\mathbf{m})\rceil+i}$.
\end{corollary}

\begin{proof} 
Clearly, the fixed points of $\psi$ and $\varphi$ are the same. Now apply Proposition~\ref{proposition:main morphic} to $\varphi$.
\end{proof}

\begin{example} 
The morphism 
\[
\psi\colon 0\mapsto 
\begin{bmatrix}
0 & 0 & 0\\
1 & 1 & 1\\
0 & 1 & 0
\end{bmatrix}, \quad 
1\mapsto
\begin{bmatrix}
0 & 1 & 0\\
1 & 0 & 1\\
1 & 1 & 0 
\end{bmatrix}
\]
satisfies the hypotheses of Corollary~\ref{cor:morphism power} for $s=3$, $i=2$. Indeed, it can be checked that for each $C\in\mathcal{C} (9)$, we can find a 1 at the same position in $C$ in both images $\psi^2(0)$ and $\psi^2(1)$.
\end{example}

\begin{remark}
\label{rem:primitivity}
The hypotheses of Proposition~\ref{proposition:main morphic} should be compared to the primitivity property of a morphism. In the unidimensional case, a morphism is said to be primitive if its incidence matrix is primitive, or equivalently, if some power of the morphism is such that all letters appear in the image of every letter; see for example \cite{Durand--1998}. It is well known that fixed points of primitive morphisms are uniformly recurrent. This notion of primitivity generalizes naturally to any dimension $d$. However, if we are interested in studying the URD property, the natural generalization of primitivity is not accurate: we should not only consider the number of times a letter occurs in the image of another letter but also the positions where the letter occurs within each image. See Section~\ref{sec:perspectives} for some perspectives in this direction.
\end{remark}

Now we give a family of examples of SURD $d$-dimensional words which do not satisfy the hypotheses of Corollary~\ref{cor:morphism power}, showing that it does not give a necessary condition. We first need the following observation on unidimensional fixed points of morphisms.

\begin{lemma}\label{lem:001-101}
Let $\varphi$ be a unidimensional morphism of constant prime size $s$ and prolongable on $a\in A$ for which there exists $i\in [\![0,s-1]\!]$ such that $\varphi(b)_i=a$ for each $b\in A$. For all positive integers $m$, any factor of length $s$ of the infinite word $(\varphi^\omega(a)_{mk})_{k\in\N}$ contains the letter $a$.
\end{lemma}

\begin{proof}
Let $w=\varphi^\omega(a)$ and let $m$ be a positive integer. The integer $m$ can be decomposed in a unique way as $m=s^e\ell$ with $e,\ell\in\N$ and $\ell\not\equiv 0\pmod s$. We prove the result by induction on $e\in\N$. If $e=0$ then $m\not\equiv 0\pmod s$. Then for all $k\in\N$, at least one of the $s$ integers $m k$, $m (k+1),\ldots,m (k+s-1)$ is congruent to $i$ modulo $s$. Since the letter $a$ appears in the $i$-th place of the images of all letters, at least one of the letters $w_{m k}$, $w_{m (k+1)},\ldots,w_{m (k+s-1)}$ is equal to $a$.
Now suppose that $e>0$ and that the result is correct for $e-1$. Observe that, for every $k\in\N$, the preimage of the letter $w_{mk}=w_{s^e\ell k}$ is the letter $w_{\frac ms k}=w_{s^{e-1}\ell k}$. Since the morphism is prolongable on $a$ and since $m\equiv 0\pmod s$, for each $k\in\N$, the letter $w_{mk}$ is equal to $a$ if its preimage is $a$. But by induction hypothesis, for all $k\in\N$, at least one of the $s$ preimages $w_{\frac ms k}$, $w_{\frac ms (k+1)},\ldots,w_{\frac ms (k+s-1)}$ is equal to $a$. Therefore, we obtain that for all $k\in\N$, at least one of the $s$ letters $w_{m k}$, $w_{m (k+1)},\ldots,w_{m (k+s-1)}$ is equal to $a$ as well.
\end{proof}

\begin{proposition}
\label{prop:SURDcor-not-necessary}
If $\varphi$ is a $d$-dimensional square morphism of some prime size $s$ and prolongable on $a\in A$ such that
\begin{enumerate}
    \item \label{eq:1} $\forall i_2,\ldots,i_d\in[\![0,s-1]\!]$, $\varphi(a)_{0,i_2,\ldots,i_d}=a$ 
    \item \label{eq:2} $\exists i_1\in[\![0,s-1]\!]$, $\forall i_2,\ldots,i_d\in[\![0,s-1]\!]$, $\varphi(b)_{i_1,\ldots,i_d}=a$ for each $b\in A$
\end{enumerate}
then $\varphi^\omega(a)$ is SURD.
\end{proposition}

\begin{proof}
By Proposition~\ref{prop:reduction}, we only have to show that there exists a uniform bound $t$ such that the letter $a$ occurs infinitely often along any direction of $\varphi^\omega(a)$ with gaps bounded by $t$. It is sufficient to prove the result for the fixed point beginning with $1$ of the binary morphism $\psi$ satisfying the hypotheses $\eqref{eq:1}$ and $\eqref{eq:2}$ and having $0$ at any other coordinates in the images of both $0$ and $1$. Indeed, the fixed point $\varphi^\omega(a)$ of any morphism $\varphi$ satisfying $\eqref{eq:1}$ and $\eqref{eq:2}$  differs from this one only by replacing occurrences of $1$ by $a$ and occurrences of $0$ by any letter of the alphabet. For example, for $d=2$, the morphism $\psi$ is 
\[
	\psi\colon 
    0\mapsto 
    \begin{bmatrix}
    0 & 0 & \cdots & 0 & 1 & 0 & \cdots & 0 \\
    0 & 0 & \cdots & 0 & 1 & 0 & \cdots & 0 \\
    \vdots & \vdots &&&&&& \vdots\\ 
    0 & 0 & \cdots & 0 & 1 & 0 & \cdots & 0 
    \end{bmatrix}, \quad 
    1\mapsto
	\begin{bmatrix}
    1 & 0 & \cdots & 0 & 1 & 0 & \cdots & 0 \\
    1 & 0 & \cdots & 0 & 1 & 0 & \cdots & 0 \\
    \vdots & \vdots &&&&&& \vdots\\ 
    1 & 0 & \cdots & 0 & 1 & 0 & \cdots & 0 
	\end{bmatrix}
\]
(where the common columns of $1$'s are placed at position $i_1$ in both images). Each of the hyperplanes 
\[
	H_k=\{\psi^\omega(1)_{k,i_2\ldots,i_d}\colon i_2,\ldots,i_d\in\N\},\  \for k\in\N
\]
of $\psi^\omega(1)$ contains either only $0$'s or only $1$'s. Therefore, for any direction $\mathbf{q}=(q_1,\ldots,q_d)$, we have $\psi^\omega(1)_{\ell \mathbf{q}}=\psi^\omega(1)_{\ell q_1,0,\ldots,0}$, hence the unidimensional word $\N\to A,\ \ell\mapsto \psi^\omega(1)_{\ell \mathbf{q}}$ is the fixed point of the unidimensional morphism
\[
    \sigma\colon 0\mapsto \begin{bmatrix}
    0 & 0 & \cdots & 0 & 1 & 0 & \cdots & 0 
    \end{bmatrix}, \quad 
    1\mapsto
	\begin{bmatrix}
    1 & 0 & \cdots & 0 & 1 & 0 & \cdots & 0 \\
    \end{bmatrix}
\]
(where, again, the common $1$'s are placed at position $i_1$ in both images). By Lemma~\ref{lem:001-101}, we obtain that $\psi^\omega(1)$ is SURD with the uniform bound $t=s$.
\end{proof}

Note that the role of the first coordinate $i_1$ could be played by any of the other coordinates $i_2,\ldots,i_d$ with the ad hoc modifications in the statement of Proposition~\ref{prop:SURDcor-not-necessary}.

Now we give a sufficient condition for a $d$-dimensional word to be non URD.

\begin{proposition} \label{prop:suff-not-nec}
Let $\varphi$ be a $d$-dimensional square morphism of a prime size $s$ prolongable on $a\in A$. Let $\mathbf{q}$ be a direction and let $C=\langle \mathbf{q} \bmod{s} \rangle$. If $\varphi(b)_\mathbf{i}\neq a$ for each $b\in A$ and $\mathbf{i}\in C$ except for $\varphi(a)_\mathbf{0}=a$, then $(\varphi^\omega(a)_{\ell\mathbf{q}})_{i\in\N}\in a (A\setminus \{a\})^\omega$. In particular, $\varphi^\omega(a)$ is not recurrent along the direction $\mathbf{q}$.
\end{proposition}

\begin{proof} 
Suppose that the first occurrence of $a$ after that in position $\mathbf{0}$ along the direction $\mathbf{q}$ occurs in position $\ell \mathbf{q}$. Since, for each $b\in A$, $\varphi(b)$  has non-$a$ elements on all places defined by $C\setminus \{\mathbf{0}\}$, the letter $\varphi^\omega(a)_{\ell \mathbf{q}}$ must be placed at the coordinate $\mathbf{0}$ of the image of $a$. In particular, the preimage of $\varphi^\omega(a)_{\ell \mathbf{q}}$ must be $a$. Because $s$ is prime, $\ell$ must be divisible by $s$ and the preimage of $\varphi^\omega(a)_{\ell \mathbf{q}}$ is $\varphi^\omega(a)_{\frac{\ell}{s}\mathbf{q}}$. But by the choice of $\ell$ and since $0<\frac{\ell}{s}<\ell$, we must also have $\varphi^\omega(a)_{\frac{\ell}{s}\mathbf{q}}\neq a$, a contradiction.
\end{proof}

The next results shows that the condition of Proposition~\ref{prop:suff-not-nec} is not necessary.

\begin{proposition}
The fixed point $\varphi^{\omega}(1)$ of the morphism
\[
	\varphi:0\mapsto 
    \begin{bmatrix}
	1 & 1 & 0\\
	0 & 0 & 0\\
	0 & 0 & 1
    \end{bmatrix}
	\quad 1\mapsto 
	\begin{bmatrix}
	1 & 1 & 1\\
	0 & 1 & 0\\
	1 & 1 & 0 
    \end{bmatrix}
\]
is not recurrent along the direction $(1,3)$. 
\end{proposition}

\begin{proof}
We let $w=\varphi^{\omega}(1)$. We show that the sequence we get along the direction $(1,3)$ is $10^\omega$. It can be seen directly that the first symbols are 100, then we proceed by induction. Suppose the converse, and that $i$ is the smallest positive integer such that $w_{i,3i}=1$. We consider three cases: $i=3i'$, $i=3i'+1$, or $i=3i'+2$. In each case, our aim is to prove that $w_{i',3i'}=1$, contradicting the minimality of $i$.

Case 1: $i=3i'$. In this case $\varphi(w_{i',3i'})_{0,0}=w_{i,3i}$. Since $\varphi(0)_{0,0}=0$ and $w_{i,3i}=1$ by the assumption, we must have $w_{i',3i'}=1$.

Case 2: $i=3i'+1$. In this case $\varphi(w_{i',3i'+1})_{1,0}=w_{i,3i}$. Since $\varphi(0)_{1,0}=0$ and $w_{i,3i}=1$, we have  $w_{i',3i'+1}=1$. The coordinate $(i',3i'+1)$ being a position $(i' \bmod 3, 1)$ in some image $\varphi(a)$, this is possible only in the case when $a=1$ and $i'\equiv 1 \pmod 3$. Indeed, this is the only non-0 position with second coordinate 1 in $\varphi(0)$ and $\varphi(1)$. Therefore, we obtain $w_{i',3i'}=1$.

Case 3: $i=3i'+2$. In this case $\varphi(w_{i',3i'+2})_{2,0}=w_{i,3i}$. Since $\varphi(1)_{2,0}=0$ and $w_{i,3i}=1$, we have $w_{i',3i'+2}=0$. The coordinate $(i',3i'+2)$ being a position $(i' \bmod 3, 2)$ in some image $\varphi(a)$, we must have $a=0$ and $i'\equiv 2 \pmod 3$. Indeed, this is the only non-1 position with second coordinate 2 in $\varphi(0)$ and $\varphi(1)$. We obtain once again that $w_{i',3i'}=1$.
\end{proof}

The next theorem gives a characterization of SURD fixed points of square binary morphisms of size $2$.

\begin{theorem}\label{thm:characterization}
Let $\varphi$ be a bidimensional binary square morphism of size $2$ prolongable on $1$. The fixed point $\varphi^\omega(1)$ is SURD if and only if either $\varphi(0)_{0,0}=1$ or $\varphi(1)=\left[\begin{smallmatrix}
    1 & 1 \\
    1 & 1
    \end{smallmatrix}\right]$.
\end{theorem}

The ``if'' part follows from Corollary~\ref{cor:1}. The ``only if'' part is proven with a rather technical argument involving a case study analysis and using certain properties of arithmetic progressions in the Thue-Morse word $\mathbf{t}=0110100110010110\cdots$. 

We first provide two useful lemmas about the Thue-Morse word \cite{Thue--1906}. Recall that this word is the fixed point of the unidimensional morphism $0\mapsto 01,\ 1\mapsto 10$. It can also be defined thanks to the function $s_2\colon\N\to \N$ that returns the sum $s_2(n)$ of the digits in the binary expansion of $n$: the $(n+1)$-th letter $t_n$ of the Thue-Morse word $\mathbf{t}$ is equal to $0$ if $s_2(n)\equiv 0\pmod 2$ and to $1$ otherwise.

\begin{lemma}\label{lemma:TM1} 
For any $\ell\in\N$, the Thue--Morse word $\mathbf{t}=(t_n)_{n\in \N}$ satisfies $t_0=0$ and $t_d=t_{2d}=t_{3d}=\ldots=t_{2^\ell d}$ with $d=2^{\ell}-1$. Moreover, $t_d=1$ if $\ell$ is odd and $t_d=0$ if $\ell$ is even.
\end{lemma}

\begin{proof}
Let $m\in[\![1,2^\ell]\!]$. There exist $r$ odd and $i\geq 0$ such that $m = r2^i$.  Denote by $r_j r_{j-1}\cdots r_1 r_0$ the binary expansion $(r)_2$ of $r$. In particular, $r_0=1$ since $r$ is odd. Also $r\leq m\leq 2^\ell$ and $r$ odd imply that $r<2^\ell$ and $|(r)_2|\leq \ell$. We have $(r2^\ell)_2=r_j r_{j-1}\cdots r_1 r_0 0^\ell$ and $(r2^\ell-1)_2=r_j r_{j-1}\cdots r_1 0 1^\ell$. Therefore, 
\begin{align*}
    s_2(r(2^\ell-1)) &= s_2(r2^\ell -1-(r-1)) \\
    &= s_2(r)-1+\ell-s_2(r-1) \tag{as $|(r)_2|\leq \ell$}\\
    &= s_2(r)-1+\ell-s_2(r)+1 \tag{as $r$ odd}\\
    &= \ell.
\end{align*}    
Since $s_2(m(2^\ell-1))=s_2(r(2^\ell-1))=\ell$, the conclusion follows.
\end{proof}

Note that the proof of the previous lemma is a modification of  Lemma 3.2 in \cite{Bucci--2013}.

\begin{lemma} \label{lemma:TM0} 
For any positive integer $\ell$, the Thue--Morse word $\mathbf{t}=(t_n)_{n\in \N}$ satisfies  $t_{0}=t_{d}=t_{2d}=\ldots=t_{2^\ell d}=0$ with $d=2^\ell+1$.
\end{lemma}

\begin{proof}
Let $m\in[\![0,2^\ell-1]\!]$. Since $m<2^\ell$, $|(m)_2|\leq \ell $. So $s_2(m(2^\ell+1))=s_2(m)+s_2(m)$ is even. It follows that $t_{m(2^\ell+1)}=0$. For $m=2^\ell$, we have $s_2(m(2^\ell+1))=s_2(2^{2\ell}+2^\ell)=2$ and $t_{m(2^\ell+1)}=0$.
\end{proof}

\begin{proof}[Proof of Theorem~\ref{thm:characterization}] 
The condition is sufficient by Corollary~\ref{cor:1}. To prove that it is necessary, we show by  a case study that the fixed points beginning with $1$ of all the other possible morphisms are not SURD.  For the sake of clarity, we set $w=\varphi^\omega(1)$.

First, note that $\varphi(1)_{i,j}=\varphi(0)_{i,j}=0$ for some $(i,j)\ne (0,0)$ implies that  $w$ contains $10^\omega$ along the direction $(i,j)$. Hence for a given position $(i,j)$, it is sufficient to consider $(\varphi(1)_{i,j},\varphi(0)_{i,j})\in\{(0,1),(1,0),(1,1)\}$. The graph of our case study is depicted in Figure~\ref{fig:2x2}.
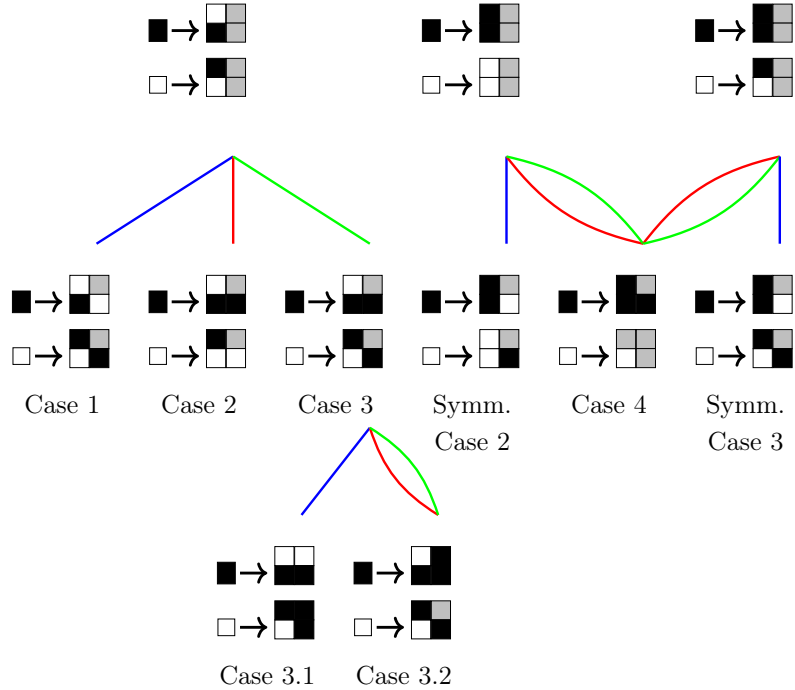
\begin{figure}[htb]
\begin{center}
\scalebox{0.9}{
\begin{tikzpicture}[scale=1]

\clip (9.8,8.4) rectangle (-1.8,-1.7);

\tikzstyle{01}=[draw=blue, line width=1pt]
\tikzstyle{10}=[draw=red, line width=1pt]
\tikzstyle{11}=[draw=green, line width=1pt]

\tikzstyle{every node}=[shape=rectangle,fill=none,draw=none,minimum size=2.7cm,inner sep=2pt]
\tikzstyle{every path}=[draw=blue]

\begin{scope}[shift={(0,8)}]
\morph{1}{0}{.}{.}{0}{1}{.}{.};
\node(a1) at (1.5,-0.5){};
\end{scope}

\begin{scope}[shift={(-2,4)}]
\morph{1}{0}{0}{?}{0}{1}{1}{?};
\node(b1) at (1.5,-0.5){};
\end{scope}
\draw[01] (a1.south)--(b1.north);

\begin{scope}[shift={(0,4)}]
\morph{1}{0}{1}{?}{0}{1}{0}{?};
\node(c1) at (1.5,-0.5){};
\end{scope}
\draw[10] (a1.south)--(c1.north);

\begin{scope}[shift={(2,4)}]
\morph{1}{0}{1}{?}{0}{1}{1}{?};
\node(d1) at (1.5,-0.5){};
\end{scope}
\draw[11] (a1.south)--(d1.north);

\begin{scope}[shift={(1,0)}]
\morph{1}{0}{1}{0}{0}{1}{1}{1};
\node(k1) at (1.5,-0.5){};
\end{scope}
\draw[01] (d1.south)--(k1.north);

\begin{scope}[shift={(3,0)}]
\morph{1}{0}{1}{1}{0}{1}{1}{*};
\node(l1) at (1.5,-0.5){};
\end{scope}
\draw[10] (d1.south) to [bend right=20](l1.north);
\draw[11] (d1.south) to [bend left=20](l1.north);

\tikzstyle{every node}=[shape=rectangle,fill=none,draw=none,minimum size=0.2cm,inner sep=2pt]

\begin{scope}[shift={(1,-1.5)}]
\node at (-2,4){Case 1};
\node at (0,4){Case 2};
\node at (2,4){Case 3};
\node at (1,0){Case 3.1};
\node at (3,0){Case 3.2};
\node at (4,3.95){Symm.};
\node at (4,3.45){Case 2};
\node at (6,4){Case 4};
\node at (8,3.95){Symm.};
\node at (8,3.45){Case 3};
\end{scope}

\tikzstyle{every node}=[shape=rectangle,fill=none,draw=none,minimum size=2.7cm,inner sep=2pt]
\tikzstyle{every path}=[draw=blue]

\begin{scope}[shift={(4,8)}]
\morph{1}{1}{?}{.}{0}{0}{?}{.};
\node(a2) at (1.5,-0.5){};
\end{scope}

\begin{scope}[shift={(8,8)}]
\morph{1}{1}{?}{.}{0}{1}{?}{.};
\node(a3) at (1.5,-0.5){};
\end{scope}

\begin{scope}[shift={(4,4)}]
\morph{1}{1}{0}{?}{0}{0}{1}{?};
\node(b) at (1.5,-0.5){};
\end{scope}
\draw[01] (a2.south)--(b.north);

\begin{scope}[shift={(6,4)}]
\morph{1}{1}{1}{*}{0}{*}{*}{*};
\node(c) at (1.5,-0.5){};
\end{scope}
\draw[10] (a2.south) to [bend right=20] (c.north);
\draw[11] (a2.south) to [bend left=20](c.north);

\draw[10] (a3.south) to [bend right=20] (c.north);
\draw[11] (a3.south) to [bend left=20](c.north);

\begin{scope}[shift={(8,4)}]
\morph{1}{1}{0}{*}{0}{1}{1}{*};
\node(d) at (1.5,-0.5){};
\end{scope}
\draw[01] (a3.south)--(d.north);
\end{tikzpicture}
}
\end{center}
\caption{Square morphisms of size 2. A  black (resp.\ white) cell corresponds to a position filled with letter $1$ (resp.\ $0$). A gray cell corresponds to a position which can contain any letter. The possible pairs $(\varphi(1)_{i,j},\varphi(0)_{i,j})$ with $i,j\in\{0,1\}$ are successively considered. A blue line corresponds to the pair $(\varphi(1)_{i,j},\varphi(0)_{i,j})=(0,1)$, a red one to $(1,0)$ and a green one to $(1,1)$.}
\label{fig:2x2}
\end{figure}
\medskip

{\bf Case 1}
\[
\varphi\colon 0\mapsto 
\begin{bmatrix} 
 	1 & *\\ 0 & 1 
\end{bmatrix}, \quad 
1\mapsto 
\begin{bmatrix}
	0 & *\\ 1 & 0 
\end{bmatrix}
\]

We show that the factor $10^{2^\ell-1}$ occurs along the direction $(p,q) =(2^{2\ell} (2^\ell-1),2^\ell +1)$ with $\ell$ odd (see Figure~\ref{fig:case1}). First note that the first row of $\varphi^\omega(1)$ is equal to $\bar{\mathbf{t}}$. Hence, the first $2^{2\ell}$ rows contain $\varphi^{2\ell}(\bar{\mathbf{t}})$. Let $d=2^{\ell}-1$. By Lemma~\ref{lemma:TM1}, the arithmetical subsequence $(\bar{t}_{md})_{m\in \N}$ begins with $10^{2^\ell}$. Thus, $w_{mp,0}=w_{md2^{2\ell},0}=w_{md,0}=0$ for $m\in\{1,\ldots,2^\ell\}$. To conclude, observe that the first column of  $\varphi^{2\ell}(0)$ is a prefix of $\mathbf{t}$. By Lemma~\ref{lemma:TM0}, the arithmetical subsequence $(t_{mq})_{m\in \N}$ begins with $0^{2^\ell}$. Let $m\in[\![1,2^\ell-1]\!]$. As $(2^\ell -1)q= 2^{2\ell}-1 < 2^{2 \ell}$, the letter $w_{mp,mq}$ is inside a square $\varphi^{2\ell}(0)$ with the bottom left corner at position $(mp,0)$, hence $w_{mp,mq}=0$.
\begin{figure}[ht!]
\begin{center}
\begin{tikzpicture}[scale=0.4]
\tikzstyle{every node}=[shape=rectangle,fill=none,draw=none,minimum size=0cm,inner sep=2pt]
\tikzstyle{every path}=[draw=black,line width = 0.5pt]

\draw (0,0) rectangle (6.4,6.4);
\node at (3.2,3.2){\large $\varphi^{2\ell}(1)$};
\node at (0.2,0.2){\tiny 1};
\node at (0.2,0.5){\tiny 0};		
\node at (0.2,0.8){\tiny 0};	
\node at (0.2,1.1){\tiny 1};
\draw[dotted] (0.2,1.3)--(0.2,1.8);
\node at (0.2,2){\tiny 1};
\draw[dotted] (0.2,2.7)--(0.2,3.1);
\node at (0.2,3.8){\tiny 1};
\draw[dotted] (0.2,4.5)--(0.2,4.9);
\node at (0.2,5.6){\tiny 1};
\draw[dotted] (0.2,5.9)--(0.2,6.3);

\foreach \x in {10,20,30}
{
	\draw (\x,0) rectangle (6.4+\x,6.4);
	\node at (\x+3.2,3.2){\large $\varphi^{2\ell}(0)$};
	\node at (\x+0.2,0.2){\tiny 0};
	\node at (\x+0.2,0.5){\tiny 1};		
	\node at (\x+0.2,0.8){\tiny 1};	
	\node at (\x+0.2,1.1){\tiny 0};
	\draw[dotted] (\x+0.2,1.3)--(\x+0.2,1.8);
	\node at (\x+0.2,2){\tiny 0};
	\draw[dotted] (\x+0.2,2.7)--(\x+0.2,3.1);
	\node at (\x+0.2,3.8){\tiny 0};
	\draw[dotted] (\x+0.2,4.5)--(\x+0.2,4.9);
	\node at (\x+0.2,5.6){\tiny 0};
	\draw[dotted] (\x+0.2,5.9)--(\x+0.2,6.3);
}
\foreach \x in {0,10,20}
{	
	\node at (\x+1.5+6.4,3){$\ldots$};
} 

\tikzstyle{every path}=[draw=red,line width = 1pt, ->]
\draw (0.4,0.1) to [bend left = 5] node [above] {\color{red}$p$} (10,0.1);
\draw (10+0.4,2) to [bend left = 5] node [above] {\color{red}$p$} (10+10,2);
\draw (20+0.4,3.9) to [bend left = 5] node [above] {\color{red}$p$} (20+10,3.9);
	\draw (10+0.5,0.2+0.1) to [bend right = 20] node [right] {\color{red}$q$} (10+0.5,0.2+1.7);
	\draw (20+0.5,2.0+0.1) to [bend right = 20] node [right] {\color{red}$q$} (20+0.5,2.0+1.7);
	\draw (30+0.5,3.8+0.1) to [bend right = 20] node [right] {\color{red}$q$} (30+0.5,3.8+1.7);
\end{tikzpicture}
\end{center}
\caption{Structure of Case 1 morphisms with $(p,q) =(2^{2\ell} (2^\ell-1),2^\ell +1)$ where $\ell$ is odd.}\label{fig:case1}
\end{figure}
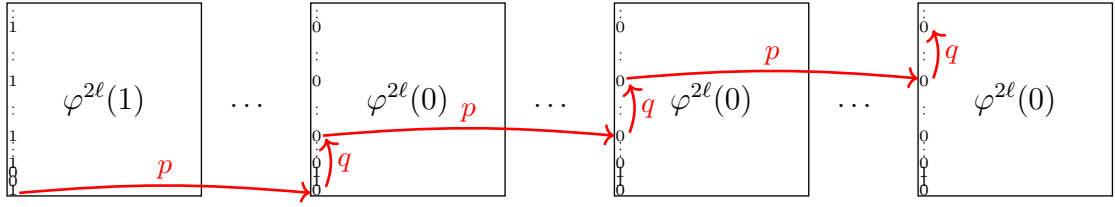
\medskip

{\bf Case 2}
\[
\varphi\colon 0\mapsto 
\begin{bmatrix} 
 1 & *\\ 0 & 0 
\end{bmatrix}, \quad 
1\mapsto 
\begin{bmatrix}
	0 & *\\ 1 & 1
\end{bmatrix}
\]

In this case we will prove that for all odd $n\in\N$, the factor $0^{2^n-1}$ occurs along the direction $(1, (2^n-1)2^n)$. More precisely, we claim that for all odd $n\in\N$ and all $i\in[\![1,2^n-1]\!]$, we have $w_{i,i(2^n-1)2^n}=0$. First, notice that there are only $0$'s on the bottom line of the images $\varphi^m(0)$ for all $m\in\N$, namely, $\varphi^m(0)_{i,0}=0$ for all $m\in\N$ and all $i\in[\![0, 2^m-1]\!]$. Second, we use Lemma~\ref{lemma:TM1} which gives $w_{0,i(2^n-1)}=0$ for all odd $n\in\N$ and all $i\in[\![1,2^n]\!]$. By applying the power morphism $\varphi^n$, we get $w_{0,i(2^n-1)2^n}=0$ for all odd $n\in\N$ and all $i\in[\![1, 2^n]\!]$. Since the latter points belong to left bottom corner of $\varphi^n(0)$, we obtain that $w_{i,i(2^n-1)2^n}=0$ for every $i\in[\![1, 2^n-1]\!]$ as desired.
\medskip

{\bf Case 3.1}
\[
\varphi\colon 0\mapsto 
\begin{bmatrix} 
 	1 & 1\\ 0 & * 
\end{bmatrix}, \quad 
1\mapsto 
\begin{bmatrix}
	0 & 0\\ 1 & * 
\end{bmatrix}
\]

Similarly to Case 1, we can show that the factor $10^{2^\ell-1}$ occurs along the direction $(p,q) =(2^\ell +1,2^{2\ell} (2^\ell-1)+2^\ell +1)$ with $\ell$ odd. Indeed, in this case, the Thue-Morse word or its complement appears in the first column and in the diagonal; see Figure~\ref{fig:case3-1}.
\begin{figure}[htbp]
\begin{center}
\begin{tikzpicture}[scale=0.4]

\tikzstyle{every node}=[shape=rectangle,fill=none,draw=none,minimum size=0cm,inner sep=2pt]
\tikzstyle{every path}=[draw=black,line width = 0.5pt]

\draw (0,0) rectangle (6.4,6.4);
\node at (5,1.3){\large $\varphi^{2\ell}(1)$};
\node at (0.2,0.2){\tiny 1};
\node at (0.5,0.5){\tiny 0};		
\node at (0.8,0.8){\tiny 0};	
\node at (1.1,1.1){\tiny 1};
\draw[dotted] (1.3,1.3)--(1.8,1.8);
\node at (2,2){\tiny 1};
\draw[dotted] (2.7,2.7)--(3.1,3.1);
\node at (3.8,3.8){\tiny 1};
\draw[dotted] (4.5,4.5)--(4.9,4.9);
\node at (5.6,5.6){\tiny 1};
\draw[dotted] (5.9,5.9)--(6.3,6.3);

\foreach \y in {10,20,30}
{
	\draw (0,\y) rectangle (6.4,\y+6.4);
	\node at (5,1.3+\y){\large $\varphi^{2\ell}(0)$};
	\node at (0.2,0.2+\y){\tiny 0};
	\node at (0.5,0.5+\y){\tiny 1};		
	\node at (0.8,0.8+\y){\tiny 1};	
	\node at (1.1,1.1+\y){\tiny 0};
	\draw[dotted] (1.3,1.3+\y)--(1.8,1.8+\y);
	\node at (2,2+\y){\tiny 0};
	\draw[dotted] (2.7,2.7+\y)--(3.1,3.1+\y);
	\node at (3.8,3.8+\y){\tiny 0};
	\draw[dotted] (4.5,4.5+\y)--(4.9,\y+4.9);
	\node at (5.6,\y+5.6){\tiny 0};
	\draw[dotted] (5.9,5.9+\y)--(6.3,\y+6.3);
}
\foreach \y in {0,10,20}
{	
	\node at (3,1.5+6.4+\y){$\vdots$};
} 

\tikzstyle{every path}=[draw=red,line width = 1pt, ->]
\draw (-0.1,0.3) to [bend left = 5] node [left] {\color{red}$q$} (-0.1,12);
\draw (1.8,12.2) to [bend left = 5] node [left] {\color{red}$q$} (1.8,23.8);
\draw (3.6,24) to [bend left = 5] node [left] {\color{red}$q$} (3.6,35.6);
\draw (0.2,12) to [bend left = 20] node [above] {\color{red}$p$} (0.2+1.6,12);
\draw (2,23.8) to [bend left = 20] node [above] {\color{red}$p$} (2.0+1.6,23.8);
\draw (3.8,35.6) to [bend left = 20] node [above] {\color{red}$p$} (3.8+1.6,35.6);
\end{tikzpicture}
\end{center}
\caption{Structure of Case 3.1 morphisms with  $(p,q) =(2^\ell +1,2^{2\ell} (2^\ell-1)+2^\ell +1)$ where $\ell$ is odd.}
\label{fig:case3-1}
\end{figure}
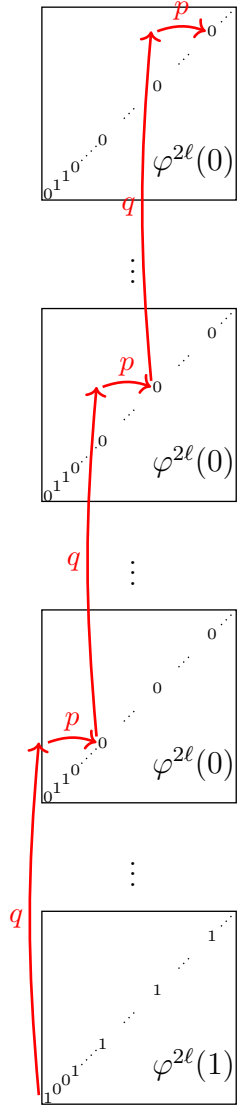
\medskip

{\bf Case 3.2}
\[
\varphi\colon 0\mapsto 
\begin{bmatrix} 
 1 & *\\ 0 & 1 
\end{bmatrix}, \quad 
1\mapsto 
\begin{bmatrix}
	0 & 1\\ 1 & 1
\end{bmatrix}
\]

In this case we will prove that the word is not recurrent in direction $(2,1)$. More precisely, we show that $(w_{2i,i})_{i\in\N}=10^\omega$. Clearly $w_{0,0}=1$. We prove $w_{2i,i}=0$ for all $i\ge 1$ by induction on $i$. The base case $w_{2,1}=0$ is easily verified. Now let $i>1$ and suppose that $w_{2i',i'}=0$ for all $1\le i'<i$. If $i$ is even, then $w_{2i,i}=\varphi(w_{i,\frac i2})_{0,0}=0$, where the last equality comes from the induction hypothesis with $i'=\frac i2$ and the fact that  $\varphi(0)_{0,0}=0$. If $i$ is odd, then $w_{2i,i}=\varphi(w_{i,\frac{i-1}{2}})_{0,1}$. Remark that $w_{i,\frac{i-1}{2}}$ is an element in a right column of a $2\times 2$ block which is an image of $0$ or $1$. An element $w_{i-1,\frac{i-1}{2}}$ (which is equal to $0$ by induction hypothesis with $i'=\frac{i-1}{2}$) is an element in the same block which is situated to the left of $w_{i,\frac{i-1}{2}}$. Due to the forms of $\varphi(0)$ and $\varphi(1)$, if a left element is 0, then the right element in the same line is $1$. So, $w_{i,\frac{i-1}{2}}=1$, hence $w_{2i,i}=\varphi(1)_{0,1}=0$.
\medskip

{\bf Case 4}

We can suppose that 
\[
\varphi\colon 
0\mapsto 
\begin{bmatrix} 
	* & *\\ 0 & * 
\end{bmatrix}, \quad 
1\mapsto 
\begin{bmatrix} 
 	1 & 0\\ 1 & 1 
\end{bmatrix}
\]
for otherwise $\varphi(1)=\left[\begin{smallmatrix}
    1 & 1 \\
    1 & 1
    \end{smallmatrix}\right]$.
\medskip

In this case we will prove that for all $n\in\N$, the factor $0^{2^n-1}$ occurs along the direction $(2^n-1,1)$. More precisely, for all $n\in\N$, we have $w_{j(2^n-1),j}=0$ for every $j\in[\![1, 2^n-1]\!]$. First, an easy induction on $n$ shows that there are $0$ just above the diagonal from upper left to lower right in the images $\varphi^n(1)$ for all $n\ge 1$, namely $\varphi^n(1)_{2^n-j,j}=0$ for all $n\ge 1$ and all $j\in[\![1,2^n-1]\!]$. For example, for $n=3$, we have
\[
\varphi^3(1)
=\varphi^2
\begin{bmatrix}
	1 & 0\\ 1 & 1 
\end{bmatrix}
=\begin{bmatrix}[cccc|cccc]
	1	& \mathbf{0}	& *	& *	& *	& * & * & * \\
    1	& 1	& \mathbf{0}	& *	& *	& * & * & * \\
    1	& 0	& 1 & \mathbf{0} & * & * & * & * \\ 
    1   & 1	& 1	& 1 & \mathbf{0} & * & * & * \\ 
	\hline
    1   & \mathbf{0}	& * & * & 1 & \mathbf{0} & * & * \\ 
    1   & 1	& \mathbf{0} & * & 1 & 1 & \mathbf{0} & * \\ 
    1   & \mathbf{0}	& 1 & \mathbf{0}	& 1	& 0 & 1 & \mathbf{0} \\ 
    1   & 1	& 1	& 1	& 1	& 1	& 1	& 1
\end{bmatrix}.
\]
Second, since $w_{i,0}=1$ for all $i\in\N$, we obtain that, for all $n,k\in\N$, the square factor of size $2^n$ occurring at position $(2^n k,0)$ is equal to $\varphi^n(1)$. Therefore, we have $w_{2^n k + 2^n-j,j}=0$ for all $n,k\in\N$ and $j\in\{1,\ldots,2^n-1\}$. The claim follows by considering the latter equality with $k=j-1$.
\end{proof}

The previous theorem gives a characterization of strong uniform recurrence along all directions for fixed points of bidimensional square binary morphisms of size $2$. For larger sizes of the morphism, we gave several conditions that are either necessary (Proposition~\ref{prop:suff-not-nec}) or sufficient (Propositions~\ref{proposition:main morphic} and~\ref{prop:SURDcor-not-necessary}). An open problem is to find a necessary and sufficient condition in general (see Section~\ref{sec:perspectives}).

We end this section by a small discussion on the SSURDO notion. First we provide an example of SSURDO aperiodic word. Then, we give an example of a SURD word that is not SSURDO.

\begin{proposition}\label{prop:SSURDO}
Let $\varphi$ be the square binary morphism defined by 
\[	
\varphi\colon 
    0\mapsto 
    \begin{bmatrix}
    1 & 0  &0 \\
    1 & 0 & 1 \\
    1 & 0 & 0
    \end{bmatrix}, \quad 
    1\mapsto
	\begin{bmatrix}
    1 & 0  &1 \\
    1 & 0 & 1 \\
    1 & 0 & 0
    \end{bmatrix}.
\]
The fixed point $\varphi^\omega(1)$ is SSURDO. 
\end{proposition}

\begin{proof}
Let $w=\varphi^\omega(1)$ and $\mathbf{q}$ be any direction. By definition of $\varphi$, for every position $\mathbf{p}\not\equiv (2,2) \pmod{3}$, we have $w(\mathbf{p})=w(\mathbf{p}\bmod{3})$. If follows that $w(\mathbf{p})= w(\mathbf{p}+3\mathbf{q})$ for all such $\mathbf{p}$. Consider now a position $\mathbf{p}\equiv (2,2)\pmod{3}$. Since $\gcd(\mathbf{q})=1$, we have $\mathbf{p}+\mathbf{q}\not \equiv (2,2)\pmod{3}$ and $\mathbf{p}+2\mathbf{q}\not \equiv (2,2)\pmod{3}$. So $w(\mathbf{p}+\mathbf{q})=w((\mathbf{p}+\mathbf{q})\bmod{3})$ and $w(\mathbf{p}+2\mathbf{q})=w((\mathbf{p}+2\mathbf{q})\bmod{3})$. By checking all the possible values modulo $3$ of $\mathbf{p}+\mathbf{q}$ and $\mathbf{p}+2\mathbf{q}$, we can verify that $w(\mathbf{p}+\mathbf{q})\neq w(\mathbf{p}+2\mathbf{q})$. So the letter $w(\mathbf{p})$ along the direction $\mathbf{q}$ is firstly repeated within a distance two, then it is repeated every three letters.

Now consider a position $\mathbf{p}$ and a factor $f$ of size $\mathbf{s}$ occurring at position $\mathbf{p}$. Let $i=\lceil \max (\log_3 \mathbf{s})\rceil$. We will show that $f$ occurs along $\mathbf{q}$ with gaps bounded by $3^{i+1}$. To do so, we will consider a covering of the grid by the square factors $\varphi^i(0)$ and $\varphi^i(1)$ and study the  position of the factor $f$ relatively to this covering; see Figure~\ref{fig:SSURDO}.
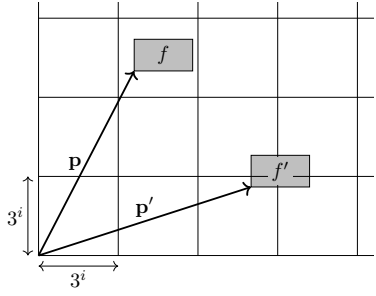
\begin{figure}[htbp]
\centering
\scalebox{0.7}{
\begin{tikzpicture}[scale=1]
	\clip(-0.6,-0.6) rectangle (6.3,4.8);
	\tikzstyle{every node}=[shape=rectangle,fill=none,draw=none,minimum size=0cm,inner sep=2pt]
	\tikzstyle{every path}=[draw=black,line width = 0.5pt]
	\draw[fill=gray!50] (1.8,3.5) rectangle (2.9,4.1);	
	\node at (2.35,3.8){$f$};
	\draw[fill=gray!50] (4,1.3) rectangle (5.1,1.9);	
	\foreach \x in {0,0.5,1,1.5,2}
{
	\draw (3*\x,0) to (3*\x,6.5);
	\draw (0,3*\x) to (7,3*\x);
}	
	\node[fill=gray!50] at (4.55,1.6){$f'$};
	\draw[<->] (0,-0.2)  to  node [below] {$3^i$}  (1.5,-0.2);
	\draw[<->] (-0.2,0)  to  node [left] {$3^i$}  (-0.2,1.5);
	
	\tikzstyle{every path}=[draw=black,line width = 1pt, ->]
	\draw (0,0)  to  node [left] {$\mathbf{p}$}  (1.8,3.5);
	\draw (0,0)  to  node [above] {$\mathbf{p'}$}  (4,1.3);
\end{tikzpicture}
}
\caption{The factor $f$ at position $\mathbf{p}$ occurs ``completely'' inside a factor of the form $\varphi^i(0)$ or $\varphi^i(1)$, while the factor $f'$ (of the same size) at position $\mathbf{p'}$ does not.}
\label{fig:SSURDO}
\end{figure}

If $f$ occurs ``completely'' inside a factor $\varphi^i(0)$ or $\varphi^i(1)$, i.e.\ if
\[
	\exists \mathbf{k}\in\N^2,\quad 
    3^i\mathbf{k} \le \mathbf{p}\le \mathbf{p}+\mathbf{s} < 3^i(\mathbf{k}+\mathbf{1}),
\]
then we use the previous observation about the occurrence of any letter every three positions along $\mathbf{q}$ to conclude that $f$ occurs infinitely often along $\mathbf{q}$ with gaps bounded by $3^{i+1}$.

Now, suppose that $f$ does not ``completely'' occur inside a factor $\varphi^i(0)$ or $\varphi^i(1)$, i.e.\ that
\[
	\exists \mathbf{k}\in\N^2,\quad 
    3^i\mathbf{k} \le \mathbf{p} < 3^i(\mathbf{k}+\mathbf{1})
\]
but there exists $j\in\{1,2\}$ such that 
\[
    p_j+s_j\ge 3^i(k_j+1)
\]
where $\mathbf{p}=(p_1,p_2)$,  $\mathbf{s}=(s_1,s_2)$ and  $\mathbf{k}=(k_1,k_2)$.
Then $\mathbf{p}+\mathbf{s} < 3^i(\mathbf{k}+\mathbf{2})$ by definition of $i$. Consider the factor $z\colon [\![0,3^i-1]\!]\times[\![0,3^i-1]\!]\to A$ of size $(3^i,3^i)$ at position $3^i\mathbf{k}$: for all $\mathbf{i}\in[\![0,3^i-1]\!]\times[\![0,3^i-1]\!]$, $z(\mathbf{i})=w(\mathbf{i}+3^i\mathbf{k})$.
This factor $z$ corresponds exactly to a square factor of the grid, that is either $\varphi^i(0)$ or $\varphi^i(1)$. Hence it occurs along $\mathbf{q}$ from the position $3^i\mathbf{k}$ infinitely many times with gaps bounded by $3^{i+1}$. Now, an easy recurrence shows that $\varphi^j(0)$ and $\varphi^j(1)$ coincide everywhere except in position $(3^j-1,3^j-1)$ for any $j\in\N$. It follows that any factor of size $(3^i, 3^i)$ occurring at a position of the form $3^{i}(x,y)$ extends in a unique way to a factor of size $(2\cdot 3^{i}-1, 2\cdot 3^i-1)$ occurring at the same position. Applying this to the factor $z$, we deduce that distances between consecutive occurrences of $f$ along $\mathbf{q}$ from the position $\mathbf{p}$ coincide with the distances between consecutive occurrences of $z$ along $\mathbf{q}$ from the position $3^i\mathbf{k}$. Hence the conclusion.
\end{proof}

Thanks to Theorem~\ref{thm:characterization}, we are able to show that SURD does not imply SSURDO, as illustrated by the following example.

\begin{example}
\label{ex:SURD_isnot_SSURDO}
SURD and SSURDO properties define two distinct classes of words. Consider the fixed point $w$ of the square binary morphism $\varphi$ defined by 
\[
	\varphi\colon 
    0\mapsto 
    \begin{bmatrix}
    1 & 1 \\
    1 & 1  
    \end{bmatrix}, \quad 
    1\mapsto
	\begin{bmatrix}
    0 & 0 \\
    1 & 0 
    \end{bmatrix},
\]
which is SURD by Theorem~\ref{thm:characterization}. We can show that for the size $\mathbf{s}=(1,1)$, the direction $\mathbf{q}=(1,0)$ and the translations $\mathbf{p}_n=(2^{n+1}-1,2^n-1)$ with $n\in\N$, the words $\dir{(w^{(\mathbf{p}_n)})}{q}{s}$ begins with $\bar{a}a^{3\cdot 2^n}\bar{a}$ where $a\in\{0,1\}$, 
by observing that $\mathbf{p}_{n+1}=2\mathbf{p}_n+(1,1)$. This is illustrated in Figure~\ref{fig:SURD_isnot_SSURDO}.
\begin{figure}[htbp]
\centering
\scalebox{0.8}{
\begin{tikzpicture}[overlay,remember picture]
	\tikzstyle{every path}=[draw=white,line width = 1pt]
	\coordinate (a) at ( $ (pic cs:bloc0a)  - (0.1,0.08)  $);
	\coordinate (b) at ( $ (pic cs:bloc0b)  + (0.2,0.28) $);
	\draw[fill=gray!50] (a) rectangle (b);
	\coordinate (a) at ( $ (pic cs:bloc1a)  - (0.1,0.08)  $);
	\coordinate (b) at ( $ (pic cs:bloc1b)  + (0.2,0.28) $);
	\draw[fill=gray!50] (a) rectangle (b);
	\coordinate (a) at ( $ (pic cs:bloc2a)  - (0.1,0.08)  $);
	\coordinate (b) at ( $ (pic cs:bloc2b)  + (0.2,0.28) $);
	\draw[fill=gray!50] (a) rectangle (b);
	\coordinate (a) at ( $ (pic cs:bloc3a)  - (0.1,0.08)  $);
	\coordinate (b) at ( $ (pic cs:bloc3b)  + (0.2,0.28) $);
	\draw[fill=gray!50] (a) rectangle (b);
\end{tikzpicture}
$\begin{array}{c|ccccccccccccccccccccccccccc}
	7
	&0&0
	&0&0
	&0&0&0&0
	&1&1&1&1&1&1&1&\tikzmark{bloc3a}1\tikzmark{D2b}
	&0&0&0&0&0&0&0&0&0&0&0\tikzmark{bloc3b}\\
	6
	&1&0
	&1&0
	&1&0&1&0
	&1&1&1&1&1&1&\tikzmark{D2a}1&1
	&1&0&1&0&1&0&1&0&1&0&1\\
	5
	&0&0
	&0&0
	&0&0&0&0
	&0&0&1&1&0&0&1&1
	&0&0&0&0&0&0&0&0&0&0&0\\
	4
	&1&0
	&1&0
	&1&0&1&0
	&1&0&1&1&1&0&1&1
	&1&0&1&0&1&0&1&0&1&0&1\\
	3
	&1&1
	&1&1
	&0&0&0&\tikzmark{bloc2a}0\tikzmark{D1b}\tikzmark{O2}
	&1&1&1&1&1&1&1&1
	&1&1&1&1&0\tikzmark{bloc2b}&0&0&0&1&1&1\\
	2
	&1&1
	&1&1
	&1&0&\tikzmark{D1a}1&0
	&1&1&1&1&1&1&1&1
	&1&1&1&1&1&0&1&0&1&1&1 \\
	1
	&0&0
	&1&\tikzmark{bloc1a}1\tikzmark{D0b}\tikzmark{O1}
	&0&0&0&0
	&0&0&1\tikzmark{bloc1b}&1&0&0&1&1
	&0&0&1&1&0&0&0&0&0&0&1\\
	0
	&1&\tikzmark{bloc0a}0\tikzmark{O0}
	&\tikzmark{D0a}1&1
	&1&0\tikzmark{bloc0b}&1&0
	&1&0&1&1&1&0&1&1
	&1&0&1&1&1&0&1&0&1&0&1\\\hline
	&0&1&2&  3&4&5& 6&7&8&9&10&11&12&13&14&15&16&17&18&19&20&21&22&23&24&25&26
	\end{array}$
\begin{tikzpicture}[overlay,remember picture]
\tikzstyle{every path}=[draw=red,line width = 1pt]
\coordinate (x) at ( $ (pic cs:O0)  + (0.12,0.3)$);
\coordinate (y) at ( $ (pic cs:O0)  - (0.12,0.15)$);
\draw (x.south) --++(0,-0.4) --++(-0.4,0)--++(0,0.4)--++(0.4,0)--++(0,-0.4);
\coordinate (a) at ( $ (pic cs:D0a)  - (0.1,0.1) $);
\coordinate (b) at ( $ (pic cs:D0b) + (0.2,0.38) $);
\draw (a) rectangle (b);
\draw[->] (y) to [bend right = 60]  (a);

\tikzstyle{every path}=[draw=blue,line width = 1pt]
\coordinate (x) at ( $ (pic cs:O1)  + (0.12,0.3)$);
\draw (x.south) --++(0,-0.4) --++(-0.4,0)--++(0,0.4)--++(0.4,0)--++(0,-0.4);
\coordinate (a) at ( $ (pic cs:D1a)  - (0.1,0.1) $);
\coordinate (b) at ( $ (pic cs:D1b) + (0.2,0.38) $);
\draw (a) rectangle (b);
\draw[->] (x) to [bend left = 10]  (a);

\tikzstyle{every path}=[draw=purple,line width = 1pt]
\coordinate (x) at ( $ (pic cs:O2)  + (0.12,0.3)$);
\draw (x.south) --++(0,-0.4) --++(-0.4,0)--++(0,0.4)--++(0.4,0)--++(0,-0.4);
\coordinate (a) at ( $ (pic cs:D2a)  - (0.1,0.1) $);
\coordinate (b) at ( $ (pic cs:D2b) + (0.2,0.38) $);
\draw (a) rectangle (b);
\draw[->] (x) to [bend left = 10]  (a);
\end{tikzpicture}
}
	\caption{A prefix of the fixed point of}
    \hspace{2.5cm}$\varphi\colon 
    0\mapsto 
    \begin{bmatrix}
    1 & 1 \\
    1 & 1  
    \end{bmatrix}, \quad 
    1\mapsto
	\begin{bmatrix}
    0 & 0 \\
    1 & 0 
    \end{bmatrix}$.
    \label{fig:SURD_isnot_SSURDO}
\end{figure}
It follows that $w$ is not SSURDO.
\end{example}

\section{Non-morphic bidimensional SURD words}
\label{sec:construction}

In this section we provide a construction of non-morphic bidimensional  SURD words. To construct such a word $w\colon \N^2\to A$ (where $A$ is any alphabet of size at least $2$), we proceed recursively. The construction is illustrated in Figure~\ref{fig:non-morphic}.
\begin{figure}[htb]
\[
\begin{matrix}
\vdots & \vdots & \vdots & \vdots & \vdots & \vdots & \vdots \\
a 		& \cdot & a 	& \cdot & a 	& \cdot  	& a 	& \cdots \\
b 		& c 	& \cdot & \cdot & b 	& c 		& \cdot & \cdots \\
a 		& d 	& a 	& \cdot & a 	& d 		& a 	& \cdots \\
\cdot 	& \cdot	& \cdot	& \cdot	& \cdot	& \cdot  	& \cdot & \cdots \\
a 		& \cdot & a 	& \cdot & a 	& \cdot  	& a 	& \cdots \\
b 		& c 	& \cdot & \cdot & b 	& c 		& \cdot & \cdots \\
a 		& d 	& a 	& \cdot & a 	& d 		& a 	& \cdots
\end{matrix}
\]
\caption{Construction of a non-morphic SURD bidimensional word.}
\label{fig:non-morphic}
\end{figure}

\noindent\textbf{Step 0}. Pick some $a\in A$ and for each $(i,j)\in\N^2$, put $w(2i,2j)=a$. 
\medskip

\noindent\textbf{Step 1}.
Fill anything you want in positions (0,1), (1,0) and (1,1). For each $(i,j)\in\N^2$, put $w(4i,4j+1)=w(0,1)$, $w(4i+1,4j)=w(1,0)$, $w(4i+1,4j+1)=w(1,1)$. Note that the filled positions are doubly periodic with period 4.
\medskip

\noindent\textbf{Step $n\ge 2$}. At step $n$, we have filled all the positions $(i,j)$ for $i,j<2^{n}$, and the positions with filled values are doubly periodic with period $2^{n+1}$. Let $S$ be the set of pairs $(k,\ell)$ with $k,\ell<2^{n+1}$ which have not been yet filled in. Fill anything you want in the positions from $S$. Now for each $(k,\ell)$ and each $(k',\ell')\in S$, define $w(2^{n+2}k+k', 2^{n+2}\ell+\ell'
)=w(k',\ell')$.  Note that the filled positions are doubly periodic with period $2^{n+2}$.

\begin{proposition}
\label{proposition:toeplitz}
The bidimensional infinite word $w$ defined by the construction above is SURD. More precisely, for all $\mathbf{s}\in\N^2$, the prefix of size $\mathbf{s}$ of $w$ occurs infinitely often along any direction with gaps at most $2^{\lceil \log_2(\max\mathbf{s})\rceil}$.
\end{proposition}

\begin{proof}
Let $p$ be the prefix of $w$ of  size $\mathbf{s}$ and let $\mathbf{q}$ be a direction. We show that the square prefix $p'$ of size $(2^k, 2^k)$ with $k=\lceil\log_2(\max\mathbf{s})\rceil$ appears within any consecutive $2^{k+1}$ positions along $\mathbf{q}$, hence this is also true for $p$ itself. By construction, at step $k$ we have filled all the positions $\mathbf{i}$ for $\mathbf{i}<(2^k,2^k)$, and the positions with filled values are doubly periodic with periods $(2^{k+1},0)$ and $(0,2^{k+1})$. Therefore the factor of size $(2^k,2^k)$ occurring at position $2^{k+1}\mathbf{q}$ in $w$ is equal to $p'$. The claim follows.
\end{proof} 

Observe that the morphic words satisfying Corollary~\ref{cor:1} for $s=2$ can be obtained by this construction. This construction can be generalized for any $s\in\N$ instead of $2$. Moreover, on each step we can choose as a period any multiple of a previous period.

\begin{proposition}
\label{proposition:toeplitz non morphic}
Among the bidimensional infinite words obtained by the construction above, there are uncountably many words which are not morphic.
\end{proposition}

\begin{proof} 
The construction provides uncountably many bidimensional infinite words. However, there exist only countably many morphic words.
\end{proof}

\section{Perspectives}
\label{sec:perspectives}
 
There remain many open questions related to the new notions of directional recurrence introduced in this paper. For example, we would like to generalize the characterization given by Theorem~\ref{thm:characterization} to any morphism size.
 
\begin{question}
Find a characterization of strong uniform recurrence along all directions for bidimensional square binary morphisms of size bigger than 2.
\end{question}

Another question is the missing relation between different notions of recurrence indicated in Figure~\ref{fig:links_recurrence}.

\begin{question} Prove or disprove: Strong uniform recurrence along all directions implies uniform recurrence. \end{question}

The original motivation to introduce new notions of recurrence comes from the study of return words. In the unidimensional case, a \emph{return word to $u$} in an infinite word $w$ is a factor starting at an occurrence of $u$ in $w$ and ending right before the next occurrence of $u$ in $w$. For instance, the set of return words to $u=011$ in the Thue-Morse word  is equal to $\{011010, 011001, 01101001, 0110\}$. When the infinite word $w$ is uniformly recurrent, there are finitely many return words. By coding each return word to $u$ by its order of occurrence in $w$, one obtains the \emph{derivative of $w$ with respect to the prefix $u$}. Pursuing our example, the derivative of the Thue-Morse word with respect to $011$ begins with $12341243123431241234124$. Using these derivatives, Durand obtained in 1998 the following characterization of primitive pure morphic words, i.e.\ fixed points of morphisms having a primitive incidence matrix. 

\begin{theorem}[Durand~\cite{Durand--1998}]
A word is primitive pure morphic if and only if the number of its derivatives is finite.
\end{theorem}

In dimension higher than one, it is not clear how to generalize the notion of primitivity of a morphism in order to study the uniform recurrence along directions (see~Remark~\ref{rem:primitivity}). A generalization of Durand's result to a bidimensional setting was investigated by Priebe~\cite{Priebe--2000}. In that generalization, words are replaced by tilings, the primitive substitutive property by self-similarity and the notion of derived tilings involves Voronoï cells. Recall that a Voronoï tessallation is a partition of the plane into regions, called Voronoï cells, based on the distance to a set of given points, called \emph{seeds} \cite{Voronoi--1907}. The Voronoï cell of a seed consists of all the points in the plane that are closer to it than to any other seed. Priebe aimed towards a characterization of self-similar tilings in terms of derived Voronoï tessellations and proved the following result.

\begin{theorem}[Priebe \cite{Priebe--2000}]
Let $\mathcal{T}$ be a tiling of the plane.
	\begin{itemize}
    \item If $\mathcal{T}$ is self-similar, then the number of its different derived Voronoï tilings is finite (up to similarity).
    \item If the number of its different derived Voronoï tilings is finite (up to similarity), then $\mathcal{T}$ is pseudo-self-similar.
	\end{itemize}
\end{theorem}

The bidimensional words we are considering in this paper are a particular case of tilings (see for instance Figure~\ref{fig:voronoi}, which has been reproduced from \cite{Priebe--2000}) where the letters correspond to colored unit squares (1 for black and 0 for white).
\begin{figure}[htbp]
\centering
\begin{tabular}{ccc}
\includegraphics[scale=0.215]{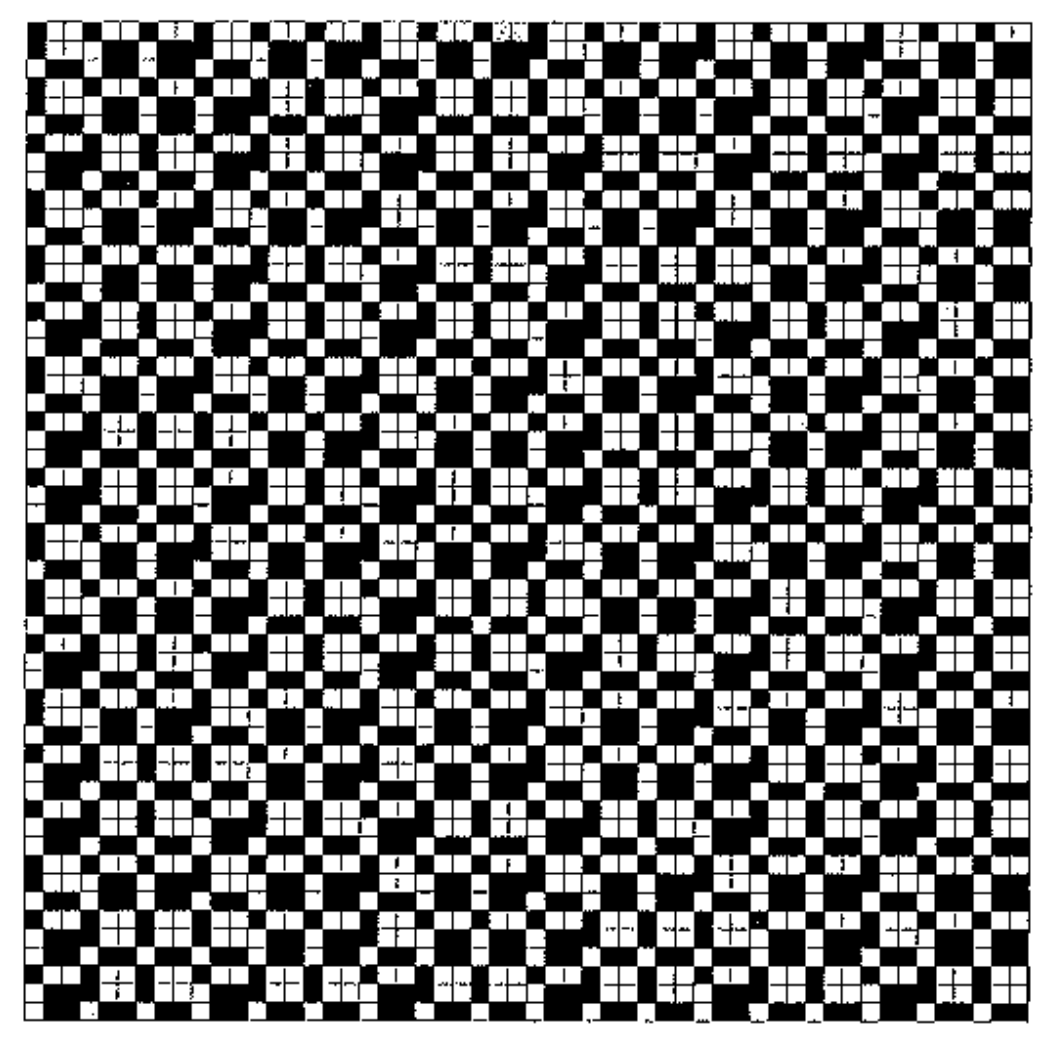}&
\includegraphics[scale=0.2]{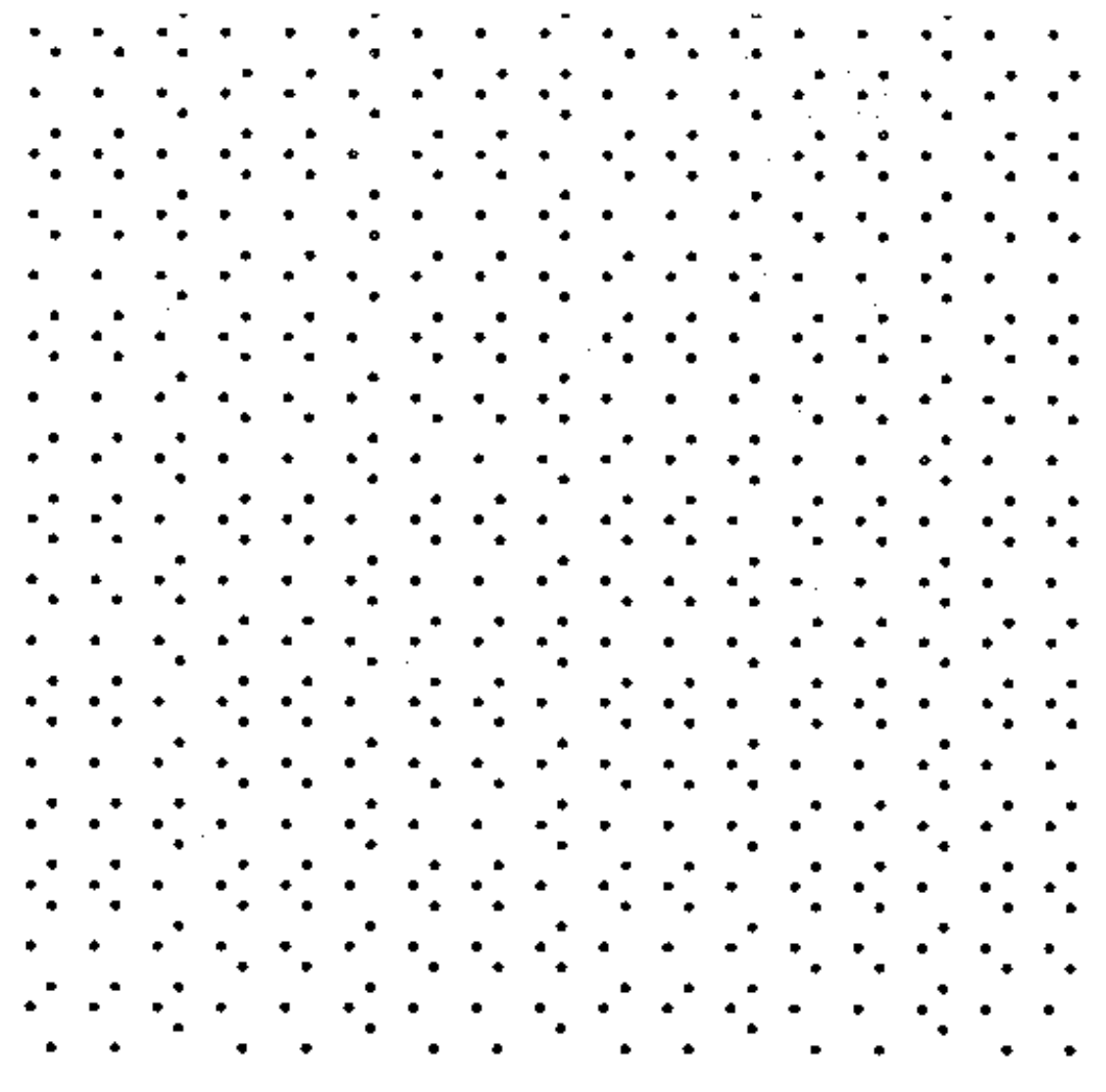}&
\includegraphics[scale=0.2]{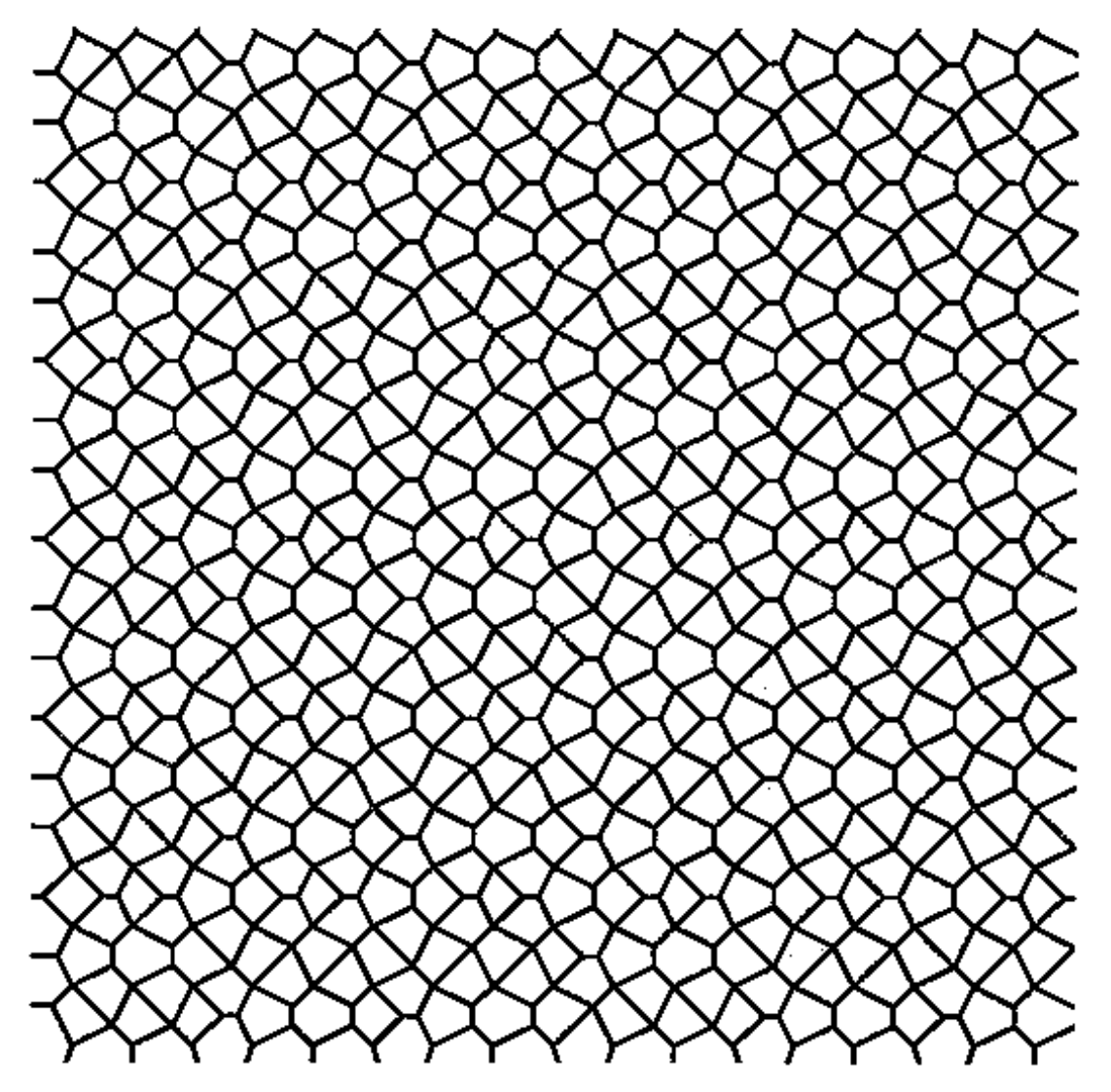}\\
(a) & (b) & (c)
\end{tabular}
\caption{A tiling (a), the set of positions where the factor $u$ occurs (b) with}
 $u\ =\ $\begin{tikzpicture}[baseline]
	\tikzstyle{1}=[shape=rectangle,fill=black,draw=black,,minimum size=3,inner sep=4pt]
 	\tikzstyle{0}=[shape=rectangle,fill=white,draw=black,minimum size=3,inner sep=4pt]
 	\node[0] at (1.25,-0.03) {};
 	\node[1] at (1.25,0.25) {};
 	\node[0] at (1.54,0.25) {};
 	\node[1] at (1.54,-0.03) {};
\end{tikzpicture},
 and the associated Voronoï tessellation (c).
\label{fig:voronoi}
\end{figure}

The main drawback of this notion of derived tilings is that, starting from a bidimensional word, we do not obtain another bidimensional word in general (as illustrated in Figure~\ref{fig:voronoi}). Hence the following questions are natural. 

\begin{question}
\label{qu:der1}
Find a differential operator for $d$-dimensional words with respect to its prefixes, that is, an operator 
\[
	D\colon (A^{\N^d},\N^d)\to B^{\N^d},
    (w,\mathbf{s})\mapsto D_\mathbf{s}(w)
\]
where $A$ and $B$ are potentially distinct alphabets and $D_\mathbf{s}(w)$ designates the \emph{derivative of $w$ with respect to its prefix of size $\mathbf{s}$}, such that the finiteness of the set
\[
	\{D_\mathbf{s}(w)\colon s\in\N^d\}
\]
would provide us with some nice property of the $d$-dimensional infinite word $w$ (such that being primitive substitutive if one thinks of Durand's theorem).
\end{question}

Here is a variant of the previous question. 

\begin{question}
\label{qu:der2}
Find a differential operator for $d$-dimensional words with respect to its prefixes, that is, an operator 
\[
	D\colon (A^{\N^d},\N^d)\to B^{\N^d},
    (w,\mathbf{s})\mapsto D_\mathbf{s}(w)
\]
where $A$ and $B$ are potentially distinct alphabets and $D_\mathbf{s}(w)$ designates the \emph{derivative of $w$ with respect to its prefix of size $\mathbf{s}$}, such that for all $w\in A^{\N^d}$ and all $\mathbf{s},\mathbf{t}\in\N^d$ we have 
\[
	D_\mathbf{s}(D_\mathbf{t}(w))=D_\mathbf{u}(w)
\]
for some well-chosen size $\mathbf{u}\in\N^d$.
\end{question}

The notion of SURD words introduced in the present paper provides us with a way of deriving $d$-dimensional words, which generalizes the unidimensional derivatives. The idea is as follows. Let $w\colon\N^d\mapsto A$ be a SURD $d$-dimensional word and let $\mathbf{s}\in\N^d$. 
Being SURD implies that there exists an integer $r\ge 1$ such that for all directions $\mathbf{q}$, there are at most $r$ distinct return words to the letter $w_{\mathbf{q},\mathbf{s}}(0)$ in the unidimensional word $w_{\mathbf{q},\mathbf{s}}$. We define the {\em derivative of $w$ with respect to the prefix of size $\mathbf{s}$} to be the $d$-dimensional word $D_\mathbf{s}(w)\colon\N^d\mapsto [\![0,r-1]\!]$ such that for all direction $\mathbf{q}$, $((D_\mathbf{s}w)_{\ell\mathbf{q}})_{\ell\in\N}$ is the unidimensional derivative of $w_{\mathbf{q},\mathbf{s}}$ with respect to its first letter obtained by coding the return words to $w_{\mathbf{q},\mathbf{s}}(0)$ in order of appearance from $0$ to $r-1$. 

For example, if $\varphi^\omega(1)$ is the fixed point of the morphism given in Example~\ref{ex:SURD_isnot_SSURDO} and depicted in Figure~\ref{fig:SURD_isnot_SSURDO}, then its derivative $D_{(1,2)}(\varphi^\omega(1))$ with respect to the prefix of size $(1,2)$ is depicted in Figure~\ref{fig:der1}. 
\begin{figure}[htb]
\centering
\scalebox{0.8}{
$
\begin{array}{c|ccccccccccccccccccccccccccc}
7 & 1 & 0 & 1 & 1 & 0 & 1 & 1 & 0 & 1 & 0 & 1 & 1 & 1 & 0 & 0 & 1 & 1 & 0 & 1 & 1 & 0 & 1 & 1 & 1 & 1 & 1 & 1 \\
6 & 1 & 1 & 1 & 0 & 0 & 1 & 1 & 1 & 1 & 1 & 1 & 0 & 1 & 1 & 2 & 1 & 2 & 1 & 1 & 1 & 2 & 3 & 2 & 1 & 1 & 1 & 1 \\
5 & 0 & 1 & 1 & 1 & 1 & 0 & 1 & 1 & 1 & 1 & 1 & 1 & 1 & 1 & 0 & 0 & 1 & 1 & 1 & 1 & 2 & 1 & 1 & 1 & 1 & 4 & 1 \\
4 & 0 & 1 & 2 & 0 & 4 & 1 & 0 & 0 & 1 & 1 & 1 & 1 & 2 & 1 & 2 & 1 & 0 & 1 & 2 & 0 & 3 & 1 & 2 & 0 & 2 & 1 & 2 \\
3 & 0 & 0 & 0 & 3 & 0 & 1 & 0 & 1 & 1 & 0 & 1 & 1 & 0 & 0 & 0 & 3 & 1 & 0 & 3 & 0 & 1 & 2 & 1 & 1 & 0 & 0 & 0 \\
2 & 1 & 1 & 2 & 0 & 2 & 1 & 1 & 1 & 1 & 1 & 2 & 1 & 2 & 1 & 1 & 0 & 2 & 1 & 2 & 0 & 2 & 1 & 1 & 1 & 1 & 0 & 2 \\
1 & 1 & 1 & 1 & 1 & 1 & 1 & 1 & 0 & 1 & 1 & 1 & 1 & 0 & 1 & 0 & 1 & 1 & 1 & 1 & 1 & 1 & 1 & 1 & 0 & 1 & 1 & 0 \\
0 & 0 & 1 & 1 & 0 & 0 & 0 & 1 & 1 & 0 & 1 & 1 & 0 & 1 & 1 & 0 & 0 & 0 & 1 & 1 & 0 & 0 & 0 & 1 & 1 & 0 & 0 & 0 \\
\hline
 & 0 & 1 & 2 & 3 & 4 & 5 & 6 & 7 & 8 & 9 & 10 & 11 & 12 & 13 & 14 & 15 & 16 & 17 & 18 & 19 & 20 & 21 & 22 & 23 & 24 & 25 & 26 
\end{array}
$
}
     \caption{The derivative of the fixed point $\varphi^\omega(1)$ of Figure~\ref{fig:SURD_isnot_SSURDO} w.r.t.\  the prefix of size $(1,2)$ when coding the return words in order of appearance along every direction.}
    \label{fig:der1}
\end{figure}
We know from Corollary~\ref{cor:1} that return words to the prefix of size $(1,2)$ in $(\varphi^\omega(1))_{\mathbf{q},(1,2)}$ have length at most $4$ for any direction $\mathbf{q}$. Therefore, there could be at most $4^3=64$ such return words. For instance, the letters on the diagonal correspond to the unidimensional derivative of $\varphi^\omega(1)_{(1,1),(1,2)}$ with respect to its first letter $\left[\begin{smallmatrix}
    0  \\
    1  \\
    \end{smallmatrix}\right]$:
\[
    \underbrace{
    \left[\begin{smallmatrix}
    0  \\
    1  \\
    \end{smallmatrix}\right]
    \left[\begin{smallmatrix}
    1  \\
    0  \\
    \end{smallmatrix}\right]
    \left[\begin{smallmatrix}
    1  \\
    1  \\
    \end{smallmatrix}\right]}_{0}
    \underbrace{
    \left[\begin{smallmatrix}
    0  \\
    1  \\
    \end{smallmatrix}\right]}_{1}
    \underbrace{
    \left[\begin{smallmatrix}
    0  \\
    1  \\
    \end{smallmatrix}\right]
    \left[\begin{smallmatrix}
    0  \\
    0  \\
    \end{smallmatrix}\right]}_{2}
    \underbrace{
    \left[\begin{smallmatrix}
    0  \\
    1  \\
    \end{smallmatrix}\right]
     \left[\begin{smallmatrix}
    1  \\
    0  \\
    \end{smallmatrix}\right]}_{3}
     \underbrace{
    \left[\begin{smallmatrix}
    0  \\
    1  \\
    \end{smallmatrix}\right]
    \left[\begin{smallmatrix}
    1  \\
    0  \\
    \end{smallmatrix}\right]
    \left[\begin{smallmatrix}
    1  \\
    1  \\
    \end{smallmatrix}\right]
    \left[\begin{smallmatrix}
    1  \\
    1  \\
    \end{smallmatrix}\right]}_{4}
     \underbrace{
     \left[\begin{smallmatrix}
    0  \\
    1  \\
    \end{smallmatrix}\right]
    \left[\begin{smallmatrix}
    1  \\
    0  \\
    \end{smallmatrix}\right]
    \left[\begin{smallmatrix}
    1  \\
    1  \\
    \end{smallmatrix}\right]}_{0}
     \underbrace{
     \left[\begin{smallmatrix}
    0  \\
    1  \\
    \end{smallmatrix}\right]}_{1}
     \underbrace{
     \left[\begin{smallmatrix}
    0  \\
    1  \\
    \end{smallmatrix}\right]
     \left[\begin{smallmatrix}
    1  \\
    0  \\
    \end{smallmatrix}\right]
    \left[\begin{smallmatrix}
    1  \\
    1  \\
    \end{smallmatrix}\right]}_{0}
     \underbrace{
     \left[\begin{smallmatrix}
    0  \\
    1  \\
    \end{smallmatrix}\right]}_{1}
     \underbrace{
    \left[\begin{smallmatrix}
    0  \\
    1  \\
    \end{smallmatrix}\right]
    \left[\begin{smallmatrix}
    0  \\
    0  \\
    \end{smallmatrix}\right]}_{2}
    \cdots
\]

In the previous definition, we chose to code the return words with respect to their order of appearance in $w_{\mathbf{q},\mathbf{s}}$ for each $\mathbf{q}$. This means that two occurrences of the same letter $i\in[\![0,r-1]\!]$, one at a position $\ell\mathbf{q}$ and the other at a position $\ell'\mathbf{q}'$ for different directions $\mathbf{q}$ and $\mathbf{q}'$, might represent different return words. For example, in Figure~\ref{fig:der1}, the letter $1$ at position $(0,1)$ corresponds to the return word $\left[\begin{smallmatrix}
    1  \\
    0  \\
    \end{smallmatrix}\right]
    \left[\begin{smallmatrix}
    0  \\
    1  \\
    \end{smallmatrix}\right]$ but the letter $1$ at position $(0,1)$ corresponds to the return word $\left[\begin{smallmatrix}
    1  \\
    0  \\
    \end{smallmatrix}\right]
    \left[\begin{smallmatrix}
    0  \\
    0  \\
    \end{smallmatrix}\right]$. An alternative definition of derivatives would be to code the return words uniformly, i.e.\ independently of the considered direction $\mathbf{q}$. The derivative of $\varphi^\omega(1)$ with respect to the prefix of size $(1,2)$ obtained by following the latter convention is depicted in Figure~\ref{fig:der2}. The codes of the used return words are given in Table~\ref{tab:codes}. Note that the letter at position $(0,0)$ in the derivative is not well defined since in general, the first return words to the prefix of size $\mathbf{s}$ along two different directions $\mathbf{q}$ and $\mathbf{q}'$ are not the same.
\begin{figure}[htb]
\centering
\scalebox{0.8}{
$
\begin{array}{c|ccccccccccccccccccccccccccc}
 7 & 1 & 9 & 7 & 10 & 11 & 3 & 7 & 2 & 1 & 3 & 1 & 4 & 1 & 3 & 6 & 1 & 12 & 9 & 12 & 9 & 11 & 3 & 7 & 3 & 12 & 5 & 12 \\
 6 & 1 & 4 & 6 & 6 & 6 & 5 & 3 & 7 & 12 & 4 & 4 & 6 & 1 & 4 & 6 & 6 & 12 & 4 & 3 & 4 & 12 & 4 & 3 & 5 & 1 & 4 & 6 \\
 5 & 0 & 4 & 6 & 3 & 1 & 2 & 0 & 8 & 0 & 0 & 1 & 3 & 1 & 3 & 6 & 6 & 1 & 4 & 6 & 3 & 11 & 1 & 1 & 3 & 1 & 9 & 6 \\
 4 & 0 & 4 & 7 & 9 & 0 & 4 & 6 & 9 & 1 & 5 & 6 & 4 & 4 & 4 & 6 & 4 & 0 & 4 & 7 & 9 & 0 & 4 & 7 & 9 & 6 & 4 & 10 \\
 3 & 0 & 3 & 6 & 1 & 11 & 4 & 6 & 10 & 1 & 6 & 1 & 1 & 0 & 3 & 6 & 0 & 1 & 3 & 7 & 3 & 12 & 3 & 6 & 16 & 0 & 3 & 6 \\
 2 & 1 & 4 & 4 & 6 & 7 & 6 & 3 & 7 & 1 & 4 & 1 & 4 & 6 & 4 & 8 & 6 & 6 & 4 & 1 & 6 & 11 & 6 & 3 & 7 & 0 & 6 & 5 \\
 1 & 1 & 3 & 1 & 3 & 1 & 4 & 1 & 9 & 1 & 3 & 1 & 3 & 11 & 13 & 6 & 3 & 1 & 3 & 1 & 3 & 1 & 1 & 6 & 9 & 1 & 4 & 6 \\
 0 & ? & 4 & 4 & 5 & 5 & 5 & 4 & 4 & 5 & 4 & 4 & 5 & 4 & 4 & 5 & 5 & 5 & 4 & 4 & 5 & 5 & 5 & 4 & 4 & 5 & 5 & 5 \\
 \hline
 & 0 & 1 & 2 & 3 & 4 & 5 & 6 & 7 & 8 & 9 & 10 & 11 & 12 & 13 & 14 & 15 & 16 & 17 & 18 & 19 & 20 & 21 & 22 & 23 & 24 & 25 & 26 \\
\end{array}
$
}
     \caption{The derivative of the fixed point $\varphi^\omega(1)$ of Figure~\ref{fig:SURD_isnot_SSURDO} w.r.t.\ the prefix of size $(1,2)$ when coding the return words uniformly.}
    \label{fig:der2}
\end{figure}

\begin{table}
    \centering
\begin{tabular}{c|c}
    return word & code \\
    \hline 
    $\left[\begin{smallmatrix}
    0  \\
    1  \\
    \end{smallmatrix}\right]
    \left[\begin{smallmatrix}
    1  \\
    0  \\
    \end{smallmatrix}\right]
    \left[\begin{smallmatrix}
    1  \\
    1  \\
    \end{smallmatrix}\right]
    \left[\begin{smallmatrix}
    1  \\
    1  \\
    \end{smallmatrix}\right]$& $0$ \\
    $\left[\begin{smallmatrix}
    0  \\
    1  \\
    \end{smallmatrix}\right]
    \left[\begin{smallmatrix}
    1  \\
    0  \\
    \end{smallmatrix}\right]$ & $1$ \\
    $\left[\begin{smallmatrix}
    0  \\
    1  \\
    \end{smallmatrix}\right]
    \left[\begin{smallmatrix}
    1  \\
    0  \\
    \end{smallmatrix}\right]
    \left[\begin{smallmatrix}
    1  \\
    1  \\
    \end{smallmatrix}\right]$ & $2$ \\
    $\left[\begin{smallmatrix}
    0  \\
    1  \\
    \end{smallmatrix}\right]$ & $3$ \\
    $\left[\begin{smallmatrix}
    0  \\
    1  \\
    \end{smallmatrix}\right]
    \left[\begin{smallmatrix}
    0  \\
    0  \\
    \end{smallmatrix}\right]$ & $4$ \\
    $\left[\begin{smallmatrix}
    0  \\
    1  \\
    \end{smallmatrix}\right]
    \left[\begin{smallmatrix}
    0  \\
    0  \\
    \end{smallmatrix}\right]
    \left[\begin{smallmatrix}
    1  \\
    1  \\
    \end{smallmatrix}\right]
    \left[\begin{smallmatrix}
    1  \\
    1  \\
    \end{smallmatrix}\right]$ & $5$ \\
\end{tabular} $\quad$
\begin{tabular}{c|c}
    return word & code \\
    \hline 
    $\left[\begin{smallmatrix}
    0  \\
    1  \\
    \end{smallmatrix}\right]
    \left[\begin{smallmatrix}
    1  \\
    1  \\
    \end{smallmatrix}\right]$ & $6$ \\
    $\left[\begin{smallmatrix}
    0  \\
    1  \\
    \end{smallmatrix}\right]
    \left[\begin{smallmatrix}
    1  \\
    1  \\
    \end{smallmatrix}\right]
    \left[\begin{smallmatrix}
    1  \\
    1  \\
    \end{smallmatrix}\right]
    \left[\begin{smallmatrix}
    1  \\
    1  \\
    \end{smallmatrix}\right]$ & $7$ \\
    $\left[\begin{smallmatrix}
    0  \\
    1  \\
    \end{smallmatrix}\right]
    \left[\begin{smallmatrix}
    0  \\
    0  \\
    \end{smallmatrix}\right]
    \left[\begin{smallmatrix}
    1  \\
    1  \\
    \end{smallmatrix}\right]
    \left[\begin{smallmatrix}
    0  \\
    0  \\
    \end{smallmatrix}\right]$ & $8$ \\
    $\left[\begin{smallmatrix}
    0  \\
    1  \\
    \end{smallmatrix}\right]
    \left[\begin{smallmatrix}
    1  \\
    1  \\
    \end{smallmatrix}\right]
    \left[\begin{smallmatrix}
    1  \\
    1  \\
    \end{smallmatrix}\right]
    \left[\begin{smallmatrix}
    0  \\
    0  \\
    \end{smallmatrix}\right]$ & $9$ \\
    $\left[\begin{smallmatrix}
    0  \\
    1  \\
    \end{smallmatrix}\right]
    \left[\begin{smallmatrix}
    1  \\
    0  \\
    \end{smallmatrix}\right]
    \left[\begin{smallmatrix}
    1  \\
    1  \\
    \end{smallmatrix}\right]
    \left[\begin{smallmatrix}
    1  \\
    0  \\
    \end{smallmatrix}\right]$ & $10$ \\
    $\left[\begin{smallmatrix}
    0  \\
    1  \\
    \end{smallmatrix}\right]
    \left[\begin{smallmatrix}
    1  \\
    1  \\
    \end{smallmatrix}\right]
    \left[\begin{smallmatrix}
    1  \\
    1  \\
    \end{smallmatrix}\right]
    \left[\begin{smallmatrix}
    1  \\
    0  \\
    \end{smallmatrix}\right]$ & $11$ \\
\end{tabular} $\quad$ 
\begin{tabular}{c|c}
    return word & code \\
    \hline 
    $\left[\begin{smallmatrix}
    0  \\
    1  \\
    \end{smallmatrix}\right]
    \left[\begin{smallmatrix}
    0  \\
    0  \\
    \end{smallmatrix}\right]
    \left[\begin{smallmatrix}
    1  \\
    1  \\
    \end{smallmatrix}\right]$ & $12$ \\
    $\left[\begin{smallmatrix}
    0  \\
    1  \\
    \end{smallmatrix}\right]
    \left[\begin{smallmatrix}
    1  \\
    1  \\
    \end{smallmatrix}\right]
    \left[\begin{smallmatrix}
    0  \\
    0  \\
    \end{smallmatrix}\right]$ & $13$ \\
    $\left[\begin{smallmatrix}
    0  \\
    1  \\
    \end{smallmatrix}\right]
    \left[\begin{smallmatrix}
    1  \\
    0  \\
    \end{smallmatrix}\right]
    \left[\begin{smallmatrix}
    1  \\
    1  \\
    \end{smallmatrix}\right]
    \left[\begin{smallmatrix}
    0  \\
    0  \\
    \end{smallmatrix}\right]$ & $14$ \\
    $\left[\begin{smallmatrix}
    0  \\
    1  \\
    \end{smallmatrix}\right]
    \left[\begin{smallmatrix}
    1  \\
    1  \\
    \end{smallmatrix}\right]
    \left[\begin{smallmatrix}
    1  \\
    1  \\
    \end{smallmatrix}\right]$& $15$ \\
    $\left[\begin{smallmatrix}
    0  \\
    1  \\
    \end{smallmatrix}\right]
    \left[\begin{smallmatrix}
    0  \\
    0  \\
    \end{smallmatrix}\right]
    \left[\begin{smallmatrix}
    1  \\
    1  \\
    \end{smallmatrix}\right]
    \left[\begin{smallmatrix}
    1  \\
    0  \\
    \end{smallmatrix}\right]$ & $16$ \\ 
    \\
\end{tabular}
   \caption{Codes of return words to the prefix}
   \hspace{0.3cm} $\left[\begin{smallmatrix}
    0  \\
    1  \\
    \end{smallmatrix}\right]$ 
    occurring in Figure~\ref{fig:der2}.
    \label{tab:codes}
\end{table}

We do not know whether one of this definition is a good candidate for answering Question~\ref{qu:der1}. Note that the second definition does not allow us to derive twice (because of the unknown letter at position $\mathbf{0}$), and hence cannot be a good candidate for answering Question~\ref{qu:der2}. In particular, in order to be able to derive twice, the SURD property must be preserved under differentiation.

\begin{question}
Does our first definition of multidimensional derivatives give rise to SURD words when starting from a SURD word?
\end{question}

Another aspect we did not treat in the paper is the symbolic dynamical one. It is well known that in the unidimensional case a word is uniformly recurrent if and only if the corresponding dynamical system is minimal (see e.g.\ \cite{Ferenczi--Monteil--2010}).

\begin{question}
What kind of dynamical properties are reflected by the modifications of the notion of uniform recurrence introduced in the paper?
\end{question}

\section{Acknowledgements}
We are grateful to Mathieu Sablik for inspiring discussions. 
The second author  is partially supported by Russian Foundation of Basic Research (grant 20-01-00488) and by the  Foundation  for  the  Advancement  of  Theoretical  Physics  and  Mathematics “BASIS”. The last author acknowledges partial funding via a Welcome Grant of the Universit\'e de Li\`ege.


\begin{thebibliography}{X}

\bibitem{Avgustinovich--2011}
Avgustinovich, S. V., Fon-Der-Flaass, D. G.,  Frid, A. E.:
Arithmetical Complexity of Infinite Words.
Words, Languages \& Combinatorics III,
51--62 (2003)

\bibitem{Berthe--Vuillon--2000}
Berth\'e, V., Vuillon, L: Tilings and rotations on the torus: a two-dimensional generalization of
Sturmian sequences. 
Discrete Math. \textbf{223}, 
27--53 (2000)

\bibitem{Berthe--Vuillon--2001}
Berth\'e, V.,  Vuillon, L.:
Palindromes and two-dimensional {S}turmian sequences.
J. Automata, Languages and Combinatorics \textbf{6}(2), 
121--138 (2001)


\bibitem{Bucci--2013}
Bucci, M., Hindman, N., Puzynina, S., Zamboni, L. Q.:
On additive properties of sets defined by the Thue-Morse word. 
J. Comb. Theory, Ser. A \textbf{120}(6),
1235--1245 (2013)

\bibitem{Cassaigne--1999} 
Cassaigne, J.:
Double Sequences with Complexity $mn+1$. 
Journal of Automata, Languages and Combinatorics \textbf{4}(3),
153--170 (1999)

\bibitem{Charlier--Karki--Rigo--2010}
Charlier, \'E., K\"{a}rki, T., Rigo, M.:
Multidimensional generalized automatic sequences and shape-symmetric morphic words.
Discrete Math. \textbf{310}(6-7),
1238--1252 (2010)

\bibitem{Cyr--Kra--2015} 
Cyr, V. and Kra, B.:
Nonexpansive {$\Z^2$}-subdynamics and {N}ivat's conjecture.
Trans. Amer. Math. Soc. \textbf{367(9)}, 
6487--6537 (2015)

\bibitem{Didier--1998}
Didier, G.:
Combinatoire des codages de rotations.
Acta Arith. \textbf{85}(2),
157--177 (1998)
    
\bibitem{Durand--1998}
Durand, F.:
A characterization of substitutive sequences using return words.
Discrete Math. \textbf{179},
89--101 (1998)

\bibitem{Durand--Rigo--2013}
Durand, F., Rigo, M.:
Multidimensional extension of the Morse-Hedlund theorem. 
European J. Combin. \textbf{34}(2),
391--409 (2013)

\bibitem{Ferenczi--Monteil--2010} Ferenczi, S., Monteil, T.: 
Infinite words with uniform frequencies, and invariant measures. 
In: Combinatorics, Automata and Number Theory, edited by 
V. Berthé, M. Rigo, Cambridge University Press, 2010

\bibitem{PriebeFrank--2008} Priebe Frank, N.: A primer of substitution tilings of the Euclidean plane. Expo. Math. \textbf{26}(4), 295--326 (2008).

\bibitem{Furstenberg--1981}
Furstenberg, H.:
Poincar\'e recurrence and number theory. 
Bull. Amer. Math. Soc. (N.S.)
\textbf{5}(3), 211--234 (1981)

\bibitem{Hardy-Wright--2008}
Hardy, G. H., Wright, E. M.:
An introduction to the theory of numbers.
Sixth edition.
Oxford University Press, Oxford, 2008

\bibitem{Kari--Szabados--2015}
Kari, J. and Szabados, M.:
An algebraic geometric approach to {N}ivat's conjecture.
Automata, languages, and programming. {P}art {II},
Springer, Heidelberg,
Lecture Notes in Comput. Sci. \textbf{9135},
273--285 (2015)

\bibitem{Labbe--2018}
Labb{\'e}, S.:
A self-similar aperiodic set of 19 Wang tiles.
Geometriae Dedicata \textbf{201}(1), 
81--109 (2019)

\bibitem{Muchnik--2003}
Muchnik, An.:
The definable criterion for definability in {P}resburger arithmetic and its applications.
Theoret. Comput. Sci. \textbf{290}(3),
1433--1444 (2003)

\bibitem{Nivat} 
Nivat, M.: Invited talk at ICALP'97

\bibitem{Priebe--2000}
Priebe, N.:
Towards a Characterization of Self-Similar Tilings in Terms of Derived Voronoï Tessellations.
Geometriae Dedicata \textbf{79},
239--265 (2000)


\bibitem{Rigo--Maes--2002} 
Rigo, M., Maes, A.:
More on generalized automatic sequences.
J. Autom. Lang. Comb. \textbf{7}(3),
351--376 (2002)
     
\bibitem{Salimov--2010} 
Salimov, P.:
On uniform recurrence of a direct product.
Discrete Math. Theor. Comput. Sci. \textbf{12}(4),
1--8 (2010)

\bibitem{Salon--1987} 
Salon, O.:
Suites automatiques \`a multi-indices et alg\'ebricit\'e.
C. R. Acad. Sci. Paris S\'{e}r. I Math. \textbf{305}(12),
501--504 (1987)

\bibitem{Semenov--1977} 
Semenov, A.:
The {P}resburger nature of predicates that are regular in two number systems.
Siberian Math. J. \textbf{18},
289--299 (1977)


\bibitem{Slater--1967}  
Slater, N. B.: 
Gaps and steps for the sequence $n\theta\ {\rm mod}\ 1$.
Proc. Cambridge, Philos. Soc. \textbf{63},
1115–1123 (1967)

\bibitem{Stewart--1995}
Stewart, I.:
Four encounters with {S}ierpi\'{n}ski's gasket.
Math. Intelligencer \textbf{17}(1), 
52--64 (1995)
    
\bibitem{Swierczkowski--1959}
\'{S}wierczkowski, S.:
On successive settings of an arc on the circumference of a circle.
Fund. Math. \textbf{46}, 
187--189 (1959)

\bibitem{Thue--1906}
Thue, A.: 
Uber unendliche Zeichenreihen,  Norske vid. Selsk. Skr. Mat. Nat. Kl. {\bf 7}, 1–22 (1906); 
reprinted in “Selected mathematical papers of Axel Thue,” T. Nagell, ed., Universitetsforlaget, Oslo, 139–158 (1977)

\bibitem{Voronoi--1907}
Voronoï, G.:
Nouvelles applications des paramètres continus à la théorie des formes quadratiques.
J. reine angew. Math. \textbf{133},
97--178 (1907) 

\bibitem{Vuillon--1998}
Vuillon, L.: Combinatoire des motifs d'une suite sturmienne bidimensionnelle. 
Theoret. Comput. Sci. \textbf{209},
261--285 (1998)
\end{thebibliography}
\end{document}